\documentclass[a4paper,11pt]{article}
\usepackage{amsmath,amsthm,amssymb,enumitem,xcolor}

\usepackage[hang]{footmisc}
\usepackage[nosort,nocompress,noadjust]{cite}

\usepackage[bookmarks=false,hyperfootnotes=false,colorlinks,
    linkcolor={red!60!black},
    citecolor={blue!50!black},
    urlcolor={blue!80!black}]{hyperref}

\renewcommand{\eqref}[1]{\hyperref[#1]{(\ref{#1})}}

\newlist{enumlist}{enumerate}{2}
\setlist[enumlist,1]{labelindent=0cm,label=\arabic*.,ref=\arabic*,labelwidth=2.5ex,labelsep=0.5ex,leftmargin=3ex,align=left,topsep=0.5ex,itemsep=1ex,parsep=1ex}
\setlist[enumlist,2]{labelindent=0cm,label=\alph*),ref=\arabic*,labelwidth=2.5ex,labelsep=0.5ex,leftmargin=3ex,align=left,topsep=0.5ex,itemsep=1ex,parsep=1ex}

\newlist{itemlist}{itemize}{1}
\setlist[itemlist]{labelindent=0cm,label=$\bullet$,labelwidth=2.5ex,labelsep=0.5ex,leftmargin=3ex,align=left,topsep=0.5ex,itemsep=1ex,parsep=1ex}

\numberwithin{equation}{section}

{\theoremstyle{definition}\newtheorem{definition}{Definition}[section]
\newtheorem*{definition*}{Definition}

\newtheorem{remark}[definition]{Remark}
\newtheorem{example}[definition]{Example}
\newtheorem*{example*}{Example}
\newtheorem*{examples*}{Examples}}

\newtheorem{proposition}[definition]{Proposition}
\newtheorem{lemma}[definition]{Lemma}
\newtheorem{theorem}[definition]{Theorem}
\newtheorem{corollary}[definition]{Corollary}

{\theoremstyle{definition}}

\newcommand{\cB}{\mathcal{B}}

\newcommand{\cD}{\mathcal{D}}
\newcommand{\cE}{\mathcal{E}}
\newcommand{\cG}{\mathcal{G}}
\newcommand{\cF}{\mathcal{F}}

\newcommand{\cM}{\mathcal{M}}

\newcommand{\cP}{\mathcal{P}}
\newcommand{\cR}{\mathcal{R}}
\newcommand{\cQ}{\mathcal{Q}}
\newcommand{\cS}{\mathcal{S}}
\newcommand{\cT}{\mathcal{T}}
\newcommand{\cU}{\mathcal{U}}

\newcommand{\cW}{\mathcal{W}}

\newcommand{\C}{\mathbb{C}}
\newcommand{\F}{\mathbb{F}}
\newcommand{\N}{\mathbb{N}}
\newcommand{\Q}{\mathbb{Q}}
\newcommand{\R}{\mathbb{R}}
\newcommand{\T}{\mathbb{T}}

\newcommand{\al}{\alpha}
\newcommand{\be}{\beta}
\newcommand{\eps}{\varepsilon}
\newcommand{\om}{\omega}

\newcommand{\vphi}{\varphi}

\newcommand{\Atil}{\widetilde{A}}

\newcommand{\gammatil}{\widetilde{\gamma}}

\newcommand{\Ghat}{\widehat{G}}

\newcommand{\Aut}{\operatorname{Aut}}
\newcommand{\ball}{\operatorname{ball}}
\newcommand{\End}{\operatorname{End}}

\newcommand{\Homeo}{\operatorname{Homeo}}
\newcommand{\Isom}{\operatorname{Isom}}

\newcommand{\Prob}{\operatorname{Prob}}
\newcommand{\res}{\operatorname{res}}
\newcommand{\lspan}{\operatorname{span}}

\newcommand{\vNalg}{\operatorname{vNalg}}

\newcommand{\ot}{\otimes}
\newcommand{\id}{\mathord{\text{\rm id}}}
\newcommand{\actson}{\curvearrowright}

\newcommand{\othat}{\mathbin{\otimes_\wedge}}
\newcommand{\otvee}{\mathbin{\otimes_\vee}}
\newcommand{\otmin}{\mathbin{\otimes_{\text{\rm min}}}}
\newcommand{\otmax}{\mathbin{\otimes_{\text{\rm max}}}}
\newcommand{\otalg}{\mathbin{\otimes_{\text{\rm alg}}}}
\newcommand{\rtimesred}{\rtimes_{\text{\rm red}}}
\newcommand{\rtimesf}{\rtimes_{\text{\rm f}}}
\newcommand{\full}{_{\text{\rm f}}}
\newcommand{\red}{_{\text{\rm red}}}
\newcommand{\nmax}{_{\text{\rm max}}}

\begin{document}
\begin{center}
{\boldmath\Large\bf Borel fields and measured fields of Polish spaces,\vspace{1ex}\\ Banach spaces, von Neumann algebras and C$^*$-algebras}

\bigskip

{\sc by Stefaan Vaes\footnote{\noindent KU~Leuven, Department of Mathematics, Celestijnenlaan 200b - box 2400, 3001 Leuven (Belgium).}\textsuperscript{,}\footnote{E-mail: stefaan.vaes@kuleuven.be. S.V.\ is supported by Methusalem grant METH/21/03 of the Flemish Government and by FWO research projects G090420N and G016325N of the Research Foundation Flanders (FWO).} and Lise Wouters\textsuperscript{1,}\footnote{L.W.\ is supported by PhD grant 11B6622N funded by the Research Foundation Flanders (FWO).}}
\end{center}

\begin{abstract}\addtocounter{footnote}{1}
\noindent Several recent articles in operator algebras make a nontrivial use of the theory of measurable fields of von Neumann algebras $(M_x)_{x \in X}$ and related structures. This includes the associated field $(\Aut M_x)_{x \in X}$ of automorphism groups and more general measurable fields of Polish groups with actions on Polish spaces. Nevertheless, a fully rigorous and at the same time sufficiently broad and flexible theory of such Borel fields and measurable fields is not available in the literature. We fill this gap in this paper and include a few counterexamples to illustrate the subtlety: for instance, for a Borel field $(M_x)_{x \in X}$ of von Neumann algebras, the field of Polish groups $(\Aut M_x)_{x \in X}$ need not be Borel.
\end{abstract}

\section{Introduction}

Since every von Neumann algebra $M$ with separable predual can be written as a direct integral of factors $(M_x)_{x \in X}$ indexed by a standard probability space $(X,\mu)$, such measurable fields of factors often appear in the operator algebra literature. They are thoroughly and rigorously treated in \cite{Tak79}. In several situations, for instance when dealing with the classification of group actions on von Neumann algebras as in \cite{ST84,PSV18,SW23,Cha23}, one needs to consider canonically associated fields of separable structures like the Polish groups $\Aut(M_x)$ of automorphisms, or the unitary groups $\cU(M_x)$, or more complicated fields of cocycle conjugacies between group actions, etc. An entirely consistent and rigorous treatment of Borel fields and measurable fields encompassing all such constructions is not available in the literature.

This lack of a rigorous treatment has never led to mistakes because in the above context, we are always free to discard from $(X,\mu)$ a Borel set of measure zero, so that at any stage the Jankov-von Neumann measurable selection theorem is available.

The theory of Borel fields is however more subtle. A rigorous theory of Borel fields of Polish groups was given in \cite{Sut85}. In such a context, one is no longer allowed to discard sets of measure zero. As a result in \cite{Sut85}, it was left as an open problem if for a Borel field $(G_x)_{x \in X}$ of second countable locally compact groups, the field of Pontryagin dual groups $\Ghat = (\Ghat_x)_{x \in X}$ was still a Borel field, the issue being to prove that there exist enough everywhere defined Borel sections $X \to \Ghat$.

With this paper, we want to fill this gap in the literature and present a fully rigorous and at the same time very flexible theory of Borel fields and measurable fields encompassing all kinds of separable structures: Polish spaces, Polish groups, separable Hilbert spaces, separable Banach spaces, separable C$^*$-algebras, von Neumann algebras with separable predual, their unitary groups and automorphism groups, etc.

Let $X$ be a standard Borel space and $(S_x)_{x \in X}$ a family of such separable structures, indexed by $X$. In this paper, a \emph{Borel field} structure only consists of a $\sigma$-algebra $\cB$ on the disjoint union $S = \sqcup_{x \in X} S_x$, with corresponding map $\pi : S \to X$. Of course, $(S,\cB)$ has to satisfy a number of properties, always including that $(S,\cB)$ is standard Borel, that $\pi : S \to X$ is measurable and that there exists a dense sequence of Borel sections $X \to S$. By restricting our definition to the only essential point, namely a $\sigma$-algebra on the total space $S$, our definition is sufficiently flexible to easily construct new Borel fields from existing ones, e.g.\ passing from a field $(M_x)_{x \in X}$ of von Neumann algebras to a field of Polish groups $(\Aut(M_x))_{x \in X}$.

When $\mu$ is a probability measure on the Borel $\sigma$-algebra of $X$, we define a \emph{measured field} structure on $S = \sqcup_{x \in X} S_x$ as a $\sigma$-algebra $\cB$ on $S$ with the property that there exists a conull Borel set $X_0 \subset X$ such that $(S_x)_{x \in X_0}$ is a Borel field.

In Sections \ref{sec.fields-Polish-spaces} and \ref{sec.fields-banach-spaces}, we define in this way Borel fields of Polish spaces, separable Banach spaces and separable C$^*$-algebras. In Sections \ref{sec.Effros-Borel} and \ref{sec.constructions-new-fields}, we prove a number of general results on how to construct new Borel fields from existing ones. We apply this in Section \ref{sec.fields-banach-spaces} to prove the following results on Borel fields $V = (V_x)_{x \in X}$ of separable Banach spaces. First, the set $V^* = (V_x^*)_{x \in X}$ has a canonical standard Borel structure such that $\ball V^*$ becomes a Borel field of compact Polish spaces (using the weak$^*$ topology on each $\ball V_x^*$). Second, if $W = (W_x)_{x \in X}$ is a family of closed subspaces of $(V_x)_{x \in X}$ such that $W \subset V$ is a Borel set, we prove that $W$ is \emph{not necessarily} a Borel field, again because there may not be enough Borel sections, but always is a Borel field if all the spaces $V_x$ are reflexive. Also, for any probability measure $\mu$ on $X$, $W$ always is a measured field.

In Section \ref{sec.completion}, we define Borel fields of separable pseudometric spaces, resp.\ separable seminormed spaces, and we prove that their separation-completions canonically are Borel fields of Polish spaces, resp.\ separable Banach spaces. This will be a key tool in Section \ref{sec.fields-vNalg} to prove that the ``abstract'' and ``concrete'' definitions of a Borel field of von Neumann algebras are equivalent.

In Section \ref{sec.fields-operators}, given Borel fields of separable Banach spaces $V = (V_x)_{x \in X}$ and $W = (W_x)_{x \in X}$, we define a canonical standard Borel structure on the field $B(V,W) = (B(V_x,W_x))_{x \in X}$ of bounded linear operators from $V_x$ to $W_x$. We prove that for any probability measure $\mu$ on $X$, $\ball B(V,W)$ is a measured field of Polish spaces (where each $B(V_x,W_x)$ is equipped with the strong topology). We also prove that in many cases, in particular when all $W_x$ are reflexive, $\ball B(V,W)$ actually is a Borel field of Polish spaces, but suspect that this is not true in general.

We also define in Section \ref{sec.fields-operators} Borel fields of Polish groups. Given a Borel field of separable Banach spaces $V = (V_x)_{x \in X}$, we define a canonical standard Borel structure on the field $\Isom V = (\Isom V_x)_{x \in X}$ of isometric automorphisms of $V_x$. While we prove that for every probability measure $\mu$ on $X$, $\Isom V$ is a measured field of Polish groups, we also show that $\Isom V$ is not always a Borel field of Polish groups.

In Section \ref{sec.fields-vNalg}, we combine several results to give two equivalent definitions for a Borel field $M = (M^x)_{x \in X}$ of von Neumann algebras $M^x$ with separable predual $M^x_*$. The first ``abstract'' definition consists of a Borel field structure on the field of separable Banach spaces $(M^x_*)_{x \in X}$ such that w.r.t.\ the corresponding Borel structure on the dual $M$ (see Section \ref{sec.fields-banach-spaces}), the multiplication and adjoint operation on $M$ are Borel maps. The second ``concrete'' definition consists of an embedding $M \subset B(H)$, where $H = (H_x)_{x \in X}$ is a Borel field of separable Hilbert spaces and $M = (M^x)_{x \in X}$ is a family of von Neumann algebras $M^x \subset B(H_x)$ such that $M \subset B(H)$ is a Borel set. We prove in Proposition \ref{prop.equivalence-defs-Borel-field-vNalg} that both definitions are equivalent. In particular, given an ``abstract'' field $M = (M^x)_{x \in X}$, we make a Borel choice of faithful normal states $\om_x$ on $M^x$ and define $H = (H_x)_{x \in X}$ as the completion of the field of pre-Hilbert spaces given by $(M^x,\om_x)$.

We then prove in Section \ref{sec.aut-groups} that there is a natural standard Borel structure on $\Aut M = (\Aut M^x)_{x \in X}$ and that, for any probability measure $\mu$ on $X$, $\Aut M$ becomes a measured field of Polish groups. We also show that $\Aut M$ is not necessarily a Borel field of Polish groups.

In Section \ref{sec.lc-groups}, we consider Borel fields of second countable locally compact groups $G = (G_x)_{x \in X}$. We solve the two problems left open in \cite{Sut85}: the field of Pontryagin dual groups $\Ghat = (\Ghat_x)_{x \in X}$ is indeed a Borel field, and we may always choose Haar measures $\lambda_x$ on $G_x$ depending on $x \in X$ in a Borel way. As a consequence, naturally associated fields such as $L^p(G) = (L^p(G_x))_{x \in X}$ are shown to be Borel fields of separable Banach spaces.

Throughout the paper, we illustrate the flexibility of our approach by proving obvious stability results of the following kind. If $V = (V_x)_{x \in X}$ is a Borel field of locally compact Polish spaces, then the field $C(V) = C(V_x)_{x \in X}$ of continuous functions from $V_x$ to $\C$ with the compact-open topology, naturally is a Borel field of Polish spaces (Proposition \ref{prop.Borel-field-cont-functions-compact-open}), while the continuous functions vanishing at infinity, $C_0(V) = C_0(V_x)_{x \in X}$, form a Borel field of separable C$^*$-algebras (Proposition \ref{prop.Borel-field-C0}).

To further demonstrate the flexibility of the theory, we prove in Section \ref{sec.tensor-crossed} that the injective and projective tensor products of Borel fields of separable Banach spaces are again Borel fields, and the same for the minimal and maximal tensor products of Borel fields of separable C$^*$-algebras. We further prove that for Borel fields of actions of locally compact Polish groups on separable C$^*$-algebras and von Neumann algebras, all associated crossed products naturally are Borel fields.

In the final Section \ref{sec.universal}, we consider universal Borel fields. We say that a Borel field $S = (S_x)_{x \in X}$ of a certain separable structure is universal if every Borel field of such a structure is isomorphic with a subfield $(S_x)_{x \in X_0}$ for a Borel set $X_0 \subset X$. We construct such universal Borel fields for Polish spaces, separable Banach spaces, von Neumann algebras with separable predual and Polish groups. This may be viewed as a reformulation in our setting of the Borel coding of Polish groups from \cite{Sut85}, the Borel coding of separable C$^*$-algebras from \cite{Kec96}, and the Borel coding of separable Banach spaces from \cite{Bos01}.

We do not really touch in this article upon the well-studied complexity questions for natural equivalence relations between separable structures (e.g.\ the isomorphism relation of separable Banach spaces). We nevertheless explain in Remark \ref{rem.about-Borel-properties} how our concept of a universal Borel field expresses which Borel codings of a separable structure are the right ones to study these complexity questions. Using any universal Borel field, one can unambiguously define which properties of, and relations between, such separable structures are Borel, or analytic, or ...

This article grew out of the appendix of the second author's PhD thesis \cite{Wou23}. Several of the results in this paper are implicitly available the literature, in particular in \cite{Eff64, Eff65, Azo81, Sut85, Kec96, Bos01, Dod08, Dod10, DG11, Bra15}. So parts of this paper can be considered as expository. The main goal of the paper is to have all definitions and basic results together in one coherent framework that we expect to serve as a very useful toolbox for anybody encountering measured fields or Borel fields of a separable mathematical structure, especially in an operator algebraic context.

{\bf Acknowledgment.} We thank Alekos Kechris, Julien Melleray, Bruno de Mendon\c{c}a Braga and the referee for useful comments on a first draft of this paper, which led to several improvements.

\section{Borel and measured fields of Polish spaces}\label{sec.fields-Polish-spaces}

We call \emph{measurable space} any set equipped with a $\sigma$-algebra. Once a measure is given on a measurable space, we call this a \emph{measured space}. We always equip product spaces with the product $\sigma$-algebra. We call \emph{standard Borel space} any measurable space that is isomorphic with a Borel subset of a Polish space.

Throughout this paper, we use that an injective Borel map $f : X \to Y$ between standard Borel maps Borel sets $U \subset X$ to Borel sets $f(U) \subset Y$, see e.g.\ \cite[Corollary 15.2]{Kec95}.

When $V = (V_x)_{x \in X}$ is a family of sets, we also view $V$ as the disjoint union of the sets $V_x$ and always denote by $\pi : V \to X$ the map satisfying $\pi(v) = x$ for all $v \in V_x$. When $V = (V_x)_{x \in X}$ and $W = (W_x)_{x \in X}$ are families of sets, we define
\[
V \times_\pi W = \{(v,w) \in V \times W \mid \pi(v) = \pi(w)\} \; .
\]
We refer to a map $\vphi : X \to V$ satisfying $\vphi(x) \in V_x$ for all $x \in X$ as a \emph{section}.

\begin{definition}\label{def.Borel-field-Polish}
A \emph{Borel field of Polish spaces} consists of a standard Borel space $X$, a family $V = (V_x)_{x \in X}$ of Polish spaces and a standard Borel structure on $V$ such that the following holds.
\begin{enumlist}
\item $\pi : V \to X$ is Borel.
\item There exists a Borel map $d : V \times_\pi V \to [0,+\infty)$ such that for every $x \in X$, the restriction of $d$ to $V_x \times V_x$ is a separable complete metric that induces the topology on $V_x$.
\item There exists a sequence of sections $\vphi_n : X \to V$ such that every $\vphi_n$ is a Borel map and such that for every $x \in X$, the set $\{\vphi_n(x) \mid n \in \N\}$ is dense in $V_x$.
\end{enumlist}
\end{definition}

We will always refer to the map $d$ in point~2 of Definition \ref{def.Borel-field-Polish} as a \emph{compatible Borel metric} and to a sequence satisfying the condition in point~3 as a \emph{dense sequence of Borel sections}.

Another way to define the measurable structure on a field of Polish spaces consists in declaring which sections $\vphi : X \to V$ are Borel. In Proposition \ref{prop.two-approach-equivalent}, we explain how this is equivalent to our Definition \ref{def.Borel-field-Polish}. We find it more natural and also more flexible to define the measurable structure just as a $\sigma$-algebra on the field satisfying properties 1, 2 and 3 in Definition \ref{def.Borel-field-Polish}.

Yet another way to define the notion of a Borel field of Polish spaces is based on the fact that every Polish space is homeomorphic to a closed subset of the Polish space $\R^\N$ with the topology of pointwise convergence. In general for a Polish space $W$, the set $F(W)$ of nonempty closed subsets of $W$ equipped with the Effros Borel structure is a standard Borel space. One may then define a Borel field of Polish spaces as a Borel map $X \to F(\R^\N)$. Also this approach is equivalent to Definition \ref{def.Borel-field-Polish} (see Sections \ref{sec.Effros-Borel} and \ref{sec.universal}), but is certainly less flexible and natural than the definition above.

Note that condition 3 in Definition \ref{def.Borel-field-Polish} is essential, as the following example illustrates.

\begin{example}[{\cite[Example 5.1.7]{Sri98}}]\label{ex.counter1}
By \cite[Example 5.1.7]{Sri98}, we can choose a Polish space $P$ and a closed subset $C \subset [0,1] \times P$ such that for every $x \in [0,1]$, the subset $C_x \subset P$ defined by $C_x = \{v \in P \mid (x,v) \in C\}$ is nonempty, but there does not exist a Borel map $\vphi : [0,1] \to P$ with the property that $\vphi(x) \in C_x$ for all $x \in [0,1]$. By restricting a compatible metric on $P$, we get a family $C = (C_x)_{x \in [0,1]}$ of Polish spaces that satisfies conditions 1 and 2 of Definition \ref{def.Borel-field-Polish}, but not condition 3, because the field $C$ does not admit any everywhere defined Borel section.
\end{example}

\begin{lemma}\label{lem.Borel-field-Polish-some-properties}
Let $V = (V_x)_{x \in X}$ be a Borel field of Polish spaces and let $d$ be a compatible Borel metric.
\begin{enumlist}
\item Whenever $\vphi_n : X \to V$ is a dense sequence of Borel sections, the Borel $\sigma$-algebra on $V$ is the smallest $\sigma$-algebra such that $\pi : V \to X$ and the maps $V \to [0,+\infty) : v \mapsto d(v,\vphi_n(\pi(v)))$ are measurable for all $n \in \N$.
\item For every $x \in X$, the restriction of the Borel $\sigma$-algebra on $V$ to $V_x$ equals the Borel $\sigma$-algebra of $V_x$.
\item Let $Y$ be a measurable space and $F_n,F : Y \to V$ maps such that $\pi(F_n(y)) = \pi(F(y))$ for all $n \in \N$, $y \in Y$. If for all $y \in Y$, $d(F_n(y),F(y)) \to 0$ and if $F_n$ is measurable for all $n$, also $F$ is measurable.
\end{enumlist}
\end{lemma}
\begin{proof}
1.\ Fix a dense sequence of Borel sections $\vphi_n : X \to V$. Define the injective Borel map
\[
\Theta : V \to X \times \R^\N : \Theta(v) = (\pi(v), (d(v,\vphi_n(\pi(v))))_{n \in \N}) \; .
\]
Then $\Theta$ is a Borel isomorphism onto its image, which means that the Borel $\sigma$-algebra on $V$ equals the smallest $\sigma$-algebra such that $\pi$ and the maps $v \mapsto d(v,\vphi_n(\pi(v)))$ are measurable. So, 1 is proven.

2 and 3.\ These are immediate consequences of 1.
\end{proof}

\begin{definition}\label{def.measured-field-Polish}
A \emph{measured field of Polish spaces} consists of a standard $\sigma$-finite measured space $(X,\mu)$, a family $V = (V_x)_{x \in X}$ of Polish spaces and a standard Borel structure on $V$ with the following properties.
\begin{enumlist}
\item $\pi : V \to X$ is Borel.
\item There exists a Borel set $X_0 \subset X$ such that $\mu(X \setminus X_0)=0$ and such that the restriction $(V_x)_{x \in X_0} = \pi^{-1}(X_0)$ is a Borel field of Polish spaces.
\end{enumlist}
\end{definition}

So, contrary to Definition \ref{def.Borel-field-Polish} of a Borel field, the definition of a measured field only requires the existence of a dense sequence of Borel sections $\vphi_n : X_0 \to V$ that are defined on a conull Borel set $X_0 \subset X$. Under the extra set theoretic axiom of $\Sigma^1_1$-determinacy (see \cite[Definition 26.3]{Kec95}), we prove that such a dense sequence of Borel sections exists automatically. We thank the referee for providing this result.

\begin{proposition}\label{prop.sections-automatic-determinacy}
Assume that the axiom of $\Sigma^1_1$-determinacy holds. Let $V = (V_x)_{x \in X}$ be a family of Polish spaces indexed by a standard $\sigma$-finite measured space $(X,\mu)$. Assume that $V$ is a standard Borel space such that $\pi : V \to X$ is Borel. Assume that there exists a Borel map $d : V \times_\pi V \to [0,+\infty)$ such that for every $x \in X$, the restriction of $d$ to $V_x \times V_x$ is a separable complete metric that induces the topology on $V_x$.

Then $V$ is a measured field of Polish spaces.
\end{proposition}
\begin{proof}
Define the Borel subset $Y \subset X \times V^\N \times V$ of elements $(x,v,w)$ satisfying $\pi(v_n) = x = \pi(w)$ for all $n \in \N$. Define $Z \subset Y$ as the subset of $(x,v,w) \in Y$ with the property that $w$ does not belong to the closure of the sequence $(v_n)_{n \in \N}$ in $V_x$. Since
$$Z = \{(x,v,w) \in Y \mid \exists k \in \N, \forall n \in \N : d(v_n,w) \geq 1/k\} \, $$
the set $Z$ is Borel. Define $A \subset X \times V^\N$ as the Borel set of elements $(x,v)$ such that $\pi(v_n) = x$ for all $n \in \N$. Defining $\theta(x,v,w) = (x,v)$, it follows that $\theta(Z) \subset A$ is an analytic subset, so that its complement $B := A \setminus \theta(Z)$ is a $\Pi^1_1$-set. By definition, an element $(x,v)$ of $A$ belongs to $B$ if and only if the sequence $(v_n)_{n \in \N}$ is dense in $V_x$.

By \cite[Corollary 36.21]{Kec95}, which relies on $\Sigma^1_1$-determinacy, we can choose a conull Borel set $X_0 \subset X$ and a Borel map $\vphi : X_0 \to V^\N$ such that $(x,\vphi(x)) \in B$ for all $x \in X_0$. Writing $\vphi(x) = (\vphi_n(x))_{n \in \N}$, we have found a dense sequence of Borel sections for $(V_x)_{x \in X_0}$.
\end{proof}

The following lemma will be used in some of our constructions of new Borel fields out of existing ones to uniquely characterize the Borel field structure.

\begin{lemma}\phantomsection\label{lem.uniqueness}
\begin{enumlist}
\item Let $V = (V_x)_{x \in X}$ be a Borel field of Polish spaces and $\vphi_k : X \to V$ a dense sequence of Borel sections.

There exists a sequence of Borel maps $\gamma_n : V \to V$ such that $\pi \circ \gamma_n = \pi$ for all $n \in \N$, such that for every $v \in V$, $\gamma_n(v) \to v$, and  such that for every $n \in \N$, there is a partition of $V$ into Borel subsets $(U_{n,k})_{k \in \N}$ such that $\gamma_n(v) = \vphi_k(\pi(v))$ for all $v \in U_{n,k}$.

\item Let $V = (V_x)_{x \in X}$ and $W = (W_x)_{x \in X}$ be Borel fields of Polish spaces. Let $\vphi_k : X \to V$ be a dense sequence of Borel sections. Let $\theta : V \to W$ be a map such that $\pi \circ \theta = \pi$ and $\theta_x : V_x \to W_x$ is continuous for every $x \in X$. Then $\theta$ is a Borel map if and only if $\theta \circ \vphi_k$ is Borel for every $k \in \N$.
\end{enumlist}
\end{lemma}
\begin{proof}
1.\ Choose a compatible Borel metric $d$ on $V$. For every $n \in \N$, define the Borel map $A_n : V \to \N$ such that $A_n(v) = k$ when $k$ is the smallest integer satisfying $d(v,\vphi_k(\pi(v))) < 1/n$. Define $\gamma_n : V \to V : \gamma_n(v) = \vphi_{A_n(v)}(\pi(v))$. With $U_{n,k} = \{v \in V \mid A_n(v) = k\}$, the first point is proven.

2.\ One implication being trivial, assume that all $\theta \circ \vphi_k$ are Borel. Take $\gamma_n : V \to V$ given by 1. Then $\theta \circ \gamma_n$ is Borel for all $n$, because its restriction to every $U_{n,k}$ is Borel. It follows from point~3 in Lemma \ref{lem.Borel-field-Polish-some-properties} that $\theta$ is Borel.
\end{proof}

\section{Closed subfields and relation with the Effros Borel structure}\label{sec.Effros-Borel}

Whenever $P$ is a Polish space, the set $F(P)$ of nonempty closed subsets of $P$ is a standard Borel space when equipped with the \emph{Effros Borel structure}. This is the smallest $\sigma$-algebra such that the sets $\{K \in F(P) \mid K \cap \cU \neq \emptyset\}$ are measurable for all open sets $\cU \subset P$. When $d$ is a metric that induces the topology of $P$, the Effros Borel structure can also be viewed as the smallest $\sigma$-algebra such that the maps $F(P) \to [0,+\infty) : K \mapsto d(v,K)$ are measurable for all $v \in P$. We refer to \cite[Section 12.C]{Kec95} for further background on the Effros Borel structure.

When $v_n$ is a dense sequence in $P$, we have that $d(v,K) = \inf_n (d(v,v_n) + d(v_n,K))$ for every $K \in F(P)$, so that also $P \times F(P) \to [0,+\infty) : (v,K) \mapsto d(v,K)$ is Borel.

A basic result (see \cite[Theorem 12.13]{Kec95} and see also Proposition \ref{prop.selection-Effros}) says that there exists a sequence of Borel maps $\vphi_n : F(P) \to P$ such that for every $K \in F(P)$, the set $K$ equals the closure of the set $\{\vphi_n(K) \mid n \in \N\}$. We thus immediately get the following result. Let $X$ be a standard Borel space and $V = (P)_{x \in X}$ the constant Borel field, with $P$ a Polish space. Let $C = (C_x)_{x \in X}$ be a family of nonempty closed subsets of $P$. Then the following are equivalent.
\begin{enumlist}
\item $C \subset V$ is a Borel set and with the restricted Borel structure, $C$ is itself a Borel field of Polish spaces.
\item The map $X \to F(P) : x \mapsto C_x$ is Borel.
\item For every $v \in P$, the map $X \to [0,+\infty) : x \mapsto d(v,C_x)$ is Borel.
\item For every Borel map $\vphi : X \to P$, the map $X \to [0,+\infty) : x \mapsto d(\vphi(x),C_x)$ is Borel.
\end{enumlist}

A similar result holds for subfields of an arbitrary, nonconstant Borel field of Polish spaces.

\begin{proposition}\label{prop.selection-Effros}
Let $V = (V_x)_{x \in X}$ be a Borel field of Polish spaces. Let $d : V \times_\pi V \to [0,+\infty)$ be a compatible Borel metric. Let $C = (C_x)_{x \in X}$ be a family of nonempty closed subsets $C_x \subset V_x$. Then the following are equivalent.
\begin{enumlist}
\item $C \subset V$ is a Borel set and with the restricted Borel structure, $C$ is itself a Borel field of Polish spaces.
\item There exists a sequence of Borel sections $\psi_n : X \to V$ such that for every $x \in X$, the set $C_x$ is the closure of $\{\psi_n(x)\mid n \in \N\}$.
\item For every Borel section $\vphi : X \to V$, the map $X \to [0,+\infty) : x \mapsto d(\vphi(x),C_x)$ is Borel.
\item There exists a dense sequence of Borel sections $\vphi_n : X \to V$ such that for every $n \in \N$, the map $X \to [0,+\infty) : x \mapsto d(\vphi_n(x),C_x)$ is Borel.
\item The map $V \to [0,+\infty) : v \mapsto d(v,C_{\pi(v)})$ is Borel.
\end{enumlist}
\end{proposition}
\begin{proof}
1 $\Rightarrow$ 2 follows immediately from the definition of a Borel field.

2 $\Rightarrow$ 3. Choose a Borel section $\vphi : X \to V$. Then, $x \mapsto d(\vphi(x),C_x) = \inf_n d(\vphi(x),\psi_n(x))$ is Borel.

3 $\Rightarrow$ 4 is trivial.

4 $\Rightarrow$ 5. Define the maps $\gamma_n : V \to V$ given by point~1 of Lemma \ref{lem.uniqueness}, with $\gamma_n(v) = \vphi_k(\pi(v))$ when $v \in U_{n,k}$. Then,
\[
d(v,C_{\pi(v)}) = \lim_n d(\gamma_n(v),C_{\pi(v)}) \; ,
\]
and for $v \in U_{n,k}$, we have that $d(\gamma_n(v),C_{\pi(v)}) = d(\vphi_k(\pi(v)),C_{\pi(v)})$, which defines a Borel map on $U_{n,k}$. It follows that $v \mapsto d(v,C_{\pi(v)})$ is Borel.

5 $\Rightarrow$ 1. Since $C = \{v \in V \mid d(v,C_{\pi(v)}) = 0\}$, it follows that $C$ is a Borel set. We have to construct a dense sequence of Borel sections $X \to C$. Fix a dense sequence of Borel sections $\vphi_n : X \to V$. Composing $\vphi_n$ with the map $v \mapsto d(v,C_{\pi(v)})$, it follows that the maps $x \mapsto d(\vphi_n(x),C_x)$ are Borel.

Fix $n \in \N$ and $\eps > 0$. Define the Borel set
\[
X_{n,\eps} = \{x \in X \mid d(\vphi_n(x),C_x) < \eps\} = \{x \in X \mid C_x \cap B(\vphi_n(x),\eps) \neq \emptyset\} \; .
\]
We construct a Borel section $\gamma_{n,\eps} : X_{n,\eps} \to C$ such that $d(\gamma_{n,\eps}(x),\vphi_n(x)) < 2 \eps$.

We inductively define for all $i \geq 0$, Borel maps $A_i : X_{n,\eps} \to \N$ such that
\begin{equation}\label{eq.my-goal-with-Ai}
C_x \cap B(\vphi_{A_i(x)}(x),2^{-i}\eps) \neq \emptyset \quad\text{and}\quad d(\vphi_{A_{i+1}(x)},\vphi_{A_{i}(x)}(x)) < 2^{-i}\eps
\end{equation}
for all $i \geq 0$ and $x \in X_{n,\eps}$. We start by defining $A_0(x) = n$ for all $x \in X_{n,\eps}$. Assume that we have defined $A_0,\ldots,A_{i-1}$ for $i \geq 1$. For every $x \in X_{n,\eps}$, write
\[J_i(x) = \{a \in \N \mid \vphi_a(x) \in B(\vphi_{A_{i-1}(x)}(x),2^{-i+1}\eps) \} \; .\]
Since $(\vphi_a(x))_{a \in \N}$ is dense in $V_x$, we get that $B(\vphi_{A_{i-1}(x)}(x),2^{-i+1}\eps) \subset \bigcup_{a \in J_i(x)} B(\vphi_a(x),2^{-i}\eps)$. Since $C_x \cap B(\vphi_{A_{i-1}(x)}(x),2^{-i+1}\eps) \neq \emptyset$, there exists an $a \in J_i(x)$ such that $C_x \cap B(\vphi_a(x),2^{-i}\eps) \neq \emptyset$. We can thus define $A_i : X_{n,\eps} \to \N$ by $A_i(x) = a$ if and only if $a$ is the smallest integer such that
\[C_x \cap B(\vphi_a(x),2^{-i}\eps) \neq \emptyset \quad\text{and}\quad d(\vphi_a(x),\vphi_{A_{i-1}(x)}(x)) < 2^{-i+1} \eps \; .\]
By construction, \eqref{eq.my-goal-with-Ai} holds.

Because $A_0(x) = n$ for all $x$, it follows from \eqref{eq.my-goal-with-Ai} that $(\vphi_{A_i(x)}(x))_{i \in \N}$ converges in $V_x$ to an element $\gamma_{n,\eps}(x)$ satisfying $d(\gamma_{n,\eps}(x),\vphi_n(x)) < 2\eps$ and $d(\gamma_{n,\eps}(x),C_x) = 0$. So, $\gamma_{n,\eps}(x) \in C_x$. By Lemma \ref{lem.Borel-field-Polish-some-properties}, the map $\gamma_{n,\eps} : X_{n,\eps} \to C$ is Borel.

Doing the same process by starting with $\eps = 1$ and $A_0 : X \to \N$ defined such that $A_0(x) = a$ if and only if $a$ is the smallest integer such that $C_x \cap B(\vphi_a(x),1) \neq \emptyset$, we find a Borel section $\gamma : X \to C$. We then define the family of Borel sections $\psi_{n,k} : X \to C$ by $\psi_{n,k}(x) = \gamma_{n,1/k}(x)$ if $x \in X_{n,1/k}$ and $\psi_{n,k}(x) = \gamma(x)$ otherwise. By construction, we have found a countable dense family of Borel sections $X \to C$.
\end{proof}

\begin{proposition}\label{prop.subfield}
Let $(X,\mu)$ be a standard $\sigma$-finite measured space and $V = (V_x)_{x \in X}$ a family of Polish spaces. Assume that $V$ is a standard Borel space such that $\pi : V \to X$ is a Borel map. Assume that the following holds.
\begin{enumlist}
\item There exists a Borel map $d : V \times_\pi V \to [0,+\infty)$ such that for every $x \in X$, the restriction of $d$ to $V_x \times V_x$ is a separable complete metric that induces the topology on $V_x$.
\item There exists a sequence of Borel sets $U_n \subset V$ such that for every $x \in X$, the set $\{U_n \cap V_x \mid n \in \N\}$ is a basis for the topology on $V_x$.
\end{enumlist}
Then $V$ is a measured field of Polish spaces.

In particular, if $V = (V_x)_{x \in X}$ is a measured field of Polish spaces and $W = (W_x)_{x \in X}$ is a family of nonempty closed subsets such that $W \subset V$ is a Borel set, then $(W_x)_{x \in X}$ is a measured field of Polish spaces.
\end{proposition}

Note that in Proposition \ref{prop.sections-automatic-determinacy}, we proved that the second assumption in Proposition \ref{prop.subfield} holds automatically, if the axiom of $\Sigma^1_1$-determinacy holds.

\begin{proof}
By the Jankov-von Neumann selection theorem (see \cite[Theorem 18.1]{Kec95}), we can choose a conull Borel set $X_0 \subset X$ such that $\pi(U_n) \cap X_0$ is Borel for all $n$ and there exist Borel sections $\psi_n : \pi(U_n) \cap X_0 \to U_n$. Write $Y_n = \pi(U_n) \cap X_0$.

Since $\bigcup_{n \in \N} (U_n \cap V_x) = V_x$ for all $x \in X$, also $\bigcup_{n \in \N} U_n = V$ and $\bigcup_{n \in \N} Y_n = X_0$. Define the Borel sections $\vphi_n : X_0 \to V$ by
\[
\vphi_n(x) = \begin{cases} \psi_n(x) &\quad\text{if $x \in Y_n$,}\\
\psi_k(x) &\quad\text{if $x \in X_0 \setminus Y_n$, $x \in Y_k$ and $x \not\in Y_i$ for all $i < k$.}
\end{cases}\]
For every $x \in X_0$, the set $\{\vphi_n(x) \mid n \in \N\}$ contains at least one element from each of the nonempty members of the family $(U_n \cap V_x)_{n \in \N}$, so that the set is dense in $V_x$.

For the second part of the proposition, we may already remove a conull Borel set from $X$ and assume that $V$ is a Borel field, with a dense sequence of Borel sections $\vphi_n : X \to V$. Then the sets
\[
U_{n,k} = \{w \in W \mid d(w,\vphi_n(\pi(w))) < 1/k\}
\]
are Borel subsets of $W$ such that for every $x \in X$, the family $(U_{n,k})_{n,k \in \N}$ is a basis for the topology on $W_x$. By the first part of the proposition, $W$ is a measured field of Polish spaces.
\end{proof}

It is quite natural to ask if Proposition \ref{prop.subfield} holds in the Borel context. More precisely:
\begin{enumlist}
\item Given a Borel field $V = (V_x)_{x \in X}$ of Polish spaces and a family $(W_x)_{x \in X}$ of closed subsets such that $W \subset V$ is Borel, does it follow that $W$ is a Borel field?
\item Does the first part of Proposition \ref{prop.subfield} hold in the Borel context?
\end{enumlist}

Clearly, a positive answer to question~2 implies a positive answer to question~1. It turns out that the answer to question~1 is in general negative, even when the Borel field $V$ is a constant field: this already follows from Example \ref{ex.counter1} above. A fortiori, the answer to question~2 is in general negative.

We also provide a positive answer to question~1 when each $W_x$ is $\sigma$-compact; see Proposition \ref{prop.subfield-sigma-compact}.

Later on we will use this to study how the answer to questions~1 and 2 changes for Borel fields of more specific structures, like Hilbert spaces and Banach spaces. The situation is surprisingly subtle: a subfield of a Borel field of Hilbert spaces (consisting of closed subspaces and forming a Borel set) is automatically a Borel field; the same holds for reflexive Banach spaces; but fails for $\ell^1(\N)$. We prove these results in Proposition \ref{prop.subfield-Banach} and Proposition \ref{prop.counterex-subfield} below.

Our positive result is based on the following selection theorem. This theorem has a long history and is essentially due to \cite{Sai75}. It can be found exactly in this form in \cite[Theorem 1.2]{Sri80}.

\begin{theorem}[{\cite[Theorem 1.2]{Sri80}}]\label{thm.selection-sigma-compact}
Let $X$ be a standard Borel space and $P$ a Polish space. If $W \subset X \times P$ is a Borel set and if for every $x \in X$, the set $W_x$ is a nonempty $\sigma$-compact subset of $P$, there exists a sequence of Borel maps $\vphi_n : X \to P$ such that for every $x \in X$, the set $\{\vphi_n(x) \mid n \in \N\}$ is dense in $W_x$.
\end{theorem}

Theorem \ref{thm.selection-sigma-compact} will be particularly useful for us in situations where $P$ is a compact Polish space, like the unit ball of a reflexive separable Banach space with the weak topology. We use Theorem \ref{thm.selection-sigma-compact} to deduce the following.

\begin{proposition}\label{prop.subfield-sigma-compact}
Let $V = (V_x)_{x \in X}$ be a Borel field of Polish spaces. Let $W = (W_x)_{x \in X}$ be a family of nonempty closed subsets $W_x \subset V_x$. Assume that $W \subset V$ is a Borel set and that for every $x \in X$, $W_x$ is a $\sigma$-compact subset of $V_x$. Then $W$ is a Borel field of Polish spaces: there exists a dense sequence of Borel sections $\psi_n : X \to W$.
\end{proposition}
\begin{proof}
Fix a dense sequence of Borel sections $\vphi_n : X \to V$ and fix a compatible Borel metric $d : V \times_\pi V \to [0,+\infty)$. As in the proof of Lemma \ref{lem.Borel-field-Polish-some-properties}, define the Borel map
\[
\Theta : V \to X \times \R^\N : \Theta(v) = (\pi(v), (d(v,\vphi_n(\pi(v))))_{n \in \N}) \; .
\]
We consider on $\R^\N$ the Polish topology of pointwise convergence. Note that $\Theta$ is injective and that, for every $x \in X$, the restriction $\Theta_x : V_x \to \R^\N$ is continuous and defines a homeomorphism from $V_x$ to $\Theta_x(V_x)$ with the induced topology.

Write $C = \Theta(V)$. It follows that $C \subset X \times \R^\N$ is a Borel set and that, for every $x \in X$, the section $C_x \subset \R^\N$ is a $\sigma$-compact subset. The conclusion thus follows from Theorem \ref{thm.selection-sigma-compact}.
\end{proof}

\section{Construction and uniqueness of standard Borel spaces}\label{sec.constructions-new-fields}

We prove the following elementary technical lemma, which will in particular be used to define a standard Borel structure on fields of continuous maps between Borel fields of Polish spaces.

\begin{lemma}\label{lem.construction-uniqueness}
Let $X$ be a standard Borel space. Let $V = (V_x)_{x \in X}$ and $W = (W_x)_{x \in X}$ be Borel fields of Polish spaces with the natural projection maps $\pi_V : V \to X$ and $\pi_W : W \to X$. Let $Y$ be a set and let $\pi_Y : Y \to X$ be a map. Write
\[Y \times_\pi V : \{(y,v) \in Y \times V \mid \pi_Y(y) = \pi_V(v)\} \quad\text{with}\quad \Pi(y,v) = \pi_Y(y) = \pi_V(v) \; ,\]
and let $C : Y \times_\pi V \to W$ be a map such that $\pi_W \circ C = \Pi$.

For every $y \in Y$, denote $C_y : V_{\pi_V(y)} \to W_{\pi_W(y)} : v \mapsto C(y,v)$. Assume that $C_y$ is continuous for every $y \in Y$.

\begin{enumlist}
\item Assume that $C_y \neq C_z$ if $y \neq z$ and let $\vphi_n : X \to V$ be a dense sequence of Borel sections. Define the injective map
\[\theta : Y \to X \times W^\N : \theta(y) = \bigl(\pi_Y(y),(C_y(\vphi_n(\pi_Y(y))))_{n \in \N}\bigr) \; .\]
\begin{enumlist}
\item The subset $\theta(Y) \subset X \times W^\N$ is Borel if and only if $Y$ admits a standard Borel structure such that the maps $\pi_Y : Y \to X$ and $C : Y \times_\pi V \to W$ are Borel.
\end{enumlist}
Assume that $\theta(Y) \subset X \times W^\N$ is a Borel subset. Then the following holds.
\begin{enumlist}[resume]
\item There is a unique standard Borel structure on $Y$ such that $\pi_Y$ and $C$ are Borel.
\item Whenever $\psi_n : X \to V$ is a dense sequence of Borel sections, the standard Borel $\sigma$-algebra on $Y$ defined in b) is the smallest $\sigma$-algebra such that the maps $\pi_Y : Y \to X$ and $Y \to W : y \mapsto C(y,\psi_n(\pi_Y(y)))$ are measurable.
\end{enumlist}
\item Assume that $Y$ is a measurable space and that $\pi_Y : Y \to X$ is measurable. Let $\vphi_n : X \to V$ be a dense sequence of Borel sections. Then the following are equivalent.
\begin{enumlist}
\item The map $C : Y \times_\pi V \to W$ is measurable.
\item For every Borel section $\vphi : X \to V$, the map $Y \to W : y \mapsto C(y,\vphi(\pi_Y(y)))$ is measurable.
\item For every $n \in \N$, the map $Y \to W : y \mapsto C(y,\vphi_n(\pi_Y(y)))$ is measurable.
\end{enumlist}
\end{enumlist}
\end{lemma}

When $W$ is merely a Polish space, rather than a Borel field of Polish spaces, and $C : Y \times_\pi V \to W$ is a map, we can consider the constant field $X \times W$ and apply Lemma \ref{lem.construction-uniqueness} to the map $Y \times_\pi V \to X \times W : (y,v) \mapsto (\pi_Y(y),C(y,v))$. This leads to an obvious variant of Lemma \ref{lem.construction-uniqueness} that we will also use and that we will also refer to as Lemma \ref{lem.construction-uniqueness}.

\begin{proof}
We first prove the second part of the lemma. The implications a) $\Rightarrow$ b) $\Rightarrow$ c) are trivial. So assume that c) holds. Choose compatible Borel metrics $d : V \times_{\pi_V} V \to [0,+\infty)$ and $d' : W \times_{\pi_W} W \to [0,+\infty)$. Define the Borel maps $I_n : V \to \N$ such that $I_n(v) = k$ if and only if $k$ is the smallest integer satisfying $d(v,\vphi_k(\pi_V(v))) < 1/n$. Then define $\theta_n : V \to V : \theta_n(v) = \vphi_{I_n(v)}(\pi_V(v))$. By construction, $d(v,\theta_n(v)) \to 0$ for all $v \in V$. By continuity of $C_y$, we find that $d'(C(y,\theta_n(v)),C(y,v)) \to 0$ for all $(y,v) \in Y \times_\pi V$. By Lemma \ref{lem.Borel-field-Polish-some-properties}, it thus suffices to prove that for every $n \in \N$, the map $C_n : (y,v) \mapsto C(y,\theta_n(v))$ is measurable.

Fix $n \in \N$ and define the Borel set $V_{n,k} \subset B$ by $V_{n,k} = \{v \in V \mid I_n(v) = k\}$. Then $\bigcup_{k \in \N} V_{n,k} = V$. The restriction of $C_n$ to $Y \times_\pi V_{n,k}$ equals the map $(y,v) \mapsto C(y,\vphi_k(\pi_Y(y)))$, which is measurable by assumption. So, $C_n$ is measurable and a) holds.

We now prove the first part of the lemma. If we have a standard Borel structure on $Y$ such that $\pi_Y$ and $C$ are Borel, it follows that $\theta$ is an injective Borel map, so that $\theta(Y) \subset X \times W^\N$ is a Borel subset. Conversely, assume that $\theta(Y)$ is a Borel set. We can then define a standard Borel structure on $Y$ such that $\theta : Y \to \theta(Y)$ is a Borel isomorphism. By construction, $\pi_Y : Y \to X$ and $y \mapsto C(y,\vphi_n(\pi_Y(y)))$ are Borel, for all $n \in \N$. It then follows from the already proven second part of the lemma that $C : Y \times_\pi V \to W$ is a Borel map.

Since for any standard Borel structure on $Y$ that makes $\pi_Y$ and $C$ Borel, the map $\theta$ is Borel, the uniqueness is obvious.

Denote by $\cB$ the standard Borel $\sigma$-algebra on $Y$. Let $\psi_n : X \to V$ be a dense sequence of Borel sections and denote by $\cB_0$ the smallest $\sigma$-algebra on $Y$ such that the maps $\pi_Y : Y \to X$ and $Y \to W : y \mapsto C(y,\psi_n(\pi_Y(y)))$ are measurable. Since $\pi_Y : Y \to X$ and $C : Y \times_\pi V \to W$ are Borel maps, we get that $\cB_0 \subset \cB$. Applying the second part of the lemma to $(Y,\cB_0)$, it follows that $C$ is measurable if we equip $Y \times_\pi V$ with the restriction of the product $\sigma$-algebra $\cB_0 \times \cB_V$. It then follows that $\theta$ is measurable from $(Y,\cB_0)$ to $\theta(Y)$ with the Borel $\sigma$-algebra. That precisely means that $\cB \subset \cB_0$. So, $\cB_0 = \cB$.
\end{proof}

In certain approaches to measurable fields, the measurable structure is rather defined by declaring which sections $\vphi : X \to V$ are Borel. This set of sections is then required to satisfy the axioms listed in Proposition \ref{prop.two-approach-equivalent} below. We prove that this approach is equivalent to ours in a canonical way.

\begin{proposition}\label{prop.two-approach-equivalent}
Let $X$ be a standard Borel space and $V = (V_x)_{x \in X}$ a family of Polish spaces, with a separable complete metric $d_x$ on $V_x$ for all $x \in X$. Assume that $\cS$ is a family of sections $\vphi : X \to V$ satisfying the following properties.
\begin{enumlist}
\item For all $\vphi,\psi \in \cS$, the map $X \to [0,+\infty) : x \mapsto d_x(\vphi(x),\psi(x))$ is Borel.
\item If $\gamma : X \to V$ is any section such that $X \to [0,+\infty) : x \mapsto d_x(\gamma(x),\vphi(x))$ is Borel for all $\vphi \in \cS$, then $\gamma \in \cS$.
\item There exists a sequence $(\vphi_n)_{n \in \N}$ in $\cS$ such that for all $x \in X$, the set $\{\vphi_n(x) \mid n \in \N\}$ is dense in $V_x$.
\end{enumlist}
There then exists a unique standard Borel structure on $V$ such that the maps $\pi : V \to X$, $d : V \times_\pi V \to [0,+\infty)$ and $\vphi : X \to V$ for $\vphi \in \cS$ are all Borel. In this way, $V = (V_x)_{x \in X}$ becomes a Borel field of Polish spaces and $\cS$ equals the set of Borel sections $X \to V$.

Conversely, if $V = (V_x)_{x \in X}$ is a Borel field of Polish spaces, the set of Borel sections $X \to V$ satisfies properties 1, 2 and 3 above.
\end{proposition}

In the next sections, we consider Borel fields of other separable structures, including separable Banach spaces, Polish groups and von Neumann algebras with separable predual. In each of these settings, an obvious analogue of Proposition \ref{prop.two-approach-equivalent} can be formulated and proven.

\begin{proof}
Fix a sequence $(\vphi_n)_{n \in \N}$ in $\cS$ as in 3. Define the injective maps
\[
\theta_x : V_x \to \R^\N : \theta_x(v) = (d_x(v,\vphi_n(x)))_{n \in \N} \quad\text{and}\quad \theta : V \to X \times \R^\N : \theta(v) = (\pi(v),\theta_{\pi(v)}(v))
\]
and write $Z = \theta(V)$. Since $\theta_x(V_x)$ is the closure of $\{\theta_x(\vphi_k(x)) \mid k \in \N\}$ in the topology of uniform convergence on $\R^\N$, we get that
\[
Z = \bigl\{ (x,z) \in X \times \R^\N \bigm| \forall r \in \N, \exists k \in \N , \forall n \in \N : |z_n - d_x(\vphi_k(x),\vphi_n(x))| < 1/r \bigr\} \; .
\]
It follows from 1 that $Z \subset X \times \R^\N$ is a Borel set. We define a standard Borel structure on $V$ such that $\theta$ is a Borel isomorphism.

By construction, the maps $\pi : V \to X$ and $\vphi : X \to V$ for $\vphi \in \cS$ are Borel. Also by construction, for every $n \in \N$, the map $V \to [0,+\infty) : v \mapsto d(v,\vphi_n(\pi(v)))$ is Borel. Since
\[
d : V \times_\pi V \to [0,+\infty) : (v,w) \mapsto d(v,w) = \inf \{d(v,\vphi_n(\pi(v))) + d(\vphi_n(\pi(w)),w) \mid n \in \N\} \; ,
\]
it follows that also $d$ is Borel.

For any other standard Borel structure making the maps $\pi$, $d$ and $\vphi \in \cS$ Borel, we have that $\theta$ is Borel, so that uniqueness follows. By construction, $V$ is a Borel field of Polish spaces and every $\vphi \in \cS$ is Borel. If $\psi : X \to V$ is a Borel section, we get that $x \mapsto d_x(\psi(x),\vphi(x))$ is Borel for every Borel section $\vphi : X \to V$ and thus, in particular, for every $\vphi \in \cS$. So by 2, $\psi \in \cS$.

Conversely assume that $V = (V_x)_{x \in X}$ is a Borel field of Polish spaces and define $\cS$ as the set of Borel sections $X \to V$. By definition, statements 1 and 3 hold, while 2 follows from Lemma \ref{lem.Borel-field-Polish-some-properties}.
\end{proof}

\section{Borel and measured fields of separable Banach spaces and duality}\label{sec.fields-banach-spaces}

In \cite{Bos01}, separable Banach spaces are encoded as the standard Borel space of closed subspaces of $C(\Delta)$, where $\Delta$ is the Cantor set. One could thus define Borel fields of separable Banach spaces as Borel maps to the space of closed subspaces of $C(\Delta)$. We however give the following more generic definition and prove in Proposition \ref{prop.universal-field-Banach} that it is equivalent to the approach of \cite{Bos01}.

\begin{definition}\label{def.Borel-field-Banach}
A \emph{Borel field of separable Banach spaces} consists of a standard Borel space $X$, a family $V = (V_x)_{x \in X}$ of separable Banach spaces and a standard Borel structure on $V$ with the following properties.
\begin{enumlist}
\item $\pi : V \to X$ is Borel.
\item The map $V \to [0,+\infty) : v \mapsto \|v\|$ is Borel.
\item The maps $V \times_\pi V \to V : (v,w) \mapsto v+w$ and $\C \times V \to V : (\lambda,v) \mapsto \lambda v$ are Borel.
\item There exists a dense sequence of Borel sections $\vphi_n : X \to V$.
\end{enumlist}

A \emph{measured field of separable Banach spaces} consists of a standard $\sigma$-finite measured space $(X,\mu)$, a family $V = (V_x)_{x \in X}$ of separable Banach spaces and a standard Borel structure on $V$ such that $\pi : V \to X$ is Borel and such that there exists a conull Borel set $X_0 \subset X$ with $(V_x)_{x \in X_0}$ being a Borel field of Banach spaces.
\end{definition}

The following provides an easy example of Borel fields of separable Banach spaces.

\begin{proposition}\label{prop.Borel-field-CK}
Let $K = (K_x)_{x \in X}$ be a Borel field of Polish spaces. Assume that every $K_x$ is compact and consider the family $C(K) = C(K_x)_{x \in X}$ of continuous functions from $K_x$ to $\C$ with the supremum norm.

There exists a unique standard Borel structure on $C(K)$ such that the maps $\pi : C(K) \to X$ and $C(K) \times_\pi K \to \C : (F,k) \mapsto F(k)$ are Borel. In this way, $C(K)$ is a Borel field of separable Banach spaces.
\end{proposition}
\begin{proof}
Fix a dense sequence of Borel sections $\vphi_n : X \to K$ and fix a compatible Borel metric $d : K \times_\pi K \to [0,+\infty)$. Since continuous functions on a compact Polish space are uniformly continuous, the image of the injective map
\[
\theta : C(K) \to X \times \C^\N : F \mapsto (\pi(F),(F(\vphi_n(\pi(F))))_{n \in \N})
\]
equals
\[
Z = \{(x,z) \mid \forall r \in \N , \exists s \in \N , \forall n,m \in \N : \;\text{if $d(\vphi_n(x),\vphi_m(x)) < 1/s$ then $|z_n-z_m| < 1/r$}\;\} \; .
\]
Since $Z$ is a Borel set, it follows from the first part of Lemma \ref{lem.construction-uniqueness} that there exists a unique standard Borel structure on $C(K)$ such that the maps $\pi : C(K) \to X$ and $C(K) \times_\pi K \to \C : (F,k) \mapsto F(k)$ are Borel. It also follows from Lemma \ref{lem.construction-uniqueness} that the norm, addition and scalar multiplication are Borel maps.

It remains to construct a dense sequence of Borel sections $X \to C(K)$. For every $n \in \N$, define the section $F_n : X \to C(K) : (F_n(x))(k) = d(k,\vphi_n(x))$. Since for all $m \in \N$, the map $X \to \C : x \mapsto (F_n(x))(\vphi_m(x)) = d(\vphi_m(x),\vphi_n(x))$ is Borel, it follows from Lemma \ref{lem.construction-uniqueness} that $F_n$ is a Borel section for every $n \in \N$. By the Stone-Weierstrass theorem, all linear combinations with coefficients in $\Q + i \Q$  of products of $F_n$ and $\overline{F_n}$ then provide a countable family of Borel sections $X \to C(K)$. Since the functions $F_n(x) \in C(K_x)$ separate the points of $K_x$, we have indeed found a countable dense family of Borel sections, so that $C(K)$ is a Borel field of Banach spaces.
\end{proof}

Having defined Borel fields of separable Banach spaces, there are obvious definitions of Borel fields of separable Banach spaces with extra structure, such as C$^*$-algebras or Banach algebras, by requiring that the extra structure maps are Borel. For later reference, we give the following explicit definition. In \cite{Kec96}, a Borel parametrization of separable C$^*$-algebras was given. In Proposition \ref{prop.universal-field-Cstar}, we prove that this Borel parametrization can be viewed as a universal Borel field of separable C$^*$-algebras.

\begin{definition}\label{def.Borel-field-Cstar}
We say that $A = (A_x)_{x \in X}$ is a Borel field of separable C$^*$-algebras, resp.\ separable Banach algebras, or separable Banach $*$-algebras, if $(A_x)_{x \in X}$ is a Borel field of separable Banach spaces such that the maps $A \times_\pi A \to A$ and $A \to A$ given by multiplication and adjoint are Borel.
\end{definition}

So by Proposition \ref{prop.Borel-field-CK}, whenever $(K_x)_{x \in X}$ is a Borel field of compact Polish spaces, we have the natural Borel field $C(K) = C(K_x)_{x \in X}$ of separable C$^*$-algebras.

Whenever $V$ is a separable Banach space, the unit ball of the dual Banach space $V^*$, which we denote as $\ball V^*$, is a compact Polish space in the weak$^*$ topology.

As another application of Theorem \ref{thm.selection-sigma-compact}, we prove that for a Borel field $V = (V_x)_{x \in X}$ of separable Banach spaces, the field $(\ball V_x^*)_{x \in X}$ is in a canonical way a Borel field of the Polish spaces $\ball V_x^*$ with the weak$^*$ topology. This result is essentially contained in \cite{Dod08}.

\begin{proposition}\label{prop.dual-field-Banach}
Let $V = (V_x)_{x \in X}$ be a Borel field of separable Banach spaces. Define $V^* = (V_x^*)_{x \in X}$.
\begin{enumlist}
\item There is a unique standard Borel structure on $V^*$ such that $\pi : V^* \to X$ and $V^* \times_\pi V \to \C : (\om,v) \mapsto \om(v)$ are Borel maps.
\item Given any total sequence of Borel sections $\psi_k : X \to V$, this standard Borel structure on $V^*$ is the smallest $\sigma$-algebra on $V^*$ such that the maps $\pi : V^* \to X$ and $V^* \to \C : \om \mapsto \om(\psi_k(\pi(\om)))$ are Borel.
\item The maps $V^* \to [0,+\infty) : \om \mapsto \|\om\|$, $V^* \times_\pi V^* \to V^* : (\om,\eta) \mapsto \om + \eta$ and $\C \times V^* \to V^* : (\lambda,\om) \mapsto \lambda \om$ are Borel.
\item Denoting $\ball V^* = (\ball V^*_x)_{x \in X}$, the restriction of the Borel structure on $V^*$ to $\ball V^*$ turns $\ball V^*$ into a Borel field of the Polish spaces $\ball V_x^*$ with the weak$^*$ topology.
\item If every $V_x^*$ is a separable Banach space, then $V^*$ is a Borel field of separable Banach spaces.
\end{enumlist}
\end{proposition}
\begin{proof}
Fix a dense sequence of Borel sections $\vphi_n : X \to \ball V$. For every $k \in \N$, write $D_k = \{\lambda \in \C \mid |\lambda| \leq k\}$ and consider the compact Polish space $D_k^\N$. Write $\cQ = \Q + i \Q$ and define $Z_k \subset X \times D_k^\N$ by
\begin{multline*}
Z_k = \bigl\{ (x,z) \in X \times D_k^\N \bigm| \forall n_1,n_2,n_3 \in \N , \forall \al_1,\al_2,\al_3 \in \cQ :\\ |\al_1 z_{n_1} - \al_2 z_{n_2} - \al_3 z_{n_3}| \leq k \|\al_1 \vphi_{n_1}(x) - \al_2 \vphi_{n_2}(x) - \al_3 \vphi_{n_3}(x) \| \bigr\} \; .
\end{multline*}
Note that for every $k \in \N$, $Z_k \subset X \times D_k^\N$ is Borel and $Z_{k,x} \subset D_k^\N$ is closed for every $x \in X$. Write $Z = \bigcup_{k=1}^\infty Z_k$. Then $Z \subset X \times \C^\N$ is a Borel set and
\begin{equation}\label{eq.Theta}
\Theta : V^* \to Z : \Theta(\om) = (\pi(\om),(\om(\vphi_n(x)))_{n \in \N})
\end{equation}
is a bijection satisfying $\Theta(\ball V^*) = Z_1$.

1 and 2 now follow immediately from the first part of Lemma \ref{lem.construction-uniqueness}. We have that $\Theta : V^* \to Z$ is a Borel isomorphism.

3.\ Define $N : Z \to [0,+\infty) : N(x,z) = \sup_{n \in \N} |z_n|$. Then $\|\om\| = N(\Theta(\om))$, so that $\om \mapsto \|\om\|$ is Borel. Since addition $\C^\N \times \C^\N \to \C^\N$ and scalar multiplication $\C \times \C^\N \to \C^\N$ are Borel maps, using $\Theta$, we conclude that addition and scalar multiplication are Borel on $V^*$.

4.\ When $\vphi_n : X \to \ball V$ is a dense sequence of Borel sections, the map
$$d : \ball V^* \times_\pi \ball V^* \to [0,+\infty) : d(\om,\mu) = \sum_{n =1}^\infty 2^{-n} |\om(\vphi_n(\pi(\om))) - \mu(\vphi_n(\pi(\mu)))|$$
is Borel and, for every $x \in X$, restricts to a separable complete metric on $\ball V_x^*$ that is compatible with the weak$^*$ topology.

Note that $D_1^\N$ is a compact Polish space and that $Z_1 \subset X \times D_1^\N$ is a Borel set with the property that for every $x \in X$, the subset $Z_{1,x}$ of $D_1^\N$ is compact. By Theorem \ref{thm.selection-sigma-compact}, we find a dense sequence of Borel sections $\gamma_n : X \to D_1^\N$. Define $\om_n(x) = \Theta^{-1}(x,\gamma_n(x))$. Then, $\om_n : X \to \ball V^*$ is a weak$^*$ dense sequence of Borel sections. So $\ball V^*$ is a Borel field of Polish spaces.

5.\ If every $V_x^*$ is a separable Banach space, we get that every $Z_{1,x}$ is a norm separable subset of $D_1^\N$ w.r.t.\ the supremum norm on $D_1^\N$. It then follows from \cite[Proposition 2]{Dod08} that there exists a sequence of Borel maps $\gamma_n : X \to D_1^\N$ such that for every $x \in X$, $Z_{1,x}$ equals the uniform closure of $\{\gamma_n(x) \mid n \in \N\}$. Then $\om_{n,q}(x) = q \Theta^{-1}(x,\gamma_n(x))$, with $q \in \Q$ and $n \in \N$, is a countable norm dense family of Borel sections $\om_n : X \to V^*$. So, $V^*$ is a Borel field of separable Banach spaces.
\end{proof}

The following two results show that subfields of closed subspaces of \emph{reflexive} Banach spaces behave well, while pathologies may arise in the non-reflexive case.

\begin{proposition}\label{prop.subfield-Banach}
Let $V = (V_x)_{x \in X}$ be a Borel field of separable \emph{reflexive} Banach spaces. Let $W \subset V$ be a Borel set such that for every $x \in X$, $W_x \subset V_x$ is a closed subspace. Then $W$ is a Borel field of Banach spaces: there exists a dense sequence of Borel sections $\psi_n : X \to W$.
\end{proposition}
\begin{proof}
By reflexivity, each $\ball V_x$ with the weak topology is a compact Polish space. By Proposition \ref{prop.dual-field-Banach}, we can take a weak$^*$ dense sequence of Borel sections $\om_n : X \to \ball V^*$. Then
\[
d : \ball V \times_\pi \ball V \to [0,+\infty) : d(v,w) = \sum_{n \in \N} 2^{-n} |\langle v-w,\om_n(\pi(v))\rangle|
\]
is a Borel map whose restriction to each $\ball V_x \times \ball V_x$ is a complete separable metric that induces the weak topology on $\ball V_x$. So, $\ball V$ is a Borel field of Polish spaces. Every $\ball W_x \subset \ball V_x$ is a nonempty compact subset. By Proposition \ref{prop.subfield-sigma-compact}, we can choose a weakly dense sequence of Borel sections $\gamma_n : X \to \ball W_x$.

Denoting $\cQ = \Q + i \Q$, by convexity, the countable family of Borel sections $\al_1 \gamma_{n_1} + \cdots + \al_k \gamma_{n_k}$ for $k \in \N$, $\al_1,\ldots,\al_k \in \cQ$ and $n_1,\ldots,n_k \in \N$ is a norm dense family of Borel sections $X \to W$.
\end{proof}

\begin{proposition}[{\cite[Theorem 3.8]{DG11}}]\label{prop.counterex-subfield}
Let $V = (\ell^1(\N))_{x \in [0,1]}$ be the constant field of Banach spaces $\ell^1(\N)$ with dual field $V^* = (\ell^\infty(\N))_{x \in [0,1]}$.
\begin{enumlist}
\item (Proposition \ref{prop.subfield-Banach} may fail in the non-reflexive case.) There exists a Borel set $W \subset [0,1] \times \ell^1(\N)$ such that for every $x \in [0,1]$, $W_x \subset \ell^1(\N)$ is a closed subspace, but nevertheless $W$ is not a Borel field of Banach spaces, because there is no dense sequence of Borel sections.
\item (The intersection operation is ill-behaved.) There exist Borel fields $W = (W_x)_{x \in [0,1]}$ and $W' = (W'_x)_{x \in [0,1]}$ of closed subspaces of $\ell^1(\N)$ such that $W \cap W'$ is no longer a Borel field of Banach spaces, because there is no dense sequence of Borel sections.
\end{enumlist}
\end{proposition}
\begin{proof}
We start by proving 2, which is essentially contained in \cite[Theorem 3.8]{DG11}. Denote by $\cS = \N^{< \N}$ the set of finite sequences $(s_1 \cdots s_n)$ of natural numbers, including the empty sequence $()$. For every $s \in \cS$, denote by $|s|$ its length. For $s,t \in \cS$, we write $s \preceq t$ if $|s| \leq |t|$ and $t$ is a ``continuation'' of $s$, i.e.\ $t_i = s_i$ for all $i \in \{1,\ldots,|s|\}$. By convention $() \preceq t$ for all $t \in T$. A \emph{tree} is a subset $T \subset \cS$ with the property that $s \in T$ whenever $t \in T$ and $s \preceq t$. We say that a tree $T$ is \emph{ill-founded} if there exists an $\al \in \N^\N$ such that $(\al_1 \cdots \al_n) \in T$ for all $n \in \N$, i.e.\ when the tree has at least one infinite branch.

The set $\cT$ of trees is a Borel subset of $2^\cS$ and in this way, $\cT$ is a standard Borel space. The subset of ill-founded trees is the standard example of a complete analytic, and thus non Borel, set. Since $\cT$ is Borel isomorphic with $[0,1]$ and since all $\ell^1$-spaces of countable infinite sets are isometrically isomorphic, we will construct $W$ and $W'$ as subfields of the constant field $(\ell^1(\cS) \oplus \ell^1(\cS))_{T \in \cT}$.

For every $s \in \cS$, denote by $\delta_s \in \ell^1(\cS)$ the element that is equal to $1$ in $s$ and equal to $0$ elsewhere. Define the elements $a_s,b_s \in \ell^1(\cS)$ by
\[
a_s = \sum_{t \in \cS, t \preceq s} 2^{-|t|} \delta_t \quad\text{and}\quad b_s = 2^{-|s|} \delta_s \; .
\]
For every tree $T \in \cT$, define the closed subspace $W_T \subset \ell^1(\cS) \oplus \ell^1(\cS)$ as the closed linear span of the vectors $a_s \oplus b_s$, $s \in T$. So by construction, $W = (W_T)_{T \in \cT}$ is a Borel field of Banach spaces.

Let $T \in \cT$ be a tree. Note that by construction, for all $s \in T$, we have that
\[
a_s(t) = 0 = b_s(t) \;\;\text{for all $t \in \cS \setminus T$, and}\quad a_s(t) = b_s(t) + 2 \sum_{k \in \N} a_s(tk) \;\;\text{for all $t \in \cS$.}
\]
It follows that for every tree $T \in \cT$ and every $\xi \oplus \eta \in W_T$, we have that
\begin{equation}\label{eq.my-prop}
\xi(t) = 0 = \eta(t) \;\;\text{for all $t \in \cS \setminus T$, and}\quad \xi(t) = \eta(t) + 2 \sum_{k \in \N} \xi(tk) \;\;\text{for all $t \in \cS$.}
\end{equation}

Define the constant field $W'_T = \ell^1(\cS) \oplus \{0\}$. Then also $W'$ is a Borel field of Banach spaces. We claim that $W_T \cap W'_T \neq \{0\}$ if and only if the tree $T$ is ill-founded. Indeed, first assume that $\xi \in \ell^1(\cS)$ such that $\xi \neq 0$ and $\xi \oplus 0 \in W_T$. We use \eqref{eq.my-prop}. Since $\xi \neq 0$ and $\xi(t) = 0$ for all $t \in \cS \setminus T$, we can choose $t_0 \in T$ such that $\xi(t_0) \neq 0$. Since now $\eta = 0$, we get that $\xi(t_0) = 2 \sum_{k \in \N} \xi(t_0 k)$. There thus exists a $k_1 \in \N$ such that $\xi(t_0 k_1) \neq 0$. Since $\xi(t) = 0$ for all $t \in \cS \setminus T$, it follows that $t_0 k_1 \in T$. We can repeat the same reasoning and find inductively a sequence $k_1,k_2,\ldots$ in $\N$ such that $(t_0 k_1 \cdots k_n) \in T$ for all $n \in \N$. So, $T$ is ill-founded.

Conversely, if $T$ is ill founded and $\al \in \N^\N$ such that $s_n := (\al_1 \cdots \al_n) \in T$ for all $n \in \N$, it follows that $a_{s_n} \oplus b_{s_n}$ converges to $\xi \oplus 0$, where
\[
\xi = \delta_{()} + \sum_{n=1}^\infty 2^{-n} \delta_{s_n} \; .
\]
So, $W_T \cap W'_T \neq \{0\}$. This concludes the proof of the claim.

To conclude the proof of 2, assume that $(W_T \cap W'_T)_{T \in \cT}$ is a Borel field and choose a dense sequence of Borel sections $\vphi_n : \cT \to W \cap W'$. Since $W_T \cap W'_T \neq \{0\}$ if and only if
\[
T \in \bigcup_{n \in \N} \bigl\{T \in \cT \bigm| \vphi_n(T) \neq 0 \bigr\} \; ,
\]
we reach the contradiction that the set of ill-founded trees is Borel.

The first statement immediately follows from the second: since $W \subset V$ and $W' \subset V$ are Borel sets, also $W \cap W'$ is a Borel set. But we have seen above that $W \cap W'$ is not a Borel field of Banach spaces.
\end{proof}

Weak$^*$ closed subfields of a dual field $V^*$ of Banach spaces always behave well as Borel fields, which we prove in the following Proposition \ref{prop.duality}. Nevertheless, certain subtleties arise in the Borel context, but not in the measured context, as we explain in Proposition \ref{prop.counterex-weak-star-closed-subfield} below.

When $V$ is a separable Banach space and $W \subset V^*$ is a weak$^*$ closed subspace, we consider the closed subspace $W^\perp \subset V$ given by $W^\perp = \{v \in V \mid \forall \om \in W : \om(v)=0\}$. Then there is a canonical isomorphism $(V/W^\perp)^* \cong W$.

\begin{proposition}\label{prop.duality}
Let $V = (V_x)_{x \in X}$ be a Borel field of separable Banach spaces. Define $V^* = (V_x^*)_{x \in X}$ with the standard Borel structure given by Proposition \ref{prop.dual-field-Banach}.
Let $W = (W_x)_{x \in X}$ be a family of weak$^*$ closed subspaces $W_x \subset V_x^*$. Write $W^\perp = (W^\perp_x)_{x \in X}$, which is a family of closed subspaces of $V_x$. Then the following statements are equivalent.

\begin{enumlist}
\item $W \subset V^*$ is a Borel set.
\item There exists a sequence of Borel sections $\psi_n : X \to V^*$ such that for every $x \in X$, $\psi_n(x) \in \ball W_x$ and the sequence $\psi_n(x)$ is weak$^*$ dense in $\ball W_x$.
\item $W^\perp$ is a Borel field of Banach spaces: $W^\perp \subset V$ is a Borel set and there exists a dense sequence of Borel sections $\gamma_n : X \to W^\perp$.
\end{enumlist}

Write $W_* = (V_x / W^\perp_x)_{x \in X}$. If one of the above statements holds, there is a unique standard Borel structure on $W_*$ such that the maps $\pi : W_* \to X$ and $W_* \times_\pi W \to \C$ are Borel. With this standard Borel structure, $W_*$ is a Borel field of Banach spaces, the quotient map $V \to W_*$ is Borel and the Borel structure on $(W_*)^*$ equals the Borel structure induced on $W \subset V^*$ by $V^*$.
\end{proposition}
\begin{proof}
1 $\Rightarrow$ 2. Since $\ball W \subset \ball V^*$ is a Borel set and for every $x \in X$, $\ball W_x$ is a closed subset of the compact Polish space $\ball V_x^*$, the conclusion follows from Proposition \ref{prop.subfield-sigma-compact}.

2 $\Rightarrow$ 3. Note that for every Borel section $\vphi : X \to V$, the map
\[
x \mapsto d(\vphi(x),W_x^\perp) = \|\vphi(x)+W_x^\perp\|_{V_x/W_x^\perp} = \sup_n |\langle \vphi(x),\psi_n(x)\rangle|
\]
is Borel, because all the maps $x \mapsto \langle \vphi(x),\psi_n(x)\rangle$ are Borel. So, 3 follows from Proposition \ref{prop.selection-Effros}.

3 $\Rightarrow$ 1. Fix a dense sequence of Borel sections $\gamma_n : X \to W^\perp$. Since for every $x \in X$, we have that $W_x = W_x^{\perp\perp}$, we get that
\[
W = \{\om \in V^* \mid \forall n \in \N : \langle \gamma_n(\pi(\om)),\om \rangle = 0\}
\]
is a Borel subset of $V^*$.

To prove the final statement, fix a weak$^*$ dense sequence of Borel sections $\psi_n : X \to \ball W$ and fix a dense sequence of Borel sections $\vphi_n : X \to V$. Define the Borel subset $Z \subset X \times \C^\N$ of pairs $(x,z)$ such that
\[
\forall r \in \N , \exists k \in \N , \forall n \in \N : |z_n -  \langle \vphi_k(x),\psi_n(x)\rangle | < 1/r \; .
\]
Then $\Delta : W_* \to Z : v \mapsto (\pi(v),(\langle v,\psi_n(\pi(v))\rangle)_{n \in \N})$ is a bijection. By the first part of Lemma \ref{lem.construction-uniqueness}, there is a unique standard Borel structure on $W_*$ such that $\pi : W_* \to X$ and the pairing $W_* \times_\pi W \to \C$ are Borel.

By construction, the quotient map $V \to W_*$ is Borel and composing $\vphi_n$ with this quotient map, we obtain a dense sequence of Borel sections for $W_*$. Also
$$W_* \to [0,+\infty) : v \mapsto \|v\| = \sup \bigl\{ |\langle v, \psi_n(\pi(v))\rangle| \bigm| n \in \N \bigr\}$$
is Borel. So $W_*$ is a Borel field of Banach spaces. It follows from the unique characterization of the Borel structure on $(W_*)^*$ that the identification $(W_*)^* \cong W$ is a Borel isomorphism.
\end{proof}

The following example shows that when working with dual fields of Banach spaces, it is essential to have sections taking values in the unit ball.

\begin{proposition}\label{prop.counterex-weak-star-closed-subfield}
Let $V = (V_x)_{x \in X}$ be a Borel field of separable Banach spaces. Define $V^* = (V_x^*)_{x \in X}$ with the standard Borel structure given by Proposition \ref{prop.dual-field-Banach}. Let $\om_n : X \to V^*$ be a sequence of Borel sections and define $W_x \subset V_x^*$ as the weak$^*$ closed linear span of the set $\{\om_n(x) \mid n \in \N\}$.
\begin{enumlist}
\item For every $\sigma$-finite measure $\mu$ on $X$, there exists a conull Borel set $X_0 \subset X$ such that $W \cap \pi^{-1}(X_0) \subset V^*$ is a Borel set and $(\ball W_x)_{x \in X_0}$ and $(W^\perp_x)_{x \in X_0}$ are Borel fields.
\item It may happen that $W \subset V^*$ is \emph{not} a Borel set.
\end{enumlist}
\end{proposition}
\begin{proof}
1.\ Since $W^\perp = \{v \in V \mid \forall n \in \N : \langle v, \om_n(\pi(v))\rangle = 0\}$, we see that $W^\perp \subset V$ is a Borel set. By Proposition \ref{prop.subfield}, we can take a conull Borel set $X_0 \subset X$ such that $(W^\perp_x)_{x \in X_0}$ is a Borel field of Banach spaces. So we can take a dense sequence of Borel sections $\eta_n : X_0 \to W^\perp$. Since $W = W^{\perp \perp}$, it follows that
\[
W \cap \pi^{-1}(X_0) = \{\om \in V^* \mid \pi(\om) \in X_0 \;\;\text{and}\;\; \forall n \in \N : \langle \eta_n(\pi(\om)), \om \rangle = 0 \}
\]
is a Borel set.

2.\ Proposition \ref{prop.counterex-subfield} provides, for $X = [0,1]$, families of closed subspaces $A_x \subset \ell^1(\N)$ and $B_x \subset \ell^1(\N)$ such that $A$ and $B$ are Borel fields of separable Banach spaces, which are subfields of the constant field $V = (\ell^1(\N))_{x \in X}$, and such that $A \cap B$ is not a Borel field of separable Banach spaces.

Consider the weak$^*$ closed subspaces $A^\perp_x \subset \ell^\infty(\N)$ and $B^\perp_x \subset \ell^\infty(\N)$. Choosing dense sequences of Borel sections $\al_n : X \to A$ and $\be_n : X \to B$, we have that
\begin{align*}
& A^\perp = \{\om \in V^* \mid \forall n \in \N : \langle \al_n(\pi(\om)),\om\rangle=0\} \quad\text{and}\\
& B^\perp = \{\om \in V^* \mid \forall n \in \N : \langle \be_n(\pi(\om)),\om\rangle=0\} \; ,
\end{align*}
so that $A^\perp \subset V^*$ and $B^\perp \subset V^*$ are Borel sets. By Proposition \ref{prop.duality}, we can choose weak$^*$ dense sequences of Borel sections $\al'_n : X \to A^\perp$ and $\be'_n : X \to B^\perp$.

Define $W_x \subset \ell^\infty(\N)$ as the weak$^*$ closed linear span of the set $\{\al'_n(x) \mid n \in \N\} \cup \{\be'_k(x) \mid k \in \N\}$. By construction, $W_x$ equals the weak$^*$ closure of $A^\perp_x + B^\perp_x$. By construction, $A \cap B = W^\perp$. Since $A \cap B$ is not a Borel field of separable Banach spaces, it follows from Proposition \ref{prop.duality} that $W \subset V^*$ is not a Borel set.
\end{proof}

We already saw in Proposition \ref{prop.counterex-subfield} that the intersection operation is ill-behaved on Borel fields of separable Banach spaces. Dually, the operation $(A,B) \mapsto \overline{A+B}$ is ill-behaved for Borel fields of dual spaces.

\begin{proposition}\label{prop.dual-intersection-and-sum}
Let $V = (V_x)_{x \in X}$ be a Borel field of separable Banach spaces. Define $V^* = (V_x^*)_{x \in X}$ with the standard Borel structure given by Proposition \ref{prop.dual-field-Banach}.

Let $A = (A_x)_{x \in X}$ and $B = (B_x)_{x \in X}$ be families of weak$^*$ closed subspaces of $V_x^*$ such that $A \subset V^*$ and $B \subset V^*$ are Borel sets (and thus all the nice properties in  Proposition \ref{prop.dual-field-Banach} hold). Define $C_x$ as the weak$^*$ closure of $A_x+B_x$.

\begin{enumlist}
\item For every $\sigma$-finite measure $\mu$ on $X$, there exists a conull Borel set $X_0 \subset X$ such that $C \cap \pi^{-1}(X_0) \subset V^*$ is a Borel set.
\item It may happen that $C \subset V^*$ is \emph{not} a Borel set.
\end{enumlist}
\end{proposition}
\begin{proof}
1.\ By Proposition \ref{prop.dual-field-Banach}, we can choose weak$^*$ dense sequences of Borel sections $\al_n : X \to \ball A$ and $\be_n : X \to \ball B$. Taking the $\al_n$ and $\be_k$ together, we have a countable family of Borel sections $X \to V^*$ and $C_x$ equals the weak$^*$ closed linear span of $\{\al_n(x) \mid n \in \N\} \cup \{\be_k(x) \mid k \in \N\}$. The conclusion then follows from statement~1 in Proposition \ref{prop.counterex-weak-star-closed-subfield}.

2.\ This is precisely the way in which we proved statement~2 of Proposition \ref{prop.counterex-weak-star-closed-subfield}.
\end{proof}

\section{Borel fields of separable metric spaces, normed spaces and completions}\label{sec.completion}

Recall that a \emph{pseudometric} on a set $V$ is a map $d : V \times V \to [0,+\infty)$ satisfying $d(v,v) = 0$, $d(v,w) = d(w,v)$ and $d(v,w) \leq d(v,u) + d(u,w)$ for all $u,v,w \in V$. The \emph{separation-completion} of a pseudometric space $(V,d)$ is a complete metric space $(V',d')$ together with a map $\gamma : V \to V'$ such that $\gamma(V)$ is dense in $V'$ and $d'(\gamma(v),\gamma(w)) = d(v,w)$ for all $v,w \in V$. Up to an isometric isomorphism preserving the map $V \to V'$, the separation-completion is unique.

We prove that for a Borel field of separable pseudometric spaces, the separation-completion can be chosen in a Borel way.

\begin{definition}\label{def.Borel-field-metric-spaces}
A \emph{Borel field of separable pseudometric spaces} consists of a standard Borel space $X$, a family $(V,d) = (V_x,d_x)_{x \in X}$ of pseudometric spaces and a standard Borel structure on $V$ such that the maps $\pi : V \to X$ and $d : V \times_\pi V \to [0,+\infty)$ are Borel and such that there exists a dense sequence of Borel sections $\vphi_n : X \to V$.
\end{definition}

\begin{proposition}\label{prop.completion}
Let $(V,d) = (V_x,d_x)_{x \in X}$ be a Borel field of separable pseudometric spaces. Assume that $(V',d') = (V'_x,d'_x)_{x \in X}$ is any family of separation-completions for $(V_x,d_x)$ with corresponding $\gamma_x : V_x \to V'_x$.
\begin{enumlist}
\item $V'$ admits a unique standard Borel structure such that the map $\pi : V' \to X$, the map $\gamma : V \to V'$ and the metric $d' : V' \times_\pi V' \to [0,+\infty)$ are Borel. In this way, $V'$ is a Borel field of Polish spaces.
\item Whenever $\vphi_n : X \to V$ is a dense sequence of Borel sections, the $\sigma$-algebra on $V'$ defined in 1 is the smallest $\sigma$-algebra such that the maps $\pi : V' \to X$ and $V' \to [0,+\infty) : v \mapsto d'(v,\gamma(\vphi_n(\pi(v))))$ are measurable.
\item Let $Y$ be a measurable space and $W = (W_x)_{x \in X}$ a Borel field of Polish spaces. Assume that $\pi : Y \to X$ is a measurable map and $C : Y \times_\pi V' \to W$ is any map such that $\pi \circ C = \pi$ and such that restrictions $C_y : V'_{\pi(y)} \to W_{\pi(y)}$ are continuous for all $y \in Y$. Then $C$ is measurable if and only if the map $Y \times_\pi V \to W : (y,v) \mapsto C(y,\gamma(v))$ is measurable.
\end{enumlist}
\end{proposition}
\begin{proof}
Fix a dense sequence of Borel sections $\vphi_n : X \to V$. Define the Borel set
\[
Z \subset X \times \R^\N : Z = \bigl\{(x,z) \bigm| \forall n \in \N , \exists k \in \N , \forall i \in \N : |z_i - d(\vphi_k(x),\vphi_i(x))| < 1/n \bigr\} \; .
\]
Write $\gamma = (\gamma_x)_{x \in X}$. Then $\theta : V' \to Z : \theta(v) = (\pi(v),(d'(v,\gamma(\vphi_i(\pi(v))))_{i \in \N}))$ is a bijection. We define the standard Borel structure on $V'$ such that $\theta$ is a Borel isomorphism. By construction, $\pi : V' \to X$ and the map $\gamma : V \to V'$ are Borel. Since for all $(v,w) \in V' \times_\pi V'$,
\[
d'(v,w) = \inf_{n \in \N} \bigl(d'(v,\gamma(\vphi_n(\pi(v)))) + d'(w,\gamma(\vphi_n(\pi(w))))\bigr)
\]
also $d'$ is Borel. Also, $\gamma \circ \vphi_n$ is a dense sequence of Borel sections. So, $V'$ is a Borel field of Polish spaces.

Given any dense sequence of Borel sections, the corresponding map $V' \to X \times \R^\N$ is Borel and injective, so that the second statement follows. In this way, also the uniqueness statement about the standard Borel structure on $V'$ is proven. So, 1 and 2 are proven.

3.\ If $C$ is measurable, also the map $Y \times_\pi V \to W : (y,v) \mapsto C(y,\gamma(v))$ is measurable. Conversely, if this map is measurable and if $\vphi_n : X \to V$ is a dense sequence of Borel sections, also the maps $Y \to W : y \mapsto C(y,\gamma(\vphi_n(\pi(y))))$ are measurable. Since $\gamma \circ \vphi_n$ is a dense sequence of Borel sections for $V'$, it follows from the second part of Lemma \ref{lem.construction-uniqueness} that $C$ is measurable.
\end{proof}

\begin{definition}\label{def.Borel-field-normed-spaces}
A \emph{Borel field of separable seminormed spaces} consists of a standard Borel space $X$, a family $V = (V_x)_{x \in X}$ of seminormed spaces and a standard Borel structure on $V$ such that the maps
\begin{align*}
& \pi : V \to X \;\; , \quad V \times_\pi V \to V : (v,w) \mapsto v+w \;\; ,\\
& \C \times V \to V : (\lambda,v) \mapsto \lambda v \quad\text{and}\quad V \to [0,+\infty) : v \mapsto \|v\|
\end{align*}
are Borel and such that there exists a dense sequence of Borel sections $\vphi_n : X \to V$.
\end{definition}

The following result is an immediate corollary of Proposition \ref{prop.completion}.

\begin{corollary}\label{cor.completion-normed}
Let $X$ be a standard Borel space and $V = (V_x)_{x \in X}$ a Borel field of separable seminormed spaces. Assume that $V' = (V'_x)_{x \in X}$ is any family of separation-completions of $V_x$ with corresponding maps $\gamma_x : V_x \to V'_x$.

With the standard Borel structure on $V'$ given by Proposition \ref{prop.completion}, we get that $V'$ is a Borel field of separable Banach spaces. This standard Borel structure on $V'$ is the unique standard Borel structure on $V'$ such that the map $\pi : V' \to X$ and the map $\gamma : V \to V'$ are Borel and such that $V'$ is a Borel field of separable Banach spaces.
\end{corollary}
\begin{proof}
Denote $d' : V' \times_\pi V' \to [0,+\infty) : d'(v,w) = \|v-w\|$. By Proposition \ref{prop.completion}, we equip $V'$ with the unique standard Borel structure such that the map $\pi : V' \to X$, the map $\gamma : V \to V'$ and the metric $d'$ are Borel. In this way, $V'$ is a Borel field of Polish spaces.

Since $X \to V : x \mapsto 0_x$ is Borel, also $V' \to [0,+\infty) : v \mapsto \|v\| = d'(v,0_{\pi(v)})=d'(v,\gamma(0_{\pi(v)}))$ is Borel. Denote by
\[
A : V' \times_\pi V' \to V' : A(v,w) = v+w \quad\text{and}\quad M : \C \times V' \to V' : (\lambda,v) \mapsto \lambda v
\]
the addition and scalar multiplication maps. Since the map $A : V \times_\pi V \to V$ is Borel and since, for every fixed $v \in V$, addition of $\gamma(v)$ is a continuous map from $V'_{\pi(v)}$ to $V'_{\pi(v)}$, it follows from Proposition \ref{prop.completion} that the map $V \times_\pi V' \to V': (v,w) \mapsto \gamma(v)+w$ is Borel. Then, for every fixed $w \in V'$, also addition of $w$ is a continuous map from $V'_{\pi(w)}$ to $V'_{\pi(w)}$. By a second application of Proposition \ref{prop.completion}, it follows that $A$ is Borel. A similar argument gives us that $M$ is Borel. So, $V'$ is a Borel field of separable Banach spaces.
\end{proof}

\section{Borel fields of bounded operators}\label{sec.fields-operators}

For the following proposition, recall that a separable Banach space $V$ has the \emph{metric approximation property} if there exists a sequence of finite rank operators $T_n \in B(V)$ such that $\|T_n\| \leq 1$ for all $n \in \N$ and $T_n \to \id$ strongly.

\begin{proposition}\label{prop.bounded-operators-banach}
Let $X$ be a standard Borel space. Let $V= (V_x)_{x \in X}$ and $W = (W_x)_{x \in X}$ be Borel fields of separable Banach spaces. Denote by $B(V,W) = (B(V_x,W_x))_{x \in X}$ the field of bounded linear operators from $V_x$ to $W_x$.
\begin{enumlist}
\item There is a unique standard Borel structure on $B(V,W)$ such that the maps $\pi : B(V,W) \to X$ and $B(V,W) \times_\pi V \to W : (T,v) \mapsto T(v)$ are Borel.
\item The operator norm, addition and scalar multiplication of bounded operators all define Borel maps. If also $U = (U_x)_{x \in X}$ is a Borel field of separable Banach spaces, composition $B(W,U) \times_\pi B(V,W) \to B(V,U) : (S,T) \mapsto S \circ T$ is a Borel map.
\item For every $\sigma$-finite Borel measure $\mu$ on $X$, we have that $\ball B(V,W)$, with the strong topology on each $\ball B(V_x,W_x)$, is a measured field of Polish spaces over $(X,\mu)$.
\item In the following two cases, $\ball B(V,W)$ is a Borel field of the Polish spaces $\ball B(V_x,W_x)$ with the strong operator topology:
\begin{itemlist}
\item when $W_x$ is reflexive for every $x \in X$~;
\item when for every $x \in X$, $V_x$ or $W_x$ has the metric approximation property.
\end{itemlist}
\item Whenever $\vphi_n : X \to V$ a total sequence of Borel sections, the Borel $\sigma$-algebra on $B(V,W)$ is the smallest $\sigma$-algebra that makes the maps $\pi : B(V,W) \to X$ and $B(V,W) \to W : T \mapsto T(\vphi_n(\pi(T)))$ measurable.
\end{enumlist}
\end{proposition}

We suspect that in full generality, $\ball B(V,W)$ may fail to be a Borel field of Polish spaces, but we could not prove this. We however present a related counterexample in Proposition \ref{prop.isom-group-Polish-group}, where we prove that the field $(\Isom V_x)_{x \in X}$ of isometric automorphisms of $V_x$ need not be a Borel field of Polish groups.

\begin{proof}
Fix dense sequences of Borel sections $\vphi_n : X \to V$ and $\psi_n : X \to W$. Define $\cW \subset W^\N$ by $\cW = \{(w_n)_{n \in \N} \mid \forall n,m \in \N : \pi(w_n) = \pi(w_m)\}$. Then $\cW$ is a Borel subset of the standard Borel space $W^\N$. As before, write $\cQ = \Q + i \Q$. Then define the Borel set $Z \subset X \times \cW$ by
\begin{multline*}
Z = \{(x,w) \in \cW \mid \exists k \in \N, \forall \al_1,\al_2,\al_3 \in \cQ, \forall n_1,n_2,n_3 \in \N : \\
\|\al_1 w_{n_1} - \al_2 w_{n_2} - \al_3 w_{n_3}\| \leq k \, \|\al_1 \vphi_{n_1}(x) - \al_2 \vphi_{n_2}(x) - \al_3 \vphi_{n_3}(x)\| \} \; .
\end{multline*}
Note that $\Theta : B(V,W) \to Z : T \mapsto (\pi(T),(T(\vphi_n(\pi(T))))_{n \in \N})$ is a bijection. By the first part of Lemma \ref{lem.construction-uniqueness}, there is a unique standard Borel structure on $B(V,W)$ such that the maps $\pi : B(V,W) \to X$ and $B(V,W) \times_\pi V \to W : (T,v) \mapsto T(v)$ are Borel. Also, $\Theta$ is a Borel isomorphism.

Fix a dense sequence of Borel sections $\vphi'_n : X \to \ball V$. Then $\|T\| = \sup_n \|T(\vphi'_n(\pi(T)))\|$ and it follows that the norm is Borel. Addition and scalar multiplication are Borel, because they are well-defined Borel maps on $\cW$ and then transfer to $\cB$ and $B(V,W)$.

To prove that the composition map $B(W,U) \times_\pi B(V,W) \to B(V,U)$ is Borel, it follows from the second part of Lemma \ref{lem.construction-uniqueness} that it suffices to prove that for every Borel section $T : X \to B(V,W)$, the map
\[
B(W,U) \to B(V,U) : S \mapsto S \circ T(\pi(S))
\]
is Borel. By statement 1.c) in Lemma \ref{lem.construction-uniqueness}, it then suffices to prove that for every Borel section $\vphi : X \to V$, the map
\[
B(W,U) \to U : S \mapsto (S \circ T(\pi(S)))(\vphi(\pi(S))) = S(T(\pi(S)) \vphi(\pi(S)))
\]
is Borel. But this holds true because $X \to W : x \mapsto T(x) \vphi(x)$ is a Borel section.

Note that $d : \ball B(V,W) \times_\pi B(V,W) \to [0,\infty) : d(T,S) = \sum_{n \in \N} 2^{-n} \|T(\vphi'_n(\pi(T))) - S(\vphi'_n(\pi(S)))\|$ is a Borel map whose restriction to $\ball B(V_x,W_x) \times \ball B(V_x,W_x)$ is a complete separable metric inducing the strong topology. For every $r,n,k \in \N, m \in \N^k$, define the Borel set $U(r,k,n,m) \in \ball B(V,W)$ by
\[
U(r,n,k,m) = \{T \in \ball B(V,W) \mid \forall i = 1,\ldots, k : \|T(\vphi_{i}(\pi(T))) - \psi_{n_i}(\pi(T))\| < 1/r \} \; .
\]
For every $x \in X$, the countable family of sets $\ball B(V_x,W_x) \cap U(r,n,k,m)$ forms a basis for the strong topology. So 3 follows from Proposition \ref{prop.subfield}.

To prove 4, first assume that every $W_x$ is reflexive. By point~5 of Proposition \ref{prop.dual-field-Banach}, we can choose a norm dense sequence of Borel sections $\om_n : X \to W^*$. Also choose a dense sequence of Borel sections $\vphi_k : X \to V$. Then the map
\[
\theta : \ball B(V,W) \to X \times \C^{\N \times \N} : \theta(T) = \bigl(\pi(T),\bigl(\langle T(\vphi_k(\pi(T))),\om_n(\pi(T))\rangle\bigr)_{(k,n) \in \N \times \N}\bigr)
\]
is injective and Borel. So $Z = \theta(\ball B(V,W))$ is a Borel set. Since every $W_x$ is reflexive, we can identify $B(V_x,W_x)$ with the set of bilinear maps $V_x \times W_x^* \to \C$ of norm at most one. It follows that
\begin{align*}
 Z = &\bigcap_{k_1,k_2,k_3,n \in \N, \al_1,\al_2,\al_3 \in \C} \bigl\{ (x,z) \in X \times \C^{\N \times \N} \bigm| |\al_1 z_{k_1,n} - \al_2 z_{k_2,n} - \al_3 z_{k_3,n}| \\ & \hspace{6cm} \leq \|\al_1 \vphi_{k_1}(x) - \al_2 \vphi_{k_2}(x) - \al_3 \vphi_{k_3}(x)\| \, \|\om_n(x)\| \bigr\} \\
 & \bigcap \bigcap_{k,n_1,n_2,n_3, \be_1,\be_2,\be_3 \in \C} \bigl\{ (x,z) \in X \times \C^{\N \times \N} \bigm| |\be_1 z_{k,n_1} - \be_2 z_{k,n_2} - \be_3 z_{k,n_3}| \\ & \hspace{6cm} \leq \|\vphi_k(x)\| \, \|\be_1 \om_{n_1}(x) - \be_2 \om_{n_2}(x) - \be_3 \om_{n_3}(x)\|\bigr\} \; .
\end{align*}
It follows from this description that for every $x \in X$, the subset $Z_x \subset \C^{\N \times \N}$ is compact, where $\C^{\N \times \N}$ is equipped with the topology of pointwise convergence. So by Theorem \ref{thm.selection-sigma-compact}, there exists a dense sequence of Borel sections $X \to \C^{\N \times \N}$ for the set $Z$. Composing with $\theta^{-1}$, we find a sequence of Borel sections $S_n : X \to \ball B(V,W)$ with the property that for every $x \in X$, the set $\{S_n(x) \mid n \in \N\}$ is dense in $\ball B(V,W)$ with the topology of pointwise weak convergence. Taking finite convex combinations of the $S_n$ with coefficients in $\Q \cap [0,1]$, we find a sequence of Borel sections $T_n : X \to \ball B(V,W)$ that is strongly dense. So, it follows that $\ball B(V,W)$ is a Borel field of Polish spaces.

Next assume that for every $x \in X$, at least one of the spaces $V_x$, $W_x$ has the metric approximation property. Above, we already fixed a dense sequence of Borel sections $\vphi_k : X \to V$. Also fix a dense sequence of Borel sections $\psi_n : X \to W$. Fix $n \in \N$. For every $x \in X$, define the closed subspace
\[
E_{n,x} \subset W_x : E_{n,x} = \lspan \{\psi_i(x) \mid 1 \leq i \leq n\} \; .
\]
It follows from Proposition \ref{prop.selection-Effros} that $E_n = (E_{n,x})_{x \in X}$ is a Borel field of finite-dimensional Banach spaces. By the preceding paragraphs, it follows that $\ball B(V,E_n)$ is a Borel field of Polish spaces. We can thus choose a sequence of Borel sections $T_{n,k} : X \to \ball B(V,E_n)$ such that for every $x \in X$ and $n \in \N$, the set $\{T_{n,k}(x) \mid k \in \N\}$ is strongly dense in $\ball B(V,E_n)$.

We conclude the proof of 4 by showing that for every $x \in X$, the set $A_x = \{T_{n,k}(x) \mid n,k \in \N\}$ is strongly dense in $\ball B(V,W)$. Fix $x \in X$ and $T \in \ball B(V_x,W_x)$. Since $V_x$ or $W_x$ has the metric approximation property, there exists a sequence of finite rank operators $S_n \in \ball B(V_x,W_x)$ such that $S_n \to T$ strongly. It thus suffices to prove that every finite rank operator $S \in \ball B(V_x,W_x)$ belongs to the strong closure of $A_x$. For that it suffices to prove that every finite rank operator $S \in B(V_x,W_x)$ with $\|S\| < 1$ belongs to the operator norm closure of $\bigcup_{n \in \N} \ball B(V_x,E_{n,x})$.

Take $\eps > 0$ arbitrary. Put $\delta = \min \{\eps,1-\|S\|\}$. Take $r \in \N$ and for all $i \in \{1,\ldots,r\}$, take $\eta_i \in V_x^*$ and $w_i \in W_x$ such that $S(v) = \sum_{i=1}^r \eta_i(v) \, w_i$ for all $v \in V_x$. Write $\kappa = 1+\sum_{i=1}^r \|\eta_i\|$. Take $n_1,\ldots,n_r \in \N$ such that $\|\psi_{n_i}(x) - w_i\| < \kappa^{-1} \delta$ for all $i = 1,\ldots,r$. Write $n = \max\{n_1,\ldots,n_r\}$ and define $S_0 \in B(V_x,E_{n,x})$ by $S_0(v) = \sum_{i=1}^r \eta_i(v) \, \psi_{n_i}(x)$. By construction, $\|S-S_0\| < \delta$. So $S_0 \in \ball B(V_x,E_{n,x})$ and $\|S-S_0\| < \eps$.

Point 5 follows directly from the first part of Lemma \ref{lem.construction-uniqueness}.
\end{proof}

Every C$^*$-algebra $A$ sits canonically as an ideal in its multiplier C$^*$-algebra $M(A)$, see \cite[Theorem 2.11]{Bus67}. Recall from \cite[Definition 3.4]{Bus67} that the strict topology on $M(A)$ is the weakest topology on $M(A)$ such that the maps $M(A) \to A$ given by $a \mapsto ab$ and $a \mapsto ba$ are continuous for all $b \in A$. As a first application of Proposition \ref{prop.bounded-operators-banach}, we prove that the field of multipliers $M(A_x)_{x \in X}$ of a Borel field of separable C$^*$-algebras has a natural standard Borel structure.

\begin{proposition}\label{prop.multiplier-field}
Let $A = (A_x)_{x \in X}$ be a Borel field of separable C$^*$-algebras and consider their multiplier algebras $M(A) = M(A_x)_{x \in X}$. There is a unique standard Borel structure on $M(A)$ such that the maps $\pi : M(A) \to X$ and $M(A) \times_\pi A \to A : (a,b) \mapsto ab$ are Borel.

Then the inclusion map $A \to M(A)$ and the multiplication $A \times_\pi M(A) \to A$ are Borel. The restriction of the Borel structure on $M(A)$ to $\ball M(A)$ turns $\ball M(A)$ into a Borel field of the Polish spaces $\ball M(A)$ with the strict topology.
\end{proposition}
\begin{proof}
Denote by $B(A,A) = B(A_x,A_x)_{x \in X}$ the spaces of bounded linear maps from $A_x$ to $A_x$. We equip $B(A,A)$ with the standard Borel structure given by Proposition \ref{prop.bounded-operators-banach}. Choose a dense sequence of Borel sections $\vphi_n : X \to A$. We then define the Borel subset $Z \subset X \times B(A,A) \times B(A,A)$ consisting of the elements $(x,L,R)$ satisfying
$$\pi(L) = \pi(R) = x \quad\text{and}\quad \vphi_n(x) \cdot L(\vphi_k(x)) = R(\vphi_n(x)) \cdot \vphi_k(x) \quad\text{for all $n,k \in \N$.}$$
For every $a \in M(A_x)$, we denote by $L_a$ and $R_a$ the left and right multiplication operators on $A_x$. Then, $M(A) \to Z : a \mapsto (\pi(a),L_a,R_a)$ is a bijection. We define the standard Borel structure on $M(A)$ such that this map is a Borel isomorphism.

By construction, the maps $\pi : M(A) \to X$, $A \to M(A)$ and the multiplication maps $M(A) \times_\pi A \to A$ and $A \times_\pi M(A) \to A$ are Borel. Since the map $M(A) \to B(A,A) : a \mapsto L_a$ is injective and Borel, the Borel $\sigma$-algebra on $M(A)$ equals the smallest $\sigma$-algebra such that $M(A) \times_\pi A \to A$ is Borel. This proves the uniqueness statement.

Choose a dense sequence of Borel sections $\psi_n : X \to \ball A$. Then the map
\begin{multline*}
d : \ball M(A) \times_\pi \ball M(A) \to [0,+\infty) : \\ (a,b) \mapsto \sum_{n=1}^\infty 2^{-n} \, \bigl(\|(a-b) \cdot \psi_n(\pi(a))\| + \|\psi_n(\pi(a)) \cdot (a-b)\|\bigr)
\end{multline*}
is Borel and its restriction to $\ball M(A_x) \times \ball M(A_x)$ is a separable complete metric that is compatible with the strict topology. By \cite[Proposition 3.5]{Bus67}, the unit ball of $A$ is strictly dense in the unit ball of $M(A)$. So, $\psi_n : X \to \ball A$ automatically is a strictly dense sequence of Borel sections $X \to \ball M(A)$. It follows that $\ball M(A)$ is a Borel field of Polish spaces.
\end{proof}

As mentioned above, we do not expect that in Proposition \ref{prop.bounded-operators-banach}, $\ball B(V,W)$ will always be a Borel field of Polish spaces. We illustrate this by the following combination of a positive result and a counterexample.

For a separable Banach space $V$ we denote by $\Isom V$ the group of bijective, isometric, linear maps $\al : V \to V$. Together with the strong operator topology, $\Isom V$ is a Polish group. Note that $\Isom V \subset \ball B(V)$.

Given a field $V = (V_x)_{x \in X}$ of separable Banach spaces, we want to consider the field $\Isom V = (\Isom V_x)_{x \in X}$. So we first need the following obvious definition. Note that a different definition was given in \cite{Sut85}, but from the proof of Proposition \ref{prop.universal-field-Polish-groups} below, it follows that both definitions are equivalent.

\begin{definition}\label{def.Borel-field-Polish-group}
A \emph{Borel field of Polish groups} consists of a standard Borel space $X$, a family $G = (G_x)_{x \in X}$ of Polish groups and a standard Borel structure on $G$ such that $G$ is a Borel field of Polish spaces and the maps
\[
X \to G : x \mapsto e_x \;\; , \;\; G \to G : g \mapsto g^{-1} \;\;\text{and}\;\; G \times_\pi G \to G : (g,h) \mapsto gh
\]
are Borel.

A \emph{measured field of Polish groups} consists of a standard $\sigma$-finite measured space $(X,\mu)$, a family $G = (G_x)_{x \in X}$ of Polish groups and a standard Borel structure on $G$ such that $\pi : G \to X$ is Borel and such that there exists a conull Borel set $X_0 \subset X$ with $(G_x)_{x \in X_0}$ being a Borel field of Polish groups.
\end{definition}

\begin{proposition}\label{prop.isom-group-Polish-group}
Let $V = (V_x)_{x \in X}$ be a Borel field of separable Banach spaces. Consider the standard Borel structure on $B(V) = B(V,V)$ given by Proposition \ref{prop.bounded-operators-banach}.
\begin{enumlist}
\item $\Isom V$ is a Borel subset of $B(V)$. The restricted Borel structure on $\Isom V$ is the unique standard Borel structure on $\Isom V$ such that the maps $\pi : \Isom V \to X$ and $\Isom V \times_\pi V \to V : (\al,v) \mapsto \al(v)$ are Borel. Then composition and inverse are Borel maps.
\item For every $\sigma$-finite measure on $X$, $\Isom V$ is a measured field of Polish groups.
\item There are examples in which $\Isom V$ is not a Borel field of Polish groups, because there is no dense sequence of Borel sections $X \to \Isom V$.
\end{enumlist}
\end{proposition}
\begin{proof}
1.\ Define the Borel set $L \subset \ball B(V) \times_\pi \ball B(V)$ by $L = \{(T,S) \in \ball B(V) \times_\pi \ball B(V) \mid T \circ S = \id = S \circ T\}$. Note that the map $L \to \ball B(V) : (T,S) \mapsto T$ is Borel, injective and has image equal to $\Isom V$. So, $\Isom V$ is a Borel subset of $B(V)$. The inverse operation on $\Isom V$ corresponds to the Borel map $(T,S) \mapsto (S,T)$ on $L$ and is therefore Borel. The remaining part of statement 1 follows from Lemma \ref{lem.construction-uniqueness}.

2.\ Let $\mu$ be a $\sigma$-finite measure on $X$. Note that for every $x \in X$, $L_x$ is a strongly closed subset of $\ball B(V_x) \times \ball B(V_x)$. By Proposition \ref{prop.subfield}, we find a conull Borel set $X_0 \subset X$ and a dense sequence of Borel sections $X_0 \to L : x \mapsto (\al_n(x),\be_n(x))$. Then $\al_n : X_0 \to \Isom V$ is a dense sequence of Borel sections for $(\Isom V_x)_{x \in X_0}$. So, statement 2 is proven.

3.\ We make use of the main result of \cite{CG15} saying that the equivalence relation of being homeomorphic on the space of connected compact metric spaces is not Borel. By \cite[Theorem 1.1]{CG15}, we can take a standard Borel space $Y$ and a Borel field $K = (K_y)_{y \in Y}$ of Polish spaces such that every $K_y$ is compact and connected and such that the set
\begin{equation}\label{eq.equiv-rel-homeo}
R = \{(y,z) \in Y \times Y \mid \;\text{$K_y$ is homeomorphic to $K_z$}\;\}
\end{equation}
is not Borel. Define $X = Y \times Y$ and for every $x =(y,z) \in X$, define $L_x$ as the disjoint union of $K_y$ and $K_z$, so that $C(L_x) = C(K_y) \oplus_\text{max} C(K_z)$. Consider the Borel field $C(L_x)_{x \in X}$ of separable Banach spaces given by Proposition \ref{prop.Borel-field-CK}. We claim that $\Isom C(L)$ is not a Borel field of Polish groups.

Assume that $\Isom C(L)$ is a Borel field of Polish groups. We will deduce that the set $R$ in \eqref{eq.equiv-rel-homeo} is Borel, which is not the case. Choose a dense sequence of Borel sections $\al_n : X \to \Isom C(L)$. Every element $\al \in \Isom C(L_x)$ is in particular a bounded operator from $C(L_x)$ to itself. Since $C(L_x) = C(K_y) \oplus_\text{max} C(K_z)$, every such bounded operator $\al$ has four matrix coefficients, with one of these coefficients being a bounded operator from $C(K_y)$ to $C(K_z)$. For $\al_n(x) \in \Isom(C(L_x))$, we denote this component as $A_n(x) \in B(C(K_y),C(K_z))$. Then, $A_n : X \to B(C(K),C(K))$ is a Borel section.

Fix a dense sequence of Borel sections $F_n : Y \to C(K)$ and a weak$^*$ dense sequence of Borel sections $\om_n : Y \to \ball C(K)^*$. For $x=(y,z)$, write $F'_n(x) = F_n(y)$ and $\om'_n(x) = \om_n(z)$. Then,
\[
S = \{x \in X \mid \forall n,i,j \in \N : \langle A_n(x) (F'_i(x)), \om'_j(x) \rangle = 0\}
\]
is a Borel set. By construction and by the density of $\al_n : X \to \Isom C(L)$,
\begin{align*}
S &= \{x \in X \mid \forall n \in \N : \al_n(x)(C(K_y) \oplus \{0\}) \subset C(K_y) \oplus \{0\}\} \\
&= \{x \in X \mid \forall \al \in \Isom C(L_x) : \al(C(K_y) \oplus \{0\}) \subset C(K_y) \oplus \{0\}\} \; .
\end{align*}
By the Banach-Stone theorem, every $\al \in \Isom C(L_x)$ is induced by a homeomorphism of $L_x$ and a multiplication with a continuous function $L_x \to \T$. Since $L_x$ is the disjoint union of the connected subsets $K_y$ and $K_z$, every homeomorphism of $L_x$ will either preserve $K_y$, $K_z$, or interchange $K_y$, $K_z$, where the latter can only happen if $K_y$ and $K_z$ are homeomorphic. It follows from this discussion that
\[
S = \{x \in X \mid \;\text{$K_y$ is not homeomorphic to $K_z$}\;\} = X \setminus R \; .
\]
So also $R$ is a Borel set, which is absurd.
\end{proof}

By the proof of the Birkhoff-Kakutani theorem (see e.g.\ \cite[Theorem 1.5.2]{Tao14}), every Polish group $G$ admits a \emph{left invariant} metric $d$ that is compatible with the topology. Note however that it is not always possible to choose a left invariant complete metric that is compatible with the topology. The usual proof of the Birkhoff-Kakutani theorem is sufficiently constructive to get the following.

\begin{proposition}\label{prop.left-invariant-metric}
Let $X$ be a standard Borel space and $G = (G_x)_{x \in X}$ a Borel field of Polish groups. There exists a Borel map $D : G \times_\pi G \to [0,+\infty)$ such that for every $x \in X$, the restriction of $D$ to $G_x \times G_x$ is a left invariant metric on $G_x$ that induces the topology of $G_x$.
\end{proposition}
\begin{proof}
We closely follow the proof of \cite[Lemma 1.5.4]{Tao14}. Fix a Borel map $d : G \times_\pi G \to [0,+\infty)$ such that for every $x \in X$, the restriction of $d$ to $G_x \times G_x$ is a separable complete metric on $G_x$ that induces the topology of $G_x$. Fix a dense sequence of Borel sections $\vphi_k : X \to G$. For every $x \in X$ and $r > 0$, define
\begin{align*}
A_x(r) &= \{g \in G_x \mid d(g,e_x) < r \;\;\text{and}\;\; d(g^{-1},e_x) < r\} \quad\text{and}\\
B_x(r) &= \{g \in G_x \mid d(g,e_x) \leq r \;\;\text{and}\;\; d(g^{-1},e_x) \leq r\} \; .
\end{align*}
We inductively define Borel functions $\rho_n : X \to (0,1]$ such that $\rho_n(x) \to 0$ for every $x \in X$ and $A_x(\rho_n(x))^2 \subset A_x(\rho_{n-1}(x))$ for all $x \in X$ and $n \geq 2$. Put $\rho_1(x) = 1$ for all $x \in X$. Then define for all $n \geq 2$,
$$\rho_n(x) = \sup\bigl\{r \in (0,1/n] \bigm| A_x(r)^2 \subset B_x(\rho_{n-1}(x)/2) \bigr\} \; .$$
Because multiplication and inverse are continuous in $G_x$, we have that $\rho_n(x) > 0$ for all $x \in X$. To see that $\rho_n$ is Borel, note that given $s \in (0,1/n)$, because $(\vphi_k(x))_k$ is dense in $G_x$, we have
\begin{align*}
\rho_n(x) > s &\Longleftrightarrow \exists r \in (s,1/n) \cap \Q : A_x(r)^2 \subset  B_x(\rho_{n-1}(x)/2)\\
& \Longleftrightarrow \exists r \in (s,1/n) \cap \Q , \forall k,l \in \N : \;\text{if}\; d(\vphi_k(x)^{\pm 1},e_x) < r \;\text{and}\; d(\vphi_l(x)^{\pm 1},e_x) < r \; ,\\
&\hspace{7.5cm}\text{then} \; d(\vphi_k(x) \vphi_l(x),e_x) \leq \rho_{n-1}(x)/2 \; .
\end{align*}

We can now literally repeat the argument of \cite[Lemma 1.5.4]{Tao14}. Denote by $\cD$ the set of dyadic rationals $0 < q < 1$. Every $q \in \cD$ can be uniquely written as $q = 2^{-n_1} + \cdots + 2^{-n_k}$ with integers $k \geq 1$ and $n_1 > n_2 > \cdots > n_k \geq 1$. We define for every $x \in X$, the subset $V_x(q) \subset G_x$ as the closure of $A_x(\rho_{n_1}(x)) \cdots A_x(\rho_{n_k}(x))$. Since $V_x(q)$ also is the closure of
$$\bigl\{ \vphi_{i_1}(x) \cdots \vphi_{i_k}(x) \bigm| \forall j \in \{1,\ldots,k\} : d(\vphi_{i_j}(x)^{\pm 1}, e_x) < \rho_{n_j}(x) \bigr\} \; ,$$
it follows from Proposition \ref{prop.selection-Effros} that $V(q) = \{g \in G \mid g \in V_{\pi(g)}(q)\}$ is a Borel subset of $G$ for every $q \in \cD$. As in \cite[Lemma 1.5.4]{Tao14}, we define
$$\psi : G \to [0,1] : \psi(g) = \sup \bigl\{ 1- q \bigm| q \in \cD \; , \; g \in V(q) \bigr\} \; ,$$
with the convention that $\psi(g) = 0$ if $g$ does not belong to any $V(q)$. Since all the sets $V(q)$ are Borel, $\psi$ is a Borel map. As in \cite[Definition 1.5.1]{Tao14}, we can define the group norm
\begin{align*}
N : G \to [0,1] : N(g) &= \sup \bigl\{ |\psi(gh) - \psi(h)| \bigm| h \in G_{\pi(g)} \bigr\} \\
&= \sup \bigl\{ |\psi(g\vphi_k(\pi(g))) - \psi(\vphi_k(\pi(g)))| \bigm| k \in \N \bigr\} \; ,
\end{align*}
which is Borel. By \cite[Exercises 1.5.1 and 1.5.2]{Tao14}, the map $D(g,h) = N(h^{-1}g)$ provides the required left invariant metric.
\end{proof}

\section{Borel fields of von Neumann algebras: abstract vs.\ concrete}\label{sec.fields-vNalg}

In the same way as von Neumann algebras can be defined abstractly as well as represented on a given Hilbert space, the same holds for Borel fields of von Neumann algebras with separable predual. We give both definitions and prove that they are equivalent. We call \emph{separable von Neumann algebra} any von Neumann algebra with a separable predual.

\begin{definition}\label{def.Borel-field-vNalg-abstract}
An abstract Borel field of separable von Neumann algebras consists of a standard Borel space $X$, a family $M = (M^x)_{x \in X}$ of von Neumann algebras with separable predual and a standard Borel structure on $M_* = (M^x_*)_{x \in X}$ such that the following holds.
\begin{enumlist}
\item $M_*$ is a Borel field of separable Banach spaces.
\item With the standard Borel structure on $M = (M_*)^*$ given by Proposition \ref{prop.dual-field-Banach}, we have that the maps
\[
1 : X \to M : x \mapsto 1_x \quad , \quad M \to M : a \mapsto a^* \quad , \quad M \times_\pi M \to M : (a,b) \mapsto ab
\]
are Borel.
\end{enumlist}
\end{definition}

In the concrete definition of a Borel field of separable von Neumann algebras, the von Neumann algebras $M^x$ are represented on Hilbert spaces $H_x$. So we first need the following definition.

\begin{definition}\label{def.Borel-field-Hilbert-spaces}
A Borel field of separable Hilbert spaces consists of a standard Borel space $X$, a family $H = (H_x)_{x \in X}$ of separable Hilbert spaces and a standard Borel structure on $H$ such that $H$ becomes a Borel field of separable Banach spaces.
\end{definition}

Note that by the polarization equality, automatically the scalar product map $H \times_\pi H \to \C : (\xi,\eta) \mapsto \langle \xi,\eta \rangle$ is Borel.

Given a Borel field $H = (H_x)_{x \in X}$ of separable Hilbert spaces, we equip $B(H) = B(H_x)_{x \in X}$ with the standard Borel structure given by Proposition \ref{prop.bounded-operators-banach}.

\begin{definition}\label{def.Borel-field-vNalg-concrete}
A concrete Borel field of separable von Neumann algebras consists of a standard Borel space $X$, a Borel field $H = (H_x)_{x \in X}$ of separable Hilbert spaces and a family of von Neumann algebras $M = (M^x)_{x \in X}$ acting on $H_x$ such that $M \subset B(H)$ is a Borel set.
\end{definition}

In \cite[Definition IV.8.17]{Tak79}, a slightly different definition of a concrete measured field of separable von Neumann algebras is given. As with all our definitions of Borel fields, also Definition \ref{def.Borel-field-vNalg-concrete} has the obvious variant defining concrete measured fields of separable von Neumann algebras. By Proposition \ref{prop.construction}, such a variant is equivalent with \cite[Definition IV.8.17]{Tak79}.

Below we prove the following result.

\begin{proposition}\label{prop.equivalence-defs-Borel-field-vNalg}
Let $X$ be a standard Borel space.
\begin{enumlist}
\item If $M = (M^x)_{x \in X}$ is an abstract Borel field of separable von Neumann algebras, there exists a Borel field $H = (H_x)_{x \in X}$ of separable Hilbert spaces and a family of faithful, normal, unital $*$-homomorphisms $\theta_x : M^x \to B(H_x)$ such that the corresponding map $\theta : M \to B(H)$ is Borel and $(\theta_x(M^x) \subset B(H_x))_{x \in X}$ is a concrete Borel field of separable von Neumann algebras.
\item If $(M^x \subset B(H_x))_{x \in X}$ is a concrete Borel field of separable von Neumann algebras, there is a unique standard Borel structure on $M_* = (M^x_*)_{x \in X}$ such that the map $M_* \times_\pi M \to \C$ is Borel. Then $M_*$ is a Borel field of separable Banach spaces and $M$ is an abstract Borel field of separable von Neumann algebras.
\end{enumlist}
\end{proposition}

To prove the first statement in Proposition \ref{prop.equivalence-defs-Borel-field-vNalg}, we need to use the GNS construction. In Definition \ref{def.Borel-field-normed-spaces}, we defined the notion of a Borel field of separable seminormed spaces and then constructed in Corollary \ref{cor.completion-normed} its separation-completion as a Borel field of separable Banach spaces. Entirely similarly, we can define the concept of a Borel field of separable vector spaces with a positive semidefinite Hermitian form and obtain its separation-completion into a Borel field of separable Hilbert spaces.

\begin{proposition}\label{prop.GNS}
Let $M = (M^x)_{x \in X}$ be an abstract Borel field of separable von Neumann algebras. Let $\om : X \to M_*$ be a Borel section such that for every $x \in X$, $\om_x \in M_*^+$. For every $x \in X$, write $H_x = L^2(M^x,\om_x)$. Denote by $\gamma_x : M^x \to H_x$ the canonical map and let $\theta_x : M^x \to B(H_x)$ be the GNS representation.

There is a unique standard Borel structure on $H = (H_x)_{x \in X}$ such that the maps $\pi : H \to X$ and $\gamma : M \to H$ are Borel and such that $H$ is a Borel field of separable Hilbert spaces. Then the map $\theta : M \to B(H)$ given by $\theta = (\theta_x)_{x \in X}$ is Borel.
\end{proposition}

Note that with the Borel section $\xi : X \to H : x \mapsto \gamma_x(1_x)$, we have $\langle \theta_x(a) \xi(x),\xi(x)\rangle = \om_x(a)$ for all $x \in X$ and all $a \in M^x$.

\begin{proof}
For every $x \in X$, the map $\langle \cdot , \cdot \rangle_x : M^x \times M^x \to \C : (a,b) \mapsto \om_x(b^* a)$ is a positive semidefinite Hermitian form and $L^2(M^x,\om_x)$ is the separation-completion of $M^x$ with respect to the associated seminorm $a \mapsto \om_x(a^* a)^{1/2}$. We claim that in this way, $M$ is a Borel field of vector spaces with a positive semidefinite Hermitian form. Since all structure maps on $M$ are Borel, we only have to prove that there exists a dense sequence of Borel sections $X \to M$. By Proposition \ref{prop.dual-field-Banach}, there exists a weak$^*$-dense sequence of Borel sections $\vphi_n : X \to \ball M$. Taking linear combinations with coefficients in $\Q + i\Q$ of finite products of $\vphi_n$ and $\vphi_n(\cdot)^*$, we obtain a countable family of Borel sections $X \to M$ that are dense in the required norm $a \mapsto \om_x(a^* a)^{1/2}$.

By Corollary \ref{cor.completion-normed}, there is a unique standard Borel structure on $H = (H_x)_{x \in X}$ such that the maps $\pi : H \to X$ and $\gamma : M \to H$ are Borel and such that $H$ is a Borel field of separable Hilbert spaces.

We claim that the map $M \times_\pi H \to H : (a,\xi) \mapsto \theta(a) \xi$ is Borel. By Proposition \ref{prop.completion}, it suffices to prove that the map
\[
M \times_\pi M \to H : (a,b) \mapsto \theta_{\pi(a)} \gamma(b) = \gamma(ab)
\]
is Borel, which follows because the multiplication map $M \times_\pi M \to M$ and the map $\gamma : M \to H$ are Borel. So the claim is proven.

It follows from the claim that for every Borel section $\vphi : X \to H$, the map $M \to H : a \mapsto \theta(a)\vphi(\pi(a))$ is Borel. So by Proposition \ref{prop.bounded-operators-banach}, the map $\theta : M \to B(H)$ is Borel.
\end{proof}

\begin{lemma}\label{lem.positive-functionals}
Let $X$ be a standard Borel space and $M = (M^x)_{x \in X}$ an abstract Borel field of separable von Neumann algebras.
\begin{enumlist}
\item If we equip $\ball M^x$ with the weak topology, the strong topology or the strong$^*$ topology, $\ball M$ is a Borel field of Polish spaces.
\item $\ball M^+$ is a Borel subset of $\ball M$ and, with the weak or strong topology, becomes a Borel field of Polish spaces.
\item $M^+_* = ((M^x_*)^+)_{x \in X}$ is a Borel subset of $M_*$ and becomes in this way a Borel field of Polish spaces.
\end{enumlist}
\end{lemma}
\begin{proof}
1.\ The weak topology on $\ball M^x$ is the weak$^*$ topology when viewing $M^x$ as the dual of $M^x_*$. So by Proposition \ref{prop.dual-field-Banach}, with the weak topology, $\ball M$ is a Borel field of Polish spaces.

Choose a dense sequence of Borel sections $\om_n : X \to \ball M_*$. Then the maps
\begin{align*}
& d : \ball M \times_\pi \ball M \to [0,+\infty) : (a,b) \mapsto \sum_n 2^{-n} |\langle \om_n(\pi(a)),(a-b)^*(a-b)\rangle| \; ,\\
& D : \ball M \times_\pi \ball M \to [0,+\infty) : (a,b) \mapsto d(a,b) + d(a^*,b^*)
\end{align*}
are Borel and restrict to separable complete metrics on $\ball M^x$ that induce, respectively, the strong and the strong$^*$ topology.

To conclude the proof of 1, it thus only remains to prove that there exists a strong$^*$ dense sequence of Borel sections $X \to \ball M$. By Proposition \ref{prop.dual-field-Banach}, we can choose a weak$^*$ dense sequence of Borel sections $\vphi_n : X \to \ball M$. Denote by $\cF$ the countable family of Borel sections $X \to M$ that are obtained as finite linear combinations with coefficients in $\Q + i \Q$ of finite products of sections $\vphi_n$ and $\vphi_n(\cdot)^*$. Define the map $F : [0,+\infty) \to [0,1]$ by $F(t) = t$ for all $t \in [0,1]$ and $F(t) = 1/t$ for all $t > 1$. Denote by $\cG$ the countable family of Borel sections $X \to \ball M$ of the form $x \mapsto F(\|\vphi(x)\|) \, \vphi(x)$ with $\vphi \in \cF$. We claim that this family is pointwise strong$^*$ dense.

Fix $x \in X$ and $a \in \ball M^x$. By the Kaplansky density theorem, we can choose a sequence $\psi_n \in \cF$ such that $\|\psi_n(x)\| < 1$ for all $n \in \N$ and $\psi_n(x) \to a$ strongly$^*$. Then the same holds for the corresponding sequence in $\cG$, so that the claim is proven.

2.\ By 1, we can fix a strong$^*$ dense sequence of Borel sections $\psi_n : X \to \ball M$. For every $x \in X$, we have that $\ball (M^x)^+$ is both the weak closure and the strong closure of the set $\{\psi_n(x)^* \psi_n(x) \mid n \in \N\}$. The conclusion the follows from point~2 of Proposition \ref{prop.selection-Effros}.

3.\ Choose a dense sequence of Borel sections $\om_n : X \to M_*$. Then $\om_n + \om_n^*$ provides a dense sequence of Borel sections for the self-adjoint part $M_*^{\text{sa}}$, so that $M_*^{\text{sa}}$ is a Borel field of Polish spaces. To prove 3, by Proposition \ref{prop.selection-Effros}, it suffices to prove that for every Borel section $\om : X \to M_*^{\text{sa}}$, the map $x \mapsto d(\om(x),(M^x_*)^+)$ is Borel.

Below we prove that for every von Neumann algebra $N$ and $\om \in N_*^{\text{sa}}$,
\begin{equation}\label{eq.nice}
d(\om,N_*^+) = \sup \{ -\om(a) \mid a \in \ball N^+ \} \; .
\end{equation}
By 2, we can choose a weak$^*$ dense sequence of Borel sections $a_n : X \to \ball M^+$. So by \eqref{eq.nice},
\[
x \mapsto d(\om(x),(M^x_*)^+) = \sup \{ - \langle \om(x), a_n(x) \rangle \mid n \in \N\}
\]
is indeed Borel, so that 3 is proven.

To prove \eqref{eq.nice}, fix $\om \in N_*^{\text{sa}}$. By \cite[Theorem III.4.2]{Tak79}, we find a projection $p \in N$ such that $p \om = \om p$ and, writing $q = 1-p$, we have
\[
\om = \om_1 - \om_2 \quad\text{where}\quad \om_1,\om_2 \in N_*^+ \;\; , \;\; \om_1 = p \om p \;\;,\;\; \om_2 = - q \om q \; .
\]
For every $a \in \ball N^+$, we have that $-\om(a) = -\om_1(a) + \om_2(a) \leq \om_2(a) \leq \|\om_2\| = - \om(q)$. So, the right hand side of \eqref{eq.nice} equals $\|\om_2\|$.

Since $\|\om_1 - \om\| = \|\om_2\|$, certainly $d(\om,N_*^+) \leq \|\om_2\|$. On the other hand, for every $\mu \in N_*^+$, we have that
\[
\|\mu - \om\| \geq \|q (\mu - \om) q\| = \|q \mu q - q \om q\| = \|q \mu q + \om_2\| = \mu(q) + \|\om_2\| \geq \|\om_2\| \; .
\]
So, $d(\om,N_*^+) = \|\om_2\|$ and \eqref{eq.nice} is proven.
\end{proof}

Given a Borel field $H = (H_x)_{x \in X}$ of separable Hilbert spaces, we equipped $B(H)$ with the standard Borel structure given by Proposition \ref{prop.bounded-operators-banach}. To prove the second statement in Proposition \ref{prop.equivalence-defs-Borel-field-vNalg}, we need to define a standard Borel structure on the family of Banach spaces $(B(H_x)_*)_{x \in X}$ and identify $B(H)$ with $(B(H)_*)^*$. Since for all Borel sections $\vphi : X \to H$ and $\psi : X \to H$, we can use the elements $T \mapsto \langle T \vphi(x),\psi(x)\rangle$ in $B(H_x)_*$ to find enough Borel sections, we leave the proof of the following result as an exercise.

\begin{proposition}\label{prop.predual-BH}
Let $H = (H_x)_{x \in X}$ be a Borel field of separable Hilbert spaces. There is a unique standard Borel structure on $B(H)_* = (B(H_x)_*)_{x \in X}$ such that $\pi : B(H)_* \to X$ and $B(H)_* \times_\pi B(H) \to \C$ are Borel maps. Then, $B(H)_*$ is a Borel field of separable Banach spaces and the standard Borel structures on $B(H)$ given by Propositions \ref{prop.dual-field-Banach} and \ref{prop.bounded-operators-banach} coincide.
\end{proposition}

We are now ready to prove Proposition \ref{prop.equivalence-defs-Borel-field-vNalg}.

\begin{proof}[{Proof of Proposition \ref{prop.equivalence-defs-Borel-field-vNalg}}]
1.\ By the third point in Lemma \ref{lem.positive-functionals}, we can choose a dense sequence of Borel sections $\om_n : X \to \ball M_*^+$. Define the Borel section
\[
\om : X \to M_*^+ : \om_x = \sum_{n \in \N} 2^{-n} \om_n(x) \; .
\]
Note that for every $x \in X$, $\om_x$ is a faithful normal positive functional on $M^x$.

By Proposition \ref{prop.GNS}, we find a Borel field $H = (H_x)_{x \in X}$ of separable Hilbert spaces and a family of normal unital $*$-homomorphisms $\theta_x : M^x \to B(H_x)$ such that the resulting map $\theta : M \to B(H)$ is Borel.

Since $\om_x$ is faithful, also $\theta_x$ is faithful and it follows that $\theta$ is an injective Borel map. So, $\theta(M)$ is a Borel subset of $B(H)$. We have thus proven that $(\theta_x(M^x) \subset B(H_x))_{x \in X}$ is a concrete Borel field of von Neumann algebras.

2.\ By Proposition \ref{prop.predual-BH}, we may as well view $B(H)$ as the dual of the Borel field of separable Banach spaces $B(H)_*$. By Proposition \ref{prop.duality}, there is a unique standard Borel structure on $M_*$ such that the maps $\pi : M_* \to X$ and $M_* \times_\pi M \to \C$ are Borel. Also, with this Borel structure, $M_*$ becomes a Borel field of separable Banach spaces and the identification of $M$ with $(M_*)^*$ is a Borel isomorphism. Since the adjoint and composition maps are Borel on $B(H)$, also their restrictions to the Borel set $M$ are Borel. So, $M$ is an abstract Borel field of separable von Neumann algebras.
\end{proof}

Having proven Proposition \ref{prop.equivalence-defs-Borel-field-vNalg}, there is no longer a need to distinguish between abstract and concrete Borel fields of separable von Neumann algebras and we will simply speak about Borel fields of separable von Neumann algebras.

Contrary to the delicate situation (and counterexamples) in Proposition \ref{prop.counterex-weak-star-closed-subfield}, we have the following result. As a consequence of this result, any reasonably explicit construction of a family of von Neumann algebras will actually give a Borel field of von Neumann algebras.

\begin{proposition}\label{prop.construction}
Let $H = (H_x)_{x \in X}$ be a Borel field of separable Hilbert spaces. Let $\vphi_n : X \to B(H)$ be a sequence of Borel sections, where $B(H)$ carries the standard Borel structure given by Proposition \ref{prop.bounded-operators-banach}. For every $x \in X$, define $M^x \subset B(H_x)$ as the von Neumann algebra generated by the set $\{\vphi_n(x) \mid n \in \N\}$. Then, $M \subset B(H)$ is a Borel set and thus, $(M^x \subset B(H_x))_{x \in X}$ is a concrete Borel field of separable von Neumann algebras.
\end{proposition}
\begin{proof}
By Proposition \ref{prop.predual-BH}, we may as well consider $B(H)$ as the dual of $B(H)_*$. As in the proof of point 1 of Lemma \ref{lem.positive-functionals}, using the Kaplansky density theorem, we find a sequence of Borel sections $\psi_n : X \to \ball B(H)$ such that for every $x \in X$, the set $\{\vphi_n(x) \mid n \in \N\}$ is weak$^*$ dense in $\ball M^x$. The conclusion then follows from Proposition \ref{prop.duality}.
\end{proof}

\begin{proposition}
Let $M = (M^x)_{x \in X}$ be a Borel field of separable von Neumann algebras. Then $\cU(M) = (\cU(M^x))_{x \in X}$ is a Borel subset of $M$. Viewing $\cU(M^x)$ as a Polish group with the strong topology, we get that $\cU(M)$ is a Borel field of Polish groups.
\end{proposition}
\begin{proof}
Since $\cU(M)$ is the set of $a \in M$ satisfying $a^* a = 1 = a a^*$, it is immediate that $\cU(M)$ is a Borel subset of $M$. By Lemma \ref{lem.positive-functionals}, we find a strongly dense sequence of Borel sections $\vphi_n : X \to \ball M^{\text{sa}}$. Then $\psi_n : X \to \cU(M) : \psi_n(x) = \exp(\pi i \vphi_n(x))$ is a dense sequence of Borel sections, so that $\cU(M)$ is indeed a Borel field of Polish groups.
\end{proof}

\section{Borel fields of automorphism groups: a counterexample}\label{sec.aut-groups}

We still use the terminology \emph{separable von Neumann algebra} for a von Neumann algebra with separable predual. For any such separable von Neumann algebra $N$, the automorphism group $\Aut N$ becomes a Polish group with the topology given by $\al_n \to \al$ if and only if $\|\om \circ \al_n - \om \circ \al\| \to 0$ for all $\om \in N_*$.

\begin{proposition}\label{prop.aut-vNalg}
Let $M = (M^x)_{x \in X}$ be a Borel field of separable von Neumann algebras.
\begin{enumlist}
\item There is a unique standard Borel structure on $\Aut M = (\Aut M^x)_{x \in X}$ such that the maps $\pi : \Aut M \to X$ and $\Aut M \times_\pi M \to M : (\al,a) \mapsto \al(a)$ are Borel.
\item The composition $\Aut M \times_\pi \Aut M \to \Aut M$ and inverse $\Aut M \to \Aut M$ are Borel maps.
\item For every $\sigma$-finite measure $\mu$ on $X$, there exists a conull Borel set $X_0 \subset X$ such that $\Aut M \cap \pi^{-1}(X_0) = (\Aut M^x)_{x \in X_0}$ is a Borel field of Polish groups.
\item It may happen that $\Aut M$ is not a Borel field of Polish groups, because there may not exist a dense sequence of Borel sections $\al_n : X \to \Aut M$.
\end{enumlist}
\end{proposition}
\begin{proof}
1.\ Consider the standard Borel structure on $B(M_*) = B(M_*,M_*)$ given by Proposition \ref{prop.bounded-operators-banach}. By duality, identify $B(M^x_*)$ with the space $B_{w^*}(M^x)$ of bounded weak$^*$ continuous linear maps from $M^x$ to $M^x$. Define a standard Borel structure on $B_{w^*}(M)$ through this identification. By construction, the natural maps $\pi : B_{w^*}(M) \to X$ and $B_{w^*}(M) \times_\pi M \to M$ are Borel.

Denote by $\End M^x$ the set of normal unital $*$-homomorphisms from $M^x$ to $M^x$. Fix a weak$^*$ dense sequence of Borel sections $\vphi_n : X \to \ball M$. Then,
\begin{multline*}
\End M = \{T \in B_{w^*}(M) \mid \forall n,k \in \N : T(\vphi_n(\pi(T))\vphi_k(\pi(T))) \\ = T(\vphi_n(\pi(T))) \, T(\vphi_k(T)) \;\;\text{and}\;\; T(\vphi_n(\pi(T))^*) = T(\vphi_n(\pi(T)))^* \} \; .
\end{multline*}
So, $\End M$ is a Borel subset of $B_{w^*}(M)$. Then
\[
\{(S,T) \in \End M \times_\pi \End M \mid S \circ T = \id = T \circ S \}
\]
is a Borel set and the map $(S,T) \mapsto S$ is a bijection with $\Aut M$. In this way, we define a standard Borel structure on $\Aut M$.

By construction, the maps $\pi : \Aut M \to X$ and $\Aut M \times_\pi M \to M : (\al,a) \mapsto \al(a)$ are Borel. The uniqueness follows by applying Lemma \ref{lem.construction-uniqueness}.

2.\ This follows from Lemma \ref{lem.construction-uniqueness}, in the same way as in the proof of point~2 of Proposition \ref{prop.bounded-operators-banach}.

3.\ Fix a dense sequence of Borel sections $\om_n : X \to \ball M_*$. Then
\begin{multline*}
d : \Aut M \times_\pi \Aut M \to [0,+\infty) : d(\al,\be) = \sum_{n \in \N} 2^{-n} \, \|\om_n(x) \circ \al - \om_n(x) \circ \be\| \\ + \sum_{n \in \N} 2^{-n} \, \|\om_n(x) \circ \al^{-1} - \om_n(x) \circ \be^{-1}\|
\end{multline*}
is a Borel map that, on each $\Aut M^x \times \Aut M^x$, is a separable complete metric inducing the topology of $\Aut M^x$.

Also, for every $r,k \in \N$ and $n,m \in \N^k$, we can define the Borel sets $U(k,n,m) \subset \Aut M$ by
\[
U(r,k,n,m) = \{\al \in \Aut M \mid \forall i=1,\ldots,k : \|\om_{n_i}(\pi(\al)) \circ \al - \om_{m_i}(\pi(\al))\| < 1/r \} \; .
\]
For every $x \in X$, the sets $U(r,k,n,m) \cap \pi^{-1}(x)$ form a basis for the Polish topology on $\Aut M^x$. The second statement thus follows from Proposition \ref{prop.subfield}.

4.\ By \cite[Corollary 15]{ST08} and using Proposition \ref{prop.construction}, we can choose a standard Borel space $Y$ and a Borel field of separable von Neumann algebras $(M^y)_{y \in Y}$ such that for every $y \in Y$, $M^y$ is a II$_1$ factor with separable predual and the set
\[
R = \{(y,z) \in Y \times Y \mid M^y \cong M^z \}
\]
is not Borel. We now use the same method as in the proof of Proposition \ref{prop.isom-group-Polish-group}: we write $X = Y \times Y$ and it follows that $(\Aut(M^y \oplus M^z))_{(y,z) \in X}$ is not a Borel field of Polish groups. To check this, it suffices to observe that any automorphism of a direct sum of factors will either preserve the direct sum decomposition or interchange the two direct summands, if they are isomorphic.
\end{proof}

\section{Borel fields of second countable locally compact spaces and groups}\label{sec.lc-groups}

The two main results of this section are the following, solving two problems left open in \cite[Section 3]{Sut85}. First we prove that for a Borel field $G = (G_x)_{x \in X}$ of locally compact second countable groups, the family of Pontryagin dual groups $\Ghat = (\Ghat_x)_{x \in X}$ canonically is a Borel field of locally compact second countable abelian groups. Then we prove that also the Haar measures on $G_x$ can be chosen in a Borel manner.

Recall that for a locally compact group $G$, the Pontryagin dual $\Ghat$ is the group of continuous homomorphisms $\om : G \to \T$. Equipped with the compact-open topology, $\Ghat$ is a locally compact abelian group.

\begin{proposition}\label{prop.Borel-dual-group}
Let $G = (G_x)_{x \in X}$ be a Borel field of Polish groups. Assume that every $G_x$ is locally compact. Consider the family $\Ghat = (\Ghat_x)_{x \in X}$ of their Pontryagin duals.

There is a unique standard Borel structure on $\Ghat$ such that the maps $\pi : \Ghat \to X$ and $\Ghat \times_\pi G \to \T : (\om,g) \mapsto \om(g)$ are Borel. In this way, $\Ghat$ is a Borel field of Polish groups.
\end{proposition}

Recall that up to multiplication with a positive real number, every locally compact group $G$ admits a unique measure $\lambda$ on its Borel $\sigma$-algebra with the following properties: $\lambda(K) < +\infty$ for every compact subset $K \subset G$ and $\lambda(g \cdot A) = \lambda(A)$ for every $g \in G$ and $A \subset G$ Borel. One calls $\lambda$ a left Haar measure on $G$.

\begin{proposition}\label{prop.Borel-Haar-measure}
Let $G = (G_x)_{x \in X}$ be a Borel field of Polish groups. Assume that every $G_x$ is locally compact. Up to multiplication by a Borel function $X \to (0,+\infty)$, there is a unique choice of left Haar measures $\lambda_x$ on $G_x$ such that for every Borel function $F : G \to [0,+\infty]$, the map
\[
X \to [0,+\infty] : x \mapsto \int_{G_x} F \; d\lambda_x \quad\text{is Borel.}
\]
\end{proposition}

Before proving Propositions \ref{prop.Borel-dual-group} and \ref{prop.Borel-Haar-measure}, we need a few general results on Borel fields of second countable locally compact spaces.

\begin{lemma}\label{lem.Borel-compact-filtration}
Let $X$ be a standard Borel space and $V = (V_x)_{x \in X}$ a Borel field of Polish spaces. Assume that every $V_x$ is locally compact.

There exists a sequence of Borel sets $K_n \subset V$ such that $V = \bigcup_{n \in \N} K_n$ and for every $x \in X$ and $n \in \N$, we have that $K_{n,x}$ is a nonempty compact subset of $V_x$ that is contained in the interior of $K_{n+1,x}$.
\end{lemma}
\begin{proof}
Choose a compatible Borel metric $d : V \times_\pi V \to [0,+\infty)$ and choose a dense sequence of Borel sections $\vphi_n : X \to V$. Replacing $d$ by $\min\{d,1\}$, we may assume that $d(v,w) \leq 1$ for all $v,w \in V \times_\pi V$.

For every $v \in V_x$ and $r > 0$, denote by $B(v,r)$ the open ball with radius $r$ and center $v$ in $V_x$. Define for every $r > 0$,
\[
D_r = \bigl\{v \in V \bigm| \overline{B(v,r)} \;\;\text{is compact in $V_{\pi(v)}$} \bigr\} \; .
\]
We claim that $D_r$ is a Borel subset of $V$ for all $r > 0$. Since a closed subset of $V_x$ is compact if and only if it is totally bounded, we get that
\begin{align*}
D_r &= \Bigl\{ v \in V \Bigm| \forall n \in \N , \exists k \in \N , \forall j \in \N : \\ &\hspace{3cm}\text{if $d(\vphi_j(\pi(v)),v)<r$ then}\; \vphi_j(\pi(v)) \in \bigcup_{i=1}^k B(\vphi_i(\pi(v)),1/n) \Bigr\} \\
&= \Bigl\{ v \in V \Bigm| \forall n \in \N , \exists k \in \N , \forall j \in \N : \\ &\hspace{3cm}\text{if $d(\vphi_j(\pi(v)),v)<r$, there exists $i \in \{1,\ldots,k\}$}\\
&\hspace{3cm}\text{with $d(\vphi_j(v),\vphi_i(\pi(v)))<1/n$} \Bigr\} \; .
\end{align*}
So, the claim that $D_r \subset V$ is a Borel set follows. Define the function
\[
C' : V \to [0,+\infty] : C'(v) = \sup \{r > \mid v \in D_r \} \; .
\]
For $r > 0$, we have that $C'(v) \leq r$ if and only if
$$v \in \bigcap_{s \in \Q \cap (r,+\infty)} (V \setminus D_s) \; ,$$
so that $C'$ is Borel. Because each $V_x$ is locally compact, we have that $C'(v) > 0$ for all $v \in V$. Because $d \leq 1$, the following holds: if $C'(v) > 1$, then $V_{\pi(v)}$ is compact and thus $C'(v) = +\infty$. We define the adapted Borel function
\[
C : V \to (0,1] : C(v) = \begin{cases} C'(v) &\quad\text{if $C'(v) \leq 1$,}\\ 1 &\quad\text{if $C'(v) > 1$.}\end{cases}
\]
We now define
\[
K_n = \Bigl\{ v \in V \Bigm| \exists i \in \{1,\ldots,n\} : d(v,\vphi_i(\pi(v))) \leq \frac{n}{n+1}C(\vphi_i(\pi(v))) \Bigr\} \; .
\]
Note that $K_n$ is a Borel set for every $n \in \N$. By definition of the function $C$, $K_{n,x}$ is a compact subset of $V_x$ for all $x \in X$ and $n \in \N$. By construction, $K_{n,x}$ is contained in the interior of $K_{n+1,x}$.

It remains to prove that $V = \bigcup_{n \in \N} K_n$. Fix $v \in V$ and write $x = \pi(v)$. Since $V_x$ is locally compact, we can take $0 < \eps < 1/3$ such that $\overline{B(v,3\eps)}$ is compact. Choose $i \in \N$ such that $d(v,\vphi_i(x)) < \eps$. Then $\overline{B(\vphi_i(x),2\eps)} \subset \overline{B(v,3\eps)}$ is compact, so that $C(\vphi_i(x)) \geq 2\eps$. Then take $n \in \N$ large enough such that $n \geq i$ and
\[
\frac{n}{n+1} C(\vphi_i(x)) > \eps \; .
\]
By construction, $v \in K_n$.
\end{proof}

\begin{proposition}\label{prop.Borel-field-cont-functions-compact-open}
Let $V = (V_x)_{x \in X}$ be a Borel field of Polish spaces. Assume that every $V_x$ is locally compact. Consider the family $C(V) = C(V_x)_{x \in X}$ of continuous functions from $V_x$ to $\C$, where we equip each $C(V_x)$ with the compact-open topology, so that each $C(V_x)$ is a Polish space.

There is a unique standard Borel structure on $C(V)$ such that the maps $\pi : C(V) \to X$ and $C(V) \times_\pi V \to \C : (F,v) \mapsto F(v)$ are Borel. In this way, $C(V)$ is a Borel field of Polish spaces. Addition, multiplication and complex conjugation on $C(V)$ are Borel maps.
\end{proposition}
\begin{proof}
Fix $K_n \subset V$ satisfying the conclusions of Lemma \ref{lem.Borel-compact-filtration}. By Proposition \ref{prop.subfield-sigma-compact} (and actually by construction in the proof of Lemma \ref{lem.Borel-compact-filtration}), every $K_n = (K_{n,x})_{x \in X}$ is a Borel field of compact Polish spaces. Consider the associated Borel field $C(K_n) = C(K_{n,x})_{x \in X}$ consisting of the separable Banach spaces of continuous functions from $K_{n,x}$ to $\C$, given by Proposition \ref{prop.Borel-field-CK}.

Define the Borel set
\begin{align*}
Z \subset X \times \prod_{n \in \N} C(K_n) \quad & \text{with $(x,(F_n)_{n \in \N})$ in $Z$ if and only if}\\
&\text{for all $n \in \N$, $\pi(F_n) = x$ and for all $n < m$, $F_m|_{K_{n,x}} = F_n$}\; .
\end{align*}
Then the map $\theta : C(V) \to Z : F \mapsto \bigl(\pi(F),(F|_{K_{n,\pi(F)}})_{n \in \N}\bigr)$ is a bijection and we define a standard Borel structure on $C(V)$ such that $\theta$ is a Borel isomorphism.

By construction, the map $\pi : C(V) \to X$ is Borel. By construction, for every $n \in \N$, the restriction of the map $C(V) \times_\pi V \to \C$ to $C(V) \times_\pi K_n$ is Borel. So also $C(V) \times_\pi V \to \C$ is Borel. The uniqueness statement then follows from Lemma \ref{lem.construction-uniqueness}. Fixing a dense sequence of Borel sections $\vphi_n : X \to V$, it also follows from Lemma \ref{lem.construction-uniqueness} that the Borel $\sigma$-algebra on $C(V)$ is the smallest $\sigma$-algebra that makes the maps $\pi : C(V) \to X$ and $C(V) \to \C : F \mapsto F(\vphi_n(\pi(F)))$ measurable.

To prove that addition $C(V) \times_\pi C(V) \to C(V)$ is Borel, it thus suffices to check that the maps $C(V) \times_\pi C(V) \to \C : (F,G) \mapsto F(\vphi_n(\pi(F))) + G(\vphi_n(\pi(G)))$ are Borel, which is the case as the sum of two Borel maps. The same argument works for the other algebraic operations on $C(V)$.

Define the Borel map
\[
D : C(V) \times_\pi C(V) \to [0,+\infty) : d(F,F') = \sum_{n \in \N} 2^{-n} \min\bigl\{1,\bigl\|(F-F')|_{K_{n,\pi(F)}}\bigr\|_\infty\bigr\} \; .
\]
Since for every $x \in X$ and $n \in \N$, the compact set $K_{n,x}$ is contained in the interior of $K_{n+1,x}$, the restriction of $D$ to $C(V_x) \times C(V_x)$ is a separable complete metric that induces the compact-open topology.

Fix a compatible Borel metric $d : V \times_\pi V \to [0,+\infty)$. Whenever $\vphi : X \to V$ is a Borel section, also
\[
F_\vphi : X \to C(V) : F_\vphi(x) = (v \mapsto d(v,\vphi(x)))
\]
is a Borel section. By taking finite linear combinations with coefficients in $\Q + i \Q$ of finite products of the Borel sections $F_{\vphi_n}$ we find a sequence of Borel sections $F_k : X \to C(V)$ with the property that for every $x \in X$ and every $K \subset V_x$ compact, $\{F_k(x)|_K \mid k \in \N\}$ is uniformly dense in $C(K)$. Thus, $F_k$ is a dense sequence of Borel sections and $C(V)$ is a Borel field of Polish spaces.
\end{proof}

We are now ready to prove Proposition \ref{prop.Borel-dual-group}.

\begin{proof}[{Proof of Proposition \ref{prop.Borel-dual-group}}]
Choose a dense sequence of Borel sections $\vphi_n : X \to G$. Consider the Borel field $C(G)$ of the Polish spaces $C(G_x)$ of continuous functions from $G_x$ to $\C$ with the compact-open topology. Since
\begin{align*}
\Ghat = \bigl\{F \in C(G) \bigm| & \;\forall n \in \N : |F(\vphi_n(\pi(F)))| = 1 \;\;\text{and}\\
& \;\forall n,m \in \N : F(\vphi_n(\pi(F))\vphi_m(\pi(F))) = F(\vphi_n(\pi(F))) \, F(\vphi_m(\pi(F))) \\
& \hspace{2.1cm} \text{and}\;\; F(\vphi_n(\pi(F))^{-1}) = \overline{F(\vphi_n(\pi(F)))} \bigr\} \; ,
\end{align*}
we see that $\Ghat \subset C(G)$ is a Borel set. Since for every $x \in X$, $\Ghat_x$ is locally compact second countable, we see that every $\Ghat_x$ is a $\sigma$-compact closed subset of $C(G)$. It follows from Proposition \ref{prop.subfield-sigma-compact} that $\Ghat$ is a Borel field of Polish groups.
\end{proof}

For a Polish space $V$, we denote by $\Prob V$ the set of probability measures on the Borel $\sigma$-algebra of $V$ and we denote by $\cM^1(V)$ the set of measures $\mu$ on the Borel $\sigma$-algebra of $V$ satisfying $\mu(V) \leq 1$. By \cite[Theorem 17.23]{Kec95}, the set $\cM^1(V)$ is a Polish space in the weak topology, in which $\mu_n \to \mu$ if and only if for every bounded continuous $F \in C_b(V)$, we have that
\[
\int_V F \, d\mu_n \to \int_V F \, d\mu \; .
\]
Clearly, $\Prob V$ is a closed subset of $\cM^1(V)$ and also is a Polish space. When $V$ is compact, we may identify $\cM^1(V)$ with the positive part of the unit ball of $C(V)^*$.

\begin{proposition}\label{prop.Borel-field-prob-measures}
Let $V = (V_x)_{x \in X}$ be a Borel field of Polish spaces. The family $\cM^1_X(V) = \cM^1(V_x)_{x \in X}$ of bounded-by-$1$ measures on $V_x$ admits a unique standard Borel structure such that the map $\pi : \cM^1_X(V) \to X$ and, with the notation of Proposition \ref{prop.Borel-field-cont-functions-compact-open}, the map
\begin{equation}\label{eq.coupling}
\cM^1_X(V) \times_\pi \ball C(V) \to \C : (\mu,F) \mapsto \int_{V_{\pi(\mu)}} F \, d\mu
\end{equation}
is Borel. In this way, $\cM^1_X(V)$ is a Borel field of Polish spaces.

Then for every bounded Borel function $F : V \to \R$, also the map
\begin{equation}\label{eq.IF}
I_F : \cM^1_X(V) \to \R : \mu \mapsto \int_{V_{\pi(\mu)}} F \, d\mu \quad\text{is Borel.}
\end{equation}

If every $V_x$ is compact, the map $\cM^1_X(V) \to C(V)^*$ is Borel.
\end{proposition}

The notation $\cM^1_X(V)$ in Proposition \ref{prop.Borel-field-prob-measures} may be confusing: by definition, $\cM^1_X(V)$ is the disjoint union of the $\cM^1(V_x)$, $x \in X$. So every $\mu \in \cM^1_X(V)$ is a measure on $V_x$ where $x = \pi(\mu)$. We do not use the notation $\cM^1(V)$, which would have been more compatible with earlier notations in this paper, because $V$ is itself a standard Borel space and we might then think of $\cM^1(V)$ as the set of bounded-by-$1$ measures on $V$.

\begin{proof}
We may assume that $X$ itself is a Polish space with its Borel $\sigma$-algebra. We also consider the Polish space $S = \R^\N$ with the topology of pointwise convergence.
Choose a dense sequence of Borel sections $\vphi_n : X \to V$ and define the injective Borel map
\[
\theta : V \to X \times S : \theta(v) = (\pi(v),(d(v,\vphi_n(\pi(v))))_{n \in \N}) \; .
\]
Then $Z = \theta(V)$ is a Borel subset of $X \times S$. By \cite[Theorem 17.23]{Kec95}, the set $\cM^1(X \times S)$ of bounded-by-$1$ measures on $X \times S$ is a Polish space when equipped with the weak topology. By \cite[Theorem 17.24]{Kec95}, for every bounded Borel function $F : X \times S \to \R$, the map
\begin{equation}\label{eq.map-JF}
J_F : \cM^1(X \times S) \to \R : \mu \mapsto \int_{X \times S} F \; d\mu \quad\text{is Borel.}
\end{equation}
In particular, for every Borel set $U \subset X \times S$, the map $\cM^1(X \times S) \to \R : \mu \mapsto \mu(U)$ is Borel.

We claim that
\[
Y_1 \subset X \times \cM^1(X \times S) : Y_1 = \{(x,\mu) \mid \;\text{$\mu$ is supported on $\{x\} \times S$}\;\}
\]
is a Borel set. Choosing a compatible metric $d$ on $X$ and a dense sequence $(x_n)_{n \in \N}$ in $X$, this follows because
\[
Y_1 = \{(x,\mu) \mid \forall k,n \in \N : \text{if $d(x,x_n) \geq 1/k$, then $\mu(B(x_n,1/k) \times S)=0$}\;\} \; .
\]
Then also
\[
Y \subset X \times \cM^1(X \times S) : Y = \{(x,\mu) \mid (x,\mu) \in Y_1 \;\text{and}\; \mu((X \times S) \setminus Z)=0\}
\]
is a Borel set.

For every $x \in X$, define the injective continuous map $\theta_x : V_x \to S : \theta_x(v) = (d(v,\vphi_n(x)))_{n \in \N}$. Then, $\theta_x(V_x) = Z_x$. Since
\[
Y = \{(x,\delta_x \times \eta) \mid \eta \in \cM^1(S) \;\text{and}\; \eta(S \setminus Z_x)=0 \} \; ,
\]
we find that
\[
\Theta : \cM^1_X(V) \to Y : \mu \mapsto (\pi(\mu),\delta_{\pi(\mu)} \times (\theta_{\pi(\mu)})_*(\mu))
\]
is a bijection. We define the standard Borel structure on $\cM^1_X(V)$ so that $\Theta$ is a Borel isomorphism.

Let $F : V \to \R$ be a bounded Borel function. Define the bounded Borel function $F' : X \times S \to \R$ by $F'(x,s) = 0$ if $(x,s) \not\in Z$ and $F'(x,s) = F(\theta^{-1}(x,s))$ if $(x,s) \in Z$. Consider the map $J_{F'}$ given by \eqref{eq.map-JF}. Then also
\[
J : X \times \cM^1(X \times S) \to \R : (x,\mu) \mapsto J_{F'}(\mu)
\]
is Borel. Using the notation of \eqref{eq.IF}, we have that $I_F = J \circ \Theta$. Thus, $I_F$ is Borel.

Whenever $\vphi : X \to \ball C(V)$ is a Borel section, the map $F : V \to \C : F(v) = \vphi(\pi(v))(v)$ is a bounded Borel function, because the natural map $C(V) \times_\pi V \to \C$ is Borel. Since $I_F$ is Borel, we deduce that the map in \eqref{eq.coupling} is Borel. The uniqueness of the standard Borel structure making \eqref{eq.coupling} Borel then follows from Lemma \ref{lem.construction-uniqueness}.

Let $\vphi_n : X \to V$ be a dense sequence of Borel sections. For every $n \in \N$, define $\om_n : X \to \cM^1_X(V) : x \mapsto \delta_{\vphi_n(x)}$. Then every $\om_n$ is a Borel section. By \cite[Theorem 17.19]{Kec95}, the countable family of Borel sections of the form $\sum_{k=1}^n t_k \om_k$, with $t_k \in \Q^+$ and $\sum_{k=1}^n t_k \leq 1$ provides a dense sequence of Borel sections. So, $\cM^1_X(V)$ is a Borel field of Polish spaces.
\end{proof}

We are now ready to prove Proposition \ref{prop.Borel-Haar-measure}.

\begin{proof}[{Proof of Proposition \ref{prop.Borel-Haar-measure}}]
By Lemma \ref{lem.Borel-compact-filtration}, we can choose a Borel subset $K = (K_x)_{x \in X}$ of $G = (G_x)_{x \in X}$ such that for every $x \in X$, $K_x$ is a compact subset of $G_x$ with nonempty interior. By Proposition \ref{prop.subfield-sigma-compact} (and actually by construction in the proof of Lemma \ref{lem.Borel-compact-filtration}), $K$ is a Borel field of compact Polish spaces. We consider the corresponding Borel field $\cM^1_X(K)$ of the Polish spaces $\cM^1(K_x)$ of bounded-by-$1$ measures on $K_x$, given by Proposition \ref{prop.Borel-field-prob-measures}.

Whenever $\mu \in \cM^1(K_x)$ and $U \subset K_x$ is a Borel set, we denote by $\mu|_U \in \cM^1(K_x)$ the restriction of $\mu$ to $U$, i.e.\ the measure given by $\mu|_U(A) = \mu(A \cap U)$. We claim that whenever $U = (U_x)_{x \in X}$ is a Borel subset of $K = (K_x)_{x \in X}$, the maps
\[
\cM^1(K_x) \to \cM^1(K_x) : \mu \mapsto \mu|_{U_x}
\]
define together a Borel map $\res_U : \cM^1_X(K) \to \cM^1_X(K)$. To prove this claim, by statement 1.c) in Lemma \ref{lem.construction-uniqueness},
it suffices to check that for every bounded Borel function $F : K \to \R$, the map $I_F \circ \res_U$ is Borel. But, for $\mu \in \cM^1(K_x)$, we have that
\[
(I_F \circ \res_U)(\mu) = \int_{K_x} F \, d\mu|_{U_x} = \int_{K_x} F 1_U \, d\mu = I_{F 1_U}(\mu) \; .
\]
So, $I_F \circ \res_U = I_{F 1_U}$ is Borel.

Choose a dense sequence of Borel sections $\vphi_n : X \to G$. Taking finite products of $\vphi_n$ and $\vphi_n(\cdot)^{-1}$, we may assume that for every $x \in X$, the set $\{\vphi_n(x) \mid n \in \N\}$ is a subgroup of $G$. Since for every $n \in \N$, the map $G \to G : g \mapsto \vphi_n(\pi(g)) \cdot g$ is a Borel isomorphism, the sets
\[
U_n = (K_x \cap \vphi_n(x) \cdot K_x)_{x \in X}
\]
are Borel subsets of $K$. We write $\al_n = \res_{U_n}$.

Whenever $\mu \in \cM^1(K_x)$ and $g \in G_x$, we can consider the element $g \cdot (\mu|_{g^{-1}\cdot K_x \cap K_x})$ of $\cM^1(K_x)$ given by $A \mapsto \mu(g^{-1}\cdot A \cap K_x)$ for every Borel set $A \subset K_x$. As above, the maps
\[
\cM^1(K_x) \to \cM^1(K_x) : \mu \mapsto \vphi_n(x) \cdot (\mu|_{\vphi_n(x)^{-1}\cdot K_x \cap K_x})
\]
then combine into a Borel map $\be_n : \cM^1_X(K) \to \cM^1_X(K)$.

We define the Borel set $H \subset \Prob_X(K)$ by
\[
H = \{\mu \in \Prob_X(K) \mid \forall n \in \N : \al_n(\mu) = \be_n(\mu)\} \; .
\]
By construction, $H = (H_x)_{x \in X}$ where $\mu \in H_x$ if and only if $\mu$ is a probability measure on $K_x$ satisfying
\[
\mu(\vphi_n(x)^{-1} \cdot A) = \mu(A) \quad\text{for all $n \in \N$ and all Borel sets $A \subset \vphi_n(x)\cdot K_x \cap K_x$.}
\]
This means that $H_x$ is a singleton: $H_x = \{\lambda_x|_{K_x}\}$, where $\lambda_x$ is the unique left Haar measure on $G_x$ satisfying $\lambda_x(K_x) = 1$. Since $H$ is a Borel set, the map $X \to \Prob_X(K) : x \mapsto \lambda_x|_{K_x}$ is a Borel section. Whenever $F : K \to \R$ is a bounded Borel function, the composition of this Borel section with $I_F$ is a Borel map. This precisely means that
\[
x \mapsto \int_{K_x} F \; d\lambda_x
\]
is Borel for every bounded Borel function $F : K \to \R$. By the monotone convergence theorem, the same holds for all Borel functions $F : K \to [0,+\infty]$.

Because for every $x \in X$, we have that $G_x = \bigcup_n \vphi_n(x) \cdot K_x$, using left invariance of the Haar measure, it then also follows that
\begin{equation}\label{eq.this-is-Borel}
X \to [0,+\infty] : x \mapsto \int_{G_x} F \; d\lambda_x
\end{equation}
is Borel for every Borel function $F : G \to [0,+\infty]$.

Assume that $(\lambda'_x)_{x \in X}$ is another choice of left Haar measures such that the maps in \eqref{eq.this-is-Borel} are Borel. Taking $F = 1_K$, it follows that $x \mapsto \lambda_x(K_x)$ and $x \mapsto \lambda'_x(K_x)$ are both strictly positive Borel functions. Taking their quotient, we find a Borel function $r : X \to (0,+\infty)$ such that $\lambda'_x = r(x) \lambda_x$ for all $x \in X$.
\end{proof}

We conclude this section by proving the following result. This can be used as the starting point to prove that all kinds of functional analytic and operator algebraic structures associated with a locally compact group $G$ can be turned into a Borel field.

In particular, we define in Proposition \ref{prop.Borel-field-LpG} the Borel field of separable Banach spaces $L^p(G) = (L^p(G_x))_{x \in X}$. We do this by completing $C_c(G)$ and thus first need the following result.

\begin{proposition}\label{prop.Borel-field-C0}
Let $V = (V_x)_{x \in X}$ be a Borel field of Polish spaces. Assume that every $V_x$ is locally compact. The family $C_0(V) = C_0(V_x)_{x \in X}$ of continuous functions from $V_x$ to $\C$ tending to zero at infinity admits a unique standard Borel structure such that the maps $\pi : C_0(V) \to X$ and $C_0(V) \times_\pi V \to \C : (F,v) \mapsto F(v)$ are Borel. In this way, $C_0(V)$ is a Borel field of the separable C$^*$-algebras $C_0(V_x)$ with the supremum norm.

Also, $C_c(V) = C_c(V_x)_{x \in X}$ is a Borel subset of $C_0(V)$ and becomes a Borel field of separable normed spaces.
\end{proposition}
\begin{proof}
Consider the standard Borel structure on $C(V) = C(V_x)_{x \in X}$ given by Proposition \ref{prop.Borel-field-cont-functions-compact-open}. Fix a dense sequence of Borel sections $\vphi_i : X \to V$. Fix a sequence of Borel sets $K_n \subset V$ satisfying the conclusion of Lemma \ref{lem.Borel-compact-filtration}. Define the Borel sets $W_{i,n} \subset X$ by $W_{i,n} = \vphi_i^{-1}(V \setminus K_n)$. Since the maps $C(V) \to \R : F \mapsto |F(\vphi_i(\pi(F)))|$ are Borel, for every $r \geq 0$, the sets
\[
T_{i,r} = \{F \in C(V) \mid |F(\vphi_i(\pi(F)))| \leq r\} \quad\text{are Borel.}
\]
So, for every $r \geq 0$ and $n \in \N$, we conclude that
\begin{align*}
S_{r,n} &= \bigcap_{i \in \N} \bigl( (\pi^{-1}(W_{i,n}) \cap T_{i,r}) \cup (C(V) \setminus \pi^{-1}(W_{i,n}))\bigr) \\
&= \bigl\{F \in C(V) \bigm| \text{with $x = \pi(F)$, we have that $\|F|_{V_x \setminus K_{n,x}}\|_\infty \leq r$}\bigr\}
\end{align*}
is Borel. Writing $\Q^+ = \Q \cap (0,+\infty)$, it follows that
\[
C_c(V) = \bigcup_{n \in \N} S_{0,n} \quad\text{and}\quad C_0(V) = \bigcap_{r \in \Q^+} \bigcup_{n \in \N} S_{r,n}
\]
are Borel subsets of $C(V)$. By restricting the standard Borel structure on $C(V)$ to $C_0(V)$, we find a standard Borel structure on $C_0(V)$ such that the maps $\pi : C_0(V) \to X$ and $C_0(V) \times_\pi V \to \C$ are Borel. The uniqueness of such a standard Borel structure follows from Lemma \ref{lem.construction-uniqueness}.

Since addition and (scalar) multiplication are Borel on $C(V)$ (see Proposition \ref{prop.Borel-field-cont-functions-compact-open}), the same is true on $C_0(V)$ and $C_c(V)$. Also note that
\[
C_0(V) \to [0,+\infty) : F \mapsto \|F\|_\infty = \sup \{|F(\vphi_i(\pi(F)))| \mid i \in \N\}
\]
is Borel.

It thus only remains to prove that there exists a dense sequence of Borel sections $F_n : X \to C_c(V)$. Fix a compatible Borel metric $d : V \times_\pi V \to [0,+\infty)$. Replacing $d$ by $\min \{1,d\}$, we may assume that $d(v,w) \leq 1$ for all $(v,w) \in V \times_\pi V$. We claim that for every $n \in \N$,
\[
s_n : X \to C_c(V) : s_{n,x}(v) = d(v,V_x \setminus K_{n,x}) \;\;\text{for $v \in V_x$}
\]
is a Borel section. By convention, we write $d(v,\emptyset) = 1$, so that $s_{n,x}(v) = 1$ for all $v \in V_x$ when $V_x = K_{n,x}$. Note that $s_{n,x}(v) = 0$ for $v \in V_x \setminus K_{n,x}$, so that $s_{n,x} \in C_c(V_x)$. Also note that $s_{n,x}(v) > 0$ for all $v$ in the interior of $K_{n,x}$. To prove the claim, it suffices to prove that $X \to [0,1] : x \mapsto s_{n,x}(\vphi_k(x))$ is Borel for all $n,k \in \N$. For every $0 < r < 1$, we have that
\[
s_{n,x}(\vphi_k(x)) < r \quad\text{if and only if}\quad \exists i \in \N : \vphi_i(x) \in V \setminus K_n \;\;\text{and}\;\; d(\vphi_k(x),\vphi_i(x)) < r \; ,
\]
so that the claim follows.

By Proposition \ref{prop.Borel-field-cont-functions-compact-open}, we can fix a sequence of Borel sections $f_i : X \to C(V)$ that is dense in the compact-open topology. We conclude the proof by showing that the countable family of Borel sections $X \to C_c(V) : x \mapsto f_{i,x} \, s_{n,x}^{1/k}$ with $i,n,k \in \N$, is dense. Fix $x \in X$, $F \in C_c(V_x)$ and $\eps > 0$. Take $n \geq 2$ such that $F(v) = 0$ for all $v \in V_x \setminus K_{n-1,x}$. When $k \to +\infty$, the sequence $s_{n,x}^{1/k}$ tends to $1$ uniformly on compact subsets of the interior of $K_{n,x}$. So we can take $k \in \N$ such that
\[
\|(s_{n,x}^{1/k} - 1)|_{K_{n-1,x}}\|_\infty < (1 + 2\|F\|_\infty)^{-1} \, \eps \; .
\]
Since $\{f_{i,x} \mid i \in \N\} \subset C(V_x)$ is dense in the compact-open topology, we can take $i \in \N$ such that $\|(f_{i,x} - F)|_{K_{n,x}}\|_\infty < \eps/2$. To conclude the proof, it suffices to show that
\begin{equation}\label{eq.goal-v}
|f_{i,x}(v) \, s_{n,x}^{1/k}(v) - F(v)| < \eps \quad\text{for all $v \in V_x$.}
\end{equation}
\begin{itemlist}
\item If $v \in K_{n-1,x}$, then $|s_{n,x}^{1/k}(v) - 1| < (1 + 2\|F\|_\infty)^{-1} \, \eps$ and $|f_{i,x}(v) - F(v)| < \eps/2$ so that \eqref{eq.goal-v} holds.
\item If $v \in K_{n,x} \setminus K_{n-1,x}$, then $|s_{n,x}^{1/k}(v)| \leq 1$, $F(v) = 0$ and $f_{i,x}(v) - F(v)| < \eps/2$, so that $|f_{i,x}(v)| < \eps/2$ and again \eqref{eq.goal-v} holds.
\item If $v \in V_x \setminus K_{n,x}$ then $s_{n,x}^{1/k}(v) = 0$ and $F(v)=0$ so that \eqref{eq.goal-v} holds trivially.\qedhere
\end{itemlist}
\end{proof}

\begin{proposition}\label{prop.Borel-field-LpG}
Let $G = (G_x)_{x \in X}$ be a Borel field of Polish groups that are all locally compact. Let $(\lambda_x)_{x \in X}$ be a choice of left Haar measures satisfying the conclusion of Proposition \ref{prop.Borel-Haar-measure}. Let $1 \leq p < +\infty$.

Then $L^p(G) = L^p(G_x,\lambda_x)_{x \in X}$ admits a unique standard Borel structure such that $\pi : L^p(G) \to X$ is Borel and the map
\begin{equation}\label{eq.Lp-pairing}
L^p(G) \times_\pi C_c(G) \to \C : (\xi,F) \mapsto \int_{\pi(\xi)} \xi \, F \, d\lambda_{\pi(\xi)} \quad\text{is Borel.}
\end{equation}
Here, the Borel structure on the family $C_c(G) = C_c(G_x)_{x \in X}$ of spaces of continuous compactly supported functions is introduced in Proposition \ref{prop.Borel-field-C0}. In this way, $L^p(G)$ is a Borel field of separable Banach spaces.

Also, for $p=1$, the map
\begin{equation}\label{eq.L1-integral}
L^1(G) \to \C : \xi \mapsto \int_{\pi(\xi)} \xi \, d\lambda_{\pi(\xi)} \quad\text{is Borel.}
\end{equation}
\end{proposition}

\begin{proof}
We start by proving that the map
\begin{equation}\label{eq.map-I}
I : C_c(G) \to \C : F \mapsto \int_{G_{\pi(F)}} F \, d\lambda_{\pi(F)}
\end{equation}
is Borel. Take $K_n \subset G$ satisfying the conclusions of Lemma \ref{lem.Borel-compact-filtration}. As explained in the proof of Proposition \ref{prop.Borel-field-cont-functions-compact-open}, each $C(K_n) = C(K_{n,x})_{x \in X}$ is a Borel field of separable Banach spaces and the restriction maps $C_c(G_x) \to C(K_{n,x})$ combine into a Borel map $\res_n : C_c(G) \to C(K_n)$. We claim that the maps
\[
I_n : C(K_n) \to \C : F \mapsto \int_{K_{n,\pi(F)}} F \, d\lambda_{\pi(F)}
\]
are Borel. Since for every $x \in X$, integration w.r.t.\ $\lambda_x$ defines a bounded linear functional on $C(K_{n,x})$, it suffices to prove that for every Borel section $\psi : X \to C(K_n)$, the map $I_n \circ \psi$ is Borel. To prove this, note that the map $H : G \to \C$ defined by $H(g) = 0$ if $g\not\in K_n$ and $H(g) = \psi(\pi(g))(g)$ if $g \in K_n$ is Borel, and that
$$(I_n \circ \psi)(x) = \int_{G_x} H \, d\lambda_x \; .$$
It thus follows from Proposition \ref{prop.Borel-Haar-measure} that $I_n \circ \psi$ is Borel. Since for every $F \in C_c(G)$, $I_n(\res_n(F)) \to I(F)$, it follows that $I$ is a Borel map.

Fix $p \in [1,+\infty)$. Since $I$ is a Borel map, also
\[
\|\cdot\|_p : C_c(G) \to [0,+\infty) : F \mapsto I(|F|^p)^{1/p}
\]
is Borel. Using the specific dense sequence of Borel sections $X \to C_c(G)$ constructed in the proof of Proposition \ref{prop.Borel-field-C0}, it follows that there also exists a dense sequence of Borel sections for $(C_c(G),\|\cdot\|_p)$, which is thus a Borel field of the separable normed spaces $(C_c(G_x),\|\cdot\|_p)_{x \in X}$.

We realize their completion as $L^p(G_x,\lambda_x)$, with the obvious inclusion $C_c(G_x) \subset L^p(G_x,\lambda_x)$. By Corollary \ref{cor.completion-normed}, there is a unique standard Borel structure on $L^p(G)$ such that $\pi : L^p(G) \to X$ and the inclusion $C_c(G) \hookrightarrow L^p(G)$ are Borel and such that $L^p(G)$ becomes a Borel field of separable Banach spaces.

To prove that \eqref{eq.Lp-pairing} is Borel, note that for every $x \in X$ and $F \in C_c(G_x)$, the map
\[
L^p(G_x) \to \C : \xi \mapsto \int_{G_x} \xi F \, d\lambda_x
\]
is continuous. By Proposition \ref{prop.completion}, it thus suffices to prove that the restriction of \eqref{eq.Lp-pairing} to $C_c(G) \times_\pi C_c(G)$ is Borel, which follows because the map $I$ in \eqref{eq.map-I} is Borel. In the same way, one proves that \eqref{eq.L1-integral} is Borel.
\end{proof}

\section{Tensor products and crossed products}\label{sec.tensor-crossed}

\begin{proposition}\label{prop.tensor-product-banach}
Let $(V_x)_{x \in X}$ and $(W_x)_{x \in X}$ be Borel fields of separable Banach spaces. Consider the injective tensor products $V \otvee W := (V_x \otvee W_x)_{x \in X}$ and projective tensor products $V \othat W = (V_x \othat W_x)_{x \in X}$. Then $V \otvee W$ and $V \othat W$ have a unique standard Borel structure such that they are Borel fields of separable Banach spaces and the canonical maps $V \times_\pi W \to V \otvee W$ and $V \times_\pi W \to V \othat W$ sending $(v,w)$ to $v \ot w$ are Borel.

The quotient map $V \othat W \to V \otvee W$ is Borel.
\end{proposition}
\begin{proof}
Mimicking the proof of Proposition \ref{prop.dual-field-Banach}, there is a unique standard Borel structure on $B(V,W^*) = B(V_x,W_x^*)_{x \in X}$ such that the map
$$V \times_\pi W \times_\pi B(V,W^*) \to \C : (v,w,S) \mapsto \langle S(v),w\rangle$$
is Borel. With the point-weak$^*$ topology, every $\ball B(V_x,W_x^*)$ is compact. It also follows in the same way as in Proposition \ref{prop.dual-field-Banach} that $\ball B(V,W^*)$ is a Borel field of compact Polish spaces. Choose a point-weak$^*$ dense sequence of Borel sections $S_n : X \to \ball B(V,W^*) : x \mapsto S_{n,x}$.

Choose dense sequences of Borel sections $\vphi_k : X \to V$ and $\psi_l : X \to W$. Define the Borel maps $\theta_{k,l} : X \to \C^\N : \theta_{k,l}(x) = (\langle S_{n,x}(\vphi_k(x)),\psi_l(x)\rangle)_{n \in \N}$. Let $\Theta$ be the countable set of Borel maps $\theta : X \to \C^\N$ consisting of all finite linear combinations of $\theta_{k,l}$ with coefficients in $\Q + i \Q$. Then define the Borel set $Z \subset X \times \C^\N$ of elements $(x,\rho)$ satisfying
$$\forall r \geq 1, \exists \theta \in \Theta, \forall n \in \N : |\rho_n - \theta(x)_n| < 1/r \; .$$
Defining $\pi(x,\rho) = x$ and $\|(x,\rho)\| = \sup_{n \in \N} |\rho_n|$, we get that $Z$ is a Borel field of separable Banach spaces.

By construction, there is a canonical bijection $\Psi : V \othat W \to Z$ such that $\pi \circ \Psi = \pi$, such that $\Psi(v \ot w) = (\pi(v),(\langle S_n(v),w\rangle)_{n \in \N})$ for all $(v,w) \in V \times_\pi W$ and such that for every $x \in X$, $\Psi_x : V_x \othat W_x \to Z_x$ is a bijective isometry. We define a standard Borel structure on $V \othat W$ such that $\Psi$ is a Borel isomorphism. So, $V \othat W$ is a Borel field of separable Banach spaces and the natural map $V \times_\pi W \to V \othat W$ is Borel. The uniqueness follows from Lemma \ref{lem.uniqueness}.

By Lemma \ref{lem.uniqueness}, the canonical pairing $(V \othat W) \times_\pi B(V,W^*) \to \C$ is Borel. Choose weak$^*$ dense sequences of Borel sections $\om_k : X \to \ball V^*$ and $\eta_l : X \to \ball W^*$. Define $\zeta_{k,l} : X \to B(V,W^*)$ by $\langle\zeta_{k,l}(x)(v),w\rangle = \om_k(x)(v) \, \eta_l(x)(w)$. Since for every $u \in V \othat W$,
$$\|u\|_\vee = \sup_{k,l \in \N} |\langle u,\zeta_{k,l}(\pi(u))\rangle| \; ,$$
it follows that the map $V \othat W \to [0,+\infty) : u \mapsto \|u\|_\vee$ is Borel. Taking separation-completion, it follows from Corollary \ref{cor.completion-normed} that $V \otvee W$ is a well-defined Borel field of separable Banach spaces. Again uniqueness follows from Lemma \ref{lem.uniqueness}.
\end{proof}

We use Proposition \ref{prop.tensor-product-banach} to also construct the tensor product of two Borel fields of separable Hilbert spaces.

\begin{corollary}\label{cor.tensor-product-Hilbert}
Let $(H_x)_{x \in X}$ and $(K_x)_{x \in X}$ be Borel fields of separable Hilbert spaces. Then the Hilbert space tensor product $H \ot K = (H_x \ot K_x)_{x \in X}$ has a unique standard Borel structure such that $H \ot K$ becomes a Borel field of separable Hilbert spaces and the map $H \times_\pi K \to H \ot K : (\xi,\eta) \mapsto \xi \ot \eta$ is Borel.
\end{corollary}
\begin{proof}
Denote by $\theta : H \othat K \to H \ot K$ the map where for every $x \in X$, $\theta_x : H_x \othat K_x \to H_x \ot K_x$ is the unique contractive linear map satisfying $\theta_x(\xi \ot \eta) = \xi \ot \eta$ for all $\xi \in H_x$ and $\eta \in K_x$. By Lemma \ref{lem.uniqueness}, the map
$$(H \othat K) \times_\pi (H \othat K) \to \C : (v,w) \mapsto \langle \theta(v),\theta(w)\rangle$$
is Borel. Taking separation-completion, it follows from Corollary \ref{cor.completion-normed} that $H \ot K$ is a well-defined Borel field of separable Hilbert spaces. By construction, the map $(\xi,\eta) \mapsto \xi \ot \eta$ is Borel. As before, uniqueness follows from Lemma \ref{lem.uniqueness}.
\end{proof}

In order to also define the minimal and maximal tensor product of Borel fields of C$^*$-algebras, we need two lemmas that will also be useful to define Borel fields of crossed products.

For the first lemma, given a $*$-algebra $A$ with norm $\|\,\cdot\,\|$ satisfying $\|ab\| \leq \|a\| \, \|b\|$ and $\|a^*\| = \|a\|$ for all $a,b \in A$, we define
\begin{align*}
\|a\|_u = \sup \bigl\{ \|\theta(a)\| \bigm| \; &\text{$\theta : A \to B(H)$ is a $*$-representation on a Hilbert space satisfying}\\
& \text{$\|\theta(b)\| \leq \|b\|$ for all $b \in A$}\,\bigr\} \; .
\end{align*}
We denote by $A^u$ the separation-completion of $(A,\|\,\cdot\,\|_u)$, with the canonical map $\theta^u : A \to A^u$.

\begin{lemma}\label{lem.universal-Cstar-algebra}
Let $A = (A_x)_{x \in X}$ be a Borel field of separable normed $*$-algebras and consider the C$^*$-algebras $A^u = (A^u_x)_{x \in X}$. There is unique standard Borel structure on $A^u$ such that $\pi : A^u \to X$ and $\theta^u : A \to A^u$ are Borel and such that $A^u$ becomes a Borel field of separable C$^*$-algebras.
\end{lemma}
\begin{proof}
By Corollary \ref{cor.completion-normed} and point~3 of Proposition \ref{prop.completion}, it suffices to prove that the map $A \to [0,+\infty) : a \mapsto \|a\|_u$ is Borel.

Denote by $A^* = (A_x^*)_{x \in X}$ the dual Banach spaces of $A_x$. As in Proposition \ref{prop.dual-field-Banach}, there is a unique standard Borel structure on $A^*$ such that the map $\pi : A^* \to X$ and the pairing $A^* \times_\pi A \to \C$ are Borel. Also as in Proposition \ref{prop.dual-field-Banach}, $\ball A^*$ with the weak$^*$ topology on each $\ball A^*_x$ is a Borel field of compact Polish spaces.

For every $x \in X$, define the weak$^*$ closed subset $S_x \subset \ball A^*_x$ of $\om \in A^*_x$ satisfying $\|\om\| \leq 1$, $\om(a^*a) \geq 0$ and $\om(b^* a^* a b) \leq \|a\|^2 \, \om(b^* b)$ for all $a,b \in A_x$. Write $S = (S_x)_{x \in X}$ so that $S \subset \ball A^*$. Since $\om \in \ball A^*$ with $x = \pi(\om)$ belongs to $S$ if and only if
$$\om(\vphi_n(x)^* \vphi_n(x)) \geq 0 \quad\text{and}\quad \om(\vphi_k(x)^* \vphi_n(x)^* \vphi_n(x) \vphi_k(x)) \leq \|\vphi_n(x)\|^2 \, \om(\vphi_k(x)^* \vphi_k(x))$$
for all $n,k \in \N$, it follows that $S \subset \ball A^*$ is Borel. By Theorem \ref{thm.selection-sigma-compact}, $S$ is a Borel field of Polish spaces. Fix a weak$^*$ dense sequence of Borel sections $\om_n : X \to S$.

By construction,
$$\|a\|_u^2 = \sup\Bigl\{ \frac{\langle \om_n(x) , \vphi_k(x)^* a^* a \vphi_k(x) \rangle}{\langle \om_n(x) , \vphi_k(x)^* \vphi_k(x) \rangle + 1/r} \Bigm| x = \pi(a) , n,k,r \in \N \Bigr\} \; .$$
It follows that $a \mapsto \|a\|_u$ is Borel and the lemma is proven.
\end{proof}

As we state more explicitly in Proposition \ref{prop.universal-field-Cstar} below, the following lemma shows that our notion of Borel fields of separable C$^*$-algebras is compatible with the Borel parametrization of separable C$^*$-algebras defined in \cite{Kec96}.

\begin{lemma}\label{lem.GNS-field-Cstar}
Let $A = (A_x)_{x \in X}$ be a Borel field of separable C$^*$-algebras and let $H$ be a separable infinite dimensional Hilbert space. There exists a Borel map $\theta : A \to B(H)$ such that for every $x \in X$, the restriction of $\theta$ to $A_x$ is a faithful, nondegenerate $*$-homomorphism.
\end{lemma}
\begin{proof}
Consider the standard Borel structure on $A^* = (A_x^*)_{x \in X}$ given by Proposition \ref{prop.dual-field-Banach}, so that $\ball A^*$, with the weak$^*$ topology on each $\ball A_x^*$, is a Borel field of compact Polish spaces. For every $x \in X$, consider the weak$^*$ closed subset $S_x \subset \ball A_x^*$ of positive functionals of norm at most $1$. Write $S = (S_x)_{x \in X}$, so that $S \subset \ball A^*$.

Let $\vphi_n : X \to A$ be a dense sequence of Borel sections. Since $\om \in \ball A^*$ with $x = \pi(\om)$ belongs to $S$ if and only if $\om(\vphi_n(x)^* \vphi_n(x)) \geq 0$ for all $n \in \N$, it follows that $S \subset \ball A^*$ is a Borel set. By Theorem \ref{thm.selection-sigma-compact}, $S$ is a Borel field of Polish spaces. Fix a weak$^*$ dense sequence of Borel sections $\om_n : X \to S$ and define
$$\rho : X \to S : x \mapsto \rho_x \in A_x^* : \rho_x(a) = \sum_{n=1}^\infty 2^{-n} \om_n(a) \; .$$
By construction, each $\rho_x$ is a faithful positive bounded functional on $A_x$.

As in the proof of Proposition \ref{prop.equivalence-defs-Borel-field-vNalg}, we can define a Borel field of separable Hilbert spaces $K = (K_x)_{x \in X}$ by completion of $(A_x,\rho_x)$, together with a Borel map $\psi : A \to B(K)$ such that every $\psi_x : A_x \to B(K_x)$ is the GNS-representation, which is a faithful nondegenerate $*$-homomorphism. Replace $K_x$ by $K_x \ot \ell^2(\N)$ and $\psi_x(a)$ by $\psi_x(a) \ot 1$. Then it follows from \cite[Lemma IV.8.12]{Tak79} that the field $(K_x)_{x \in X}$ is isomorphic with the constant field $H$, so that the lemma follows.
\end{proof}

\begin{proposition}\label{prop.min-max-tensor-product}
Let $A = (A_x)_{x \in X}$ and $B = (B_x)_{x \in X}$ be Borel fields of separable C$^*$-algebras. Consider their minimal and maximal tensor products $A \otmin B = (A_x \otmin B_x)_{x \in X}$ and $A \otmax B = (A_x \otmax B_x)_{x \in X}$. Then $A \otmin B$ and $A \otmax B$ have a unique standard Borel structure such that they are Borel fields of separable C$^*$-algebras and the canonical maps $A \times_\pi B \to A \otmin B$ and $A \times_\pi B \to A \otmax B$ sending $(a,b)$ to $a \ot b$ are Borel.

The quotient map $A \otmax B \to A \otmin B$ is Borel.
\end{proposition}
\begin{proof}
By Proposition \ref{prop.tensor-product-banach}, we consider the Borel field $C = A \othat B$ of separable Banach spaces. Note that every $C_x$ naturally is a Banach $*$-algebra. Using Lemma \ref{lem.uniqueness}, one checks that the product and adjoint maps are Borel, so that $C$ is a Borel field of separable Banach $*$-algebras. Since $\|u\|\nmax \leq \|u\|_\wedge$ for all $u \in A_x \otalg B_x$, the statement about $A \otmax B$ now follows immediately from Lemma \ref{lem.universal-Cstar-algebra}.

Let $H$ and $K$ be separable infinite dimensional Hilbert spaces. By Lemma \ref{lem.GNS-field-Cstar}, we can choose Borel maps $\theta_A : A \to B(H)$ and $\theta_B : B \to B(K)$ such that the restrictions to $A_x$, resp.\ $B_x$ are faithful, nondegenerate $*$-homomorphisms. For every $x \in X$, denote by $\theta_x : C_x \to B(H \ot K)$ the unique $*$-homomorphism satisfying $\theta_x(a \ot b) = \theta_A(a) \ot \theta_B(b)$ for all $a \in A_x$ and $b \in B_x$. By Lemma \ref{lem.uniqueness}, the map $\theta : C \to B(H \ot K)$ is Borel.

For every $c \in C$, define $\|c\|_{\text{\rm min}} = \|\theta(c)\|$. It follows that $c \mapsto \|c\|_{\text{\rm min}}$ is Borel. By separation-completion and Corollary \ref{cor.completion-normed}, also the statement about $A \otmin B$ follows.
\end{proof}

In the same spirit as for Proposition \ref{prop.min-max-tensor-product}, we consider reduced and full crossed products. Since the arguments are somehow repetitive, we only sketch the proof. Recall that we defined the standard Borel structure on the multiplier algebras of a Borel field of separable C$^*$-algebras in Proposition \ref{prop.multiplier-field}.

\begin{proposition}\label{prop.full-and-reduced-crossed-product}
Let $A = (A_x)_{x \in X}$ be a Borel field of separable C$^*$-algebras and $G = (G_x)_{x \in X}$ a Borel field of locally compact Polish groups. Assume that $\al : G \times_\pi A \to A$ is a Borel map such that $\pi \circ \al = \al$ and such that for every $x \in X$, the restriction of $\al$ to $G_x \times A_x$ is a continuous action $G_x \actson A_x$ by $*$-automorphisms. Consider the reduced and full crossed products $A \rtimesred G = (A_x \rtimesred G_x)_{x \in X}$ and $A \rtimesf G = (A_x \rtimesf G_x)_{x \in X}$.

There are unique standard Borel structures on $A \rtimesred G$ and $A \rtimesf G$ such that they become Borel fields of separable C$^*$-algebras and such that the canonical maps from $A$ and $G$ to the multiplier algebra are Borel.

The quotient map $A \rtimesf G \to A \rtimesred G$ is Borel.
\end{proposition}
\begin{proof}
Choose Haar measures $\lambda_x$ on $G_x$ satisfying Proposition \ref{prop.Borel-Haar-measure}. Throughout the proof, we use the following observation. If $V = (V_x)_{x \in X}$ is any Borel field of separable normed spaces and $F : G \to V$ is a Borel function that satisfies $\pi \circ F = \pi$ and that is weakly integrable, in the sense that there exists a map $H : X \to V$ such that for every $x \in X$ and $\om \in V_x^*$, the map $G_x \to \C : g \mapsto \om(F(g))$ is $\lambda_x$-integrable with integral $\om(H(x))$, then $H$ is Borel. To prove this, consider the dual Banach spaces $V^* = (V_x^*)_{x \in X}$ and use the standard Borel structure on $V^*$ given by Proposition \ref{prop.dual-field-Banach}. It suffices to prove that for every Borel section $\om : X \to V^*$, the map $X \to \C : x \mapsto \om_x(H(x))$ is Borel. But,
$$\om_x(H(x)) = \int_{G_x} \om_x(F(g)) \, d\lambda_x(g) \; .$$
Since the map $g \mapsto \om_{\pi(g)}(F(g))$ is Borel, it follows from Proposition \ref{prop.Borel-Haar-measure} that $x \mapsto \om_x(H(x))$ is Borel.

As in Proposition \ref{prop.Borel-field-C0}, there is a unique standard Borel structure on $C_c(G,A) = C_c(G_x,A_x)_{x \in X}$ such that the maps $\pi : C_c(G,A) \to X$ and $C_c(G,A) \times G \to A : (F,g) \mapsto F(g)$ are Borel and in this way, $C_c(G,A)$ becomes a Borel field of separable normed spaces, with the supremum norm on each $C_c(G_x,A_x)$. With the usual $L^1$-norm, adjoint and convolution product, $C_c(G,A)$ is a Borel field of separable normed $*$-algebras. The norm on $C_c(G,A)$ given by $A \rtimesf G$ is bounded above by the $L^1$-norm. So, by Lemma \ref{lem.universal-Cstar-algebra}, there is a unique standard Borel structure on $A \rtimesf G$ such that $\pi : A \rtimesf G \to X$ and $\theta\full : C_c(G,A) \to A \rtimesf G$ are Borel and $A \rtimesf G$ is a Borel field of separable C$^*$-algebras.

Let $K$ be a separable infinite dimensional Hilbert space. By Lemma \ref{lem.GNS-field-Cstar}, we can choose a Borel map $\gamma : A \to B(K)$ such that for every $x \in X$, the restriction of $\gamma$ to $A_x$ is a faithful $*$-homomorphism. As in Proposition \ref{prop.Borel-field-LpG}, we canonically define the Borel field of separable Hilbert spaces $L^2(G,K) = L^2(G_x,K)_{x \in X}$. The formula
$$\theta\red : C_c(G,A) \to B(L^2(G,K)) : (\theta\red(F) \xi)(h) = \int_{G_{\pi(h)}} \gamma(\al_{h^{-1}}(F(g))) \, \xi(g^{-1} h) \, d\lambda_{\pi(h)}(g)$$
defines a Borel map that realizes the reduced C$^*$-norm on every $C_c(G_x,A_x)$. So we find a unique standard Borel structure on $A \rtimesred G$ such that $\pi : A \rtimesred G \to X$ and $\theta\red : C_c(G,A) \to A \rtimesred G$ are Borel and $A \rtimesred G$ is a Borel field of separable C$^*$-algebras.

To prove the unique characterization statement in the proposition, let $\rtimes$ be either the full or reduced crossed product. Since the (left and right) multiplication maps $G \times_\pi C_c(G,A) \to C_c(G,A)$ and $A \times_\pi C_c(G,A) \to C_c(G,A)$ are Borel, it follows from Proposition \ref{prop.completion} that the natural maps $A \to M(A \rtimes G)$ and $G \to M(A \rtimes G)$ are Borel.

If $A \rtimes G$ is in some other way a Borel field of separable C$^*$-algebras such that $A \to M(A \rtimes G)$ and $G \to M(A \rtimes G) : g \mapsto u_g$ are Borel, to conclude the proof of the proposition, we have to show that $\theta : C_c(G,A) \to A \rtimes G$ is Borel. It suffices to prove that for all Borel sections $F : X \to C_c(G,A)$ and $d : X \to A \rtimes G$, the map $X \to A \rtimes G : x \mapsto \theta(F_x) d_x$ is Borel. But for every $x \in X$, we have
\begin{equation}\label{eq.formula-mult}
\theta(F_x) d_x = \int_{G_x} F_x(g) \, u_g \, d_x \, d\lambda_x(g) \; .
\end{equation}
Since $G \to M(A \rtimes G) : g \mapsto u_g$ is Borel, the map $G \to A \rtimes G : g \mapsto u_g \, d_{\pi(g)}$ is Borel. Since $A \to M(A \rtimes G)$ is Borel, the multiplication map $A \times_\pi (A \rtimes G) \to A \rtimes G$ is Borel. Since $G \to A : g \mapsto F_{\pi(g)}(g)$ is Borel, it follows that $G \to A \rtimes G : g \mapsto F_{\pi(g)} \, u_g \, d_{\pi(g)}$ is Borel. Using the observation from the first paragraph of the proof and using \eqref{eq.formula-mult}, it follows that $x \mapsto \theta(F_x) d_x$ is Borel.
\end{proof}

With essentially the same argument as the one for the reduced crossed product in Proposition \ref{prop.full-and-reduced-crossed-product}, we obtain the following. We leave the proof as an exercise.

\begin{proposition}
Let $P = (P_x)_{x \in X}$ be a Borel field of von Neumann algebras with separable predual and let $G = (G_x)_{x \in X}$ be a Borel field of locally compact Polish groups. Assume that $\al : G \times_\pi P \to P$ is a Borel map such that $\pi \circ \al = \al$ and such that for every $x \in X$, the restriction of $\al$ to $G_x \times P_x$ is a continuous action $G_x \actson P_x$ by $*$-automorphisms.

There is a unique standard Borel structure on $(P \rtimes G)_* = ((P_x \rtimes G_x)_*)_{x \in X}$ such that $P \rtimes G$ becomes a Borel field of separable von Neumann algebras and such that the canonical maps $P \to P \rtimes G$ and $G \to P \rtimes G$ are Borel.
\end{proposition}

\section{Universal Borel fields}\label{sec.universal}

We say that two Borel fields $V = (V_x)_{x \in X}$ and $W = (W_y)_{y \in Y}$ of Polish spaces are \emph{isomorphic} if there exist bijective Borel maps $\psi : X \to Y$ and $\theta : V \to W$ such that $\pi \circ \theta = \psi \circ \pi$ and for every $x \in X$, the map $\theta_x : V_x \to W_{\psi(x)}$ is a homeomorphism.

Of course, one defines entirely similarly the notion of isomorphism between Borel fields of separable Banach spaces, Polish groups, separable von Neumann algebras, etc.

\begin{definition}\label{def.universal-field-Polish}
We say that a Borel field $V = (V_x)_{x \in X}$ of Polish spaces is \emph{universal} if any Borel field of Polish spaces is isomorphic with $(V_x)_{x \in X_0}$ for some Borel set $X_0 \subset X$.

We define similarly the notion of a universal Borel field of any other separable structure.
\end{definition}

At first sight, it may be surprising that such a universal Borel field of Polish spaces even exists. But even more is true: the following provides a concrete such universal Borel field. To prove this proposition, one has to show that the abstract result saying that every Polish space is homeomorphic to a closed subset of $\R^\N$ can be proven in a constructive, Borel manner.

Recall that given a Polish space $U$, we denote by $F(U)$ the set of nonempty closed subsets of $U$, equipped with the Effros Borel structure (see Section \ref{sec.Effros-Borel}). In particular, recall from Proposition \ref{prop.selection-Effros} that for every standard Borel space $Y$ and Borel map $\gamma : Y \to F(U)$, we have the Borel field of Polish spaces $V_\gamma = (\gamma(y))_{y \in Y}$, on which the standard Borel structure is given by viewing $V_\gamma$ as a Borel subset of $Y \times U$.

\begin{proposition}\label{prop.universal-field-RN}
Write $U = \R^\N$.
\begin{enumlist}
\item Every Borel field $(W_y)_{y \in Y}$ of Polish spaces is isomorphic with a Borel field of the form $V_\gamma$ for some Borel map $\gamma : Y \to F(U)$.
\item Consider the standard Borel space $X = [0,1] \times F(U)$. The tautological Borel field $V_{\gammatil}$ given by $\gammatil : X \to F(U) : \gammatil(t,P) = P$ is universal.
\end{enumlist}
\end{proposition}

The role of the interval $[0,1]$ in point 2 of Proposition \ref{prop.universal-field-RN} is only to create enough multiplicity to really get a universal field in the sense of Definition \ref{def.universal-field-Polish}.

\begin{proof}
It is a classical result that every Polish space is homeomorphic to a closed subset of $U = \R^\N$, see e.g.\ \cite[Theorem 4.17]{Kec95}. To prove the proposition, it suffices to rewrite that proof in a constructive, Borel manner. It will be more convenient to write $U = \R^\N \times \R^\N \cong \R^\N$ and embed any Polish space $P$ as a closed subset of $U$.

For every $n \in \N$, define the continuous function
\[
\kappa_n : [0,+\infty) \to [0,1] : \kappa_n(t) = \begin{cases} 1-nt &\quad\text{if $0 \leq t \leq 1/n$,}\\ 0 &\quad\text{if $t \geq 1/n$.}\end{cases}
\]
Assume that $P$ is a Polish space and that $(a_n)_{n \in \N}$ is a dense sequence in $P$. Let $d$ be a compatible complete metric on $P$. We then define the continuous injective map
\[
\theta : P \to \R^\N : \theta(a) = (d(a,a_n))_{n \in \N}
\]
and the continuous function
\[
F_n : \R^\N \to [0,1] : F_n(z) = \sum_{r = n}^\infty 2^{-r} \, \kappa_n\bigl(\max_{k \in \{1,\ldots,r\}} |z_k - d(a_r,a_k)|\bigr) \; .
\]
We claim that $z \in \R^\N$ satisfies $F_n(z) > 0$ for all $n \in \N$ if and only if $z \in \theta(P)$.

To prove this claim, first assume that $z = \theta(a)$ for $a \in P$ and take $n \in \N$. Choose $r \in \N$ such that $d(a,a_r) < 1/n$. Then
\[
|\theta(a)_k - d(a_r,a_k)| \leq d(a,a_r) < 1/n \quad\text{for all $k \in \N$,}
\]
so that $F_n(\theta(a)) > 0$.

Conversely, assume that $z \in \R^\N$ and $F_n(z) > 0$ for all $n \in \N$. So for every $n \in \N$, we find an $r_n \geq n$ such that
\[
\kappa_n\bigl(\max_{k \in \{1,\ldots,r_n\}} |z_k - d(a_{r_n},a_k)|\bigr) > 0 \; .
\]
This means that $|z_k - d(a_{r_n},a_k)| < 1/n$ for all $k \in \{1,\ldots,r_n\}$. For every $n,m \in \N$, we have $r_n \leq r_m$ or $r_m \leq r_n$. In the first case, we get that
\[
|z_{r_n} - d(a_{r_m},a_{r_n})| < 1/m \quad\text{and}\quad |z_{r_n}| = |z_{r_n} - d(a_{r_n},a_{r_n})| < 1/n
\]
so that $d(a_{r_m},a_{r_n}) < 1/m + 1/n$. By symmetry, the same inequality holds when $r_m \leq r_n$. It follows that $(a_{r_n})_{n \in \N}$ is a Cauchy sequence in $P$. Take $a \in P$ such that $a_{r_n} \to a$. As explained above, $|z_k - d(a_{r_n},a_k)| < 1/n$ for all $k \leq r_n$. Since $r_n \geq n$, we get in particular that $|z_k - d(a_{r_n},a_k)| < 1/n$ for all $k \leq n$. Taking $n \to \infty$, it follows that $z_k = d(a,a_k)$ for all $k \in \N$. This means that $z = \theta(a)$.

Because of the claim proven above, we have a well defined continuous function
\[
\zeta : P \to \R^\N \times \R^\N : \zeta(a) = \bigl(\theta(a),(F_n(\theta(a))^{-1})_{n \in \N}\bigr) \; .
\]
We next claim that $\zeta(P) \subset \R^\N \times \R^\N$ is closed and that $\zeta$ is a homeomorphism from $P$ onto $\zeta(P)$.

Let $(b_i)_{i \in \N}$ be a sequence such that $\zeta(b_i) \to (z,f) \in \R^\N \times \R^\N$. It suffices to prove that there exists a $b \in P$ such that $b_i \to b$ and $\zeta(b) = (z,f)$.

Fix $n \in \N$. Since $F_n(\theta(b_i))^{-1} \to f_n$, there exists a $\delta > 0$ and $i_0 \in \N$ such that $F_n(\theta(b_i)) \geq \delta$ for all $i \geq i_0$. Since $\theta(b_i) \to z$ and $F_n$ is continuous, it follows that $F_n(z) \geq \delta > 0$. So $F_n(z) > 0$ for all $n \in \N$. By the claim proven above, it follows that $z = \theta(b)$ for a unique $b \in P$.

Choose $\eps > 0$. Take $n \in \N$ such that $d(b,b_n) < \eps/3$. Since $\lim_i \theta(b_i)_n = z_n = \theta(b)_n$, we can take $i_0 \in \N$ such that $|d(b_i,b_n) - d(b,b_n)| < \eps/3$ for all $i \geq i_0$. Since $d(b,b_n) < \eps/3$, it follows that $d(b_i,b) < \eps$ for all $i \geq i_0$. So we have proven that $b_i \to b$. Since $\zeta$ is continuous, the claim is proven.

Now assume that $(W_y)_{y \in Y}$ is any Borel field of Polish spaces. Take a compatible Borel metric $d : W \times_\pi W \to [0,+\infty)$ and take a dense sequence of Borel sections $\vphi_k : Y \to W$. Define the Borel maps
\begin{align*}
& \theta : W \to \R^\N : \theta(w) = d(w,\vphi_n(\pi(w)))_{n \in \N} \quad\text{and}\\
& G_n : W \to (0,1] : G_n(w) = \sum_{r = n}^\infty 2^{-r} \, \kappa_n\bigl(\max_{k \in \{1,\ldots,r\}} |\theta(w)_k - d(\vphi_r(\pi(w)),\vphi_k(\pi(w)))|\bigr) \; .
\end{align*}
By the discussion above, for every $y \in Y$, the map
\[
\zeta_y : W_y \to \R^\N \times \R^\N : \zeta_y(w) = \bigl(\theta(w),(G_n(w)^{-1})_{n \in \N} \bigr)
\]
has closed range and is a homeomorphism of $W_y$ onto $\zeta_y(W_y)$. Write $\gamma(y) = \zeta_y(W_y)$. By construction, $\gamma(y)$ equals the closure of $\{\zeta_y(\vphi_k(y)) \mid k \in \N\}$. So, $\gamma : Y \to F(U)$ is Borel by Proposition \ref{prop.selection-Effros}. By construction, $V$ is isomorphic with the Borel field $V_\gamma$.

To prove the final statement, define $X = [0,1] \times F(U)$ and consider the Borel field $V_{\gammatil}$ where $\gammatil : X \to F(U) : \gammatil(t,P) = P$. Choose any Borel field $(W_y)_{y \in Y}$ of Polish spaces. By the previous point, we can take an injective Borel map $\zeta : W \to Y \times U : w \mapsto (\pi(w),\zeta_{\pi(w)}(w))$ that provides an isomorphism between $W$ and the field $V_\gamma$ where $\gamma : Y \to F(U)$ is the Borel map $\gamma(y) = \zeta_y(W_y)$. Choose an injective Borel map $\rho : Y \to [0,1]$. Then,
\[
\psi : Y \to [0,1] \times F(U) : \psi(y) = (\rho(y),\gamma(y))
\]
is an injective Borel map. Define $X_0 = \psi(Y)$. By construction, the map
\[
\Theta : W \to [0,1] \times F(U) \times U : w \mapsto (\rho(\pi(w)),\gamma(y),\zeta_{\pi(w)}(w))
\]
together with $\psi$ realizes an isomorphism between $(W_y)_{y \in Y}$ and $V_{\gammatil}|_{X_0}$. So, $V_{\gammatil}$ is universal.
\end{proof}

For Borel fields of separable Banach spaces, we have a similar result. It is based on the classical result that every separable Banach space can be isometrically embedded as a closed subspace of $C(\Delta)$, where $\Delta$ is the Cantor set. Before stating and proving this result, we show that also the classical Alexandroff-Hausdorff theorem can be easily Borel coded.

\begin{lemma}\label{lem.Alexandroff-Hausdorff}
Let $K = (K_x)_{x \in X}$ be a Borel field of Polish spaces. Assume that every $K_x$ is compact. Let $\Delta$ be the Cantor set.

There exists a Borel map $\theta : X \times \Delta \to K$ such that for every $x \in X$, the map $\theta_x$ defined by $\theta_x(b) = \theta(x,b)$ is a continuous surjective map of $\Delta$ onto $K_x$.
\end{lemma}
\begin{proof}
We follow the approach of \cite{Ros76}. We realize $\Delta \subset [0,1]$ as the middle-two-thirds set, i.e.\ the set of all $x \in [0,1]$ that can be written as $x = \sum_{n=1}^\infty a_n \, 6^{-n}$ with $a_n \in \{0,5\}$.

Let $\vphi_n : X \to K$ be a dense sequence of Borel sections and choose a compatible Borel metric $d : K \times_\pi K \to [0,+\infty)$. Replacing $d$ by $\min\{d,1\}$, we may assume that $d \leq 1$. For every $x \in X$, define the continuous injective map
\[
\rho_x : K_x \to [0,1]^\N : \rho_x(k) = (d(k,\vphi_n(x)))_{n \in \N} \; .
\]
Then define the injective Borel map $\rho : K \to X \times [0,1]^\N : \rho(k) = (\pi(k),\rho_{\pi(k)}(k))$. So, $\rho(K) \subset X \times [0,1]^\N$ is a Borel set. Fix a continuous surjection $\zeta : \Delta \to [0,1]^\N$. Define the Borel set $L \subset X \times \Delta$ by $L = (\id \times \zeta)^{-1}(\rho(K))$. By construction, for every $x \in X$, the section $L_x \subset \Delta$ equals the closed subset $\zeta^{-1}(\rho_x(K_x))$. Since all $L_x$ are compact, it follows from Theorem \ref{thm.selection-sigma-compact} that $L = (L_x)_{x \in X}$ is a Borel field of closed subsets of $\Delta$.

By Proposition \ref{prop.selection-Effros}, the map $X \times \Delta \to [0,1] : (x,b) \mapsto d(b,L_x)$ is Borel. So,
$$Z = \{(x,b,a) \in X \times \Delta \times \Delta \mid (x,a) \in L , |b-a| = d(b,L_x)\}$$
is Borel. As in \cite{Ros76}, when $b,b' \in \Delta$, we have that $(b+b')/2 \not\in \Delta$. So for every $(x,b) \in X \times \Delta$, there is a unique element $a \in L_x$ such that $|b-a| = d(b,L_x)$. We write $a = \gamma(x,b)$. By construction, the Borel set $Z$ is the graph of $\gamma$, so that $\gamma$ is a Borel map. By construction, the Borel map
$$\theta : X \times \Delta \to K : \theta(x,b) = \rho^{-1}(x,\zeta(\gamma(x,b)))$$
satisfies the conclusions of the lemma.
\end{proof}

Given a Banach space $V$, we denote by $S(V)$ the set of closed subspaces of $V$ with the Effros Borel structure. As above, we denote by $\Delta$ the Cantor set.

\begin{proposition}\phantomsection\label{prop.universal-field-Banach}
\begin{enumlist}
\item Every Borel field $(W_y)_{y \in Y}$ of separable Banach spaces is isomorphic with a Borel field of the form $V_\gamma$ for some Borel map $\gamma : Y \to S(C(\Delta))$.
\item Consider the standard Borel space $X = [0,1] \times S(C(\Delta))$. The tautological Borel field $V_{\gammatil}$ given by $\gammatil : X \to S(C(\Delta)) : \gammatil(t,A) = A$ is universal.
\end{enumlist}
\end{proposition}
\begin{proof}
Let $(W_y)_{y \in Y}$ be a Borel field of separable Banach spaces. By Proposition \ref{prop.dual-field-Banach}, we get the Borel field $(K_y)_{y \in Y}$ of Polish spaces $K_y = \ball W_y^*$, equipped with the weak$^*$ topology. By Proposition \ref{prop.Borel-field-CK}, we have the corresponding Borel field $C(K) = C(K_y)_{y \in Y}$ of the separable Banach spaces of continuous functions on $K_y$. By duality, every $a \in W$ defines a continuous function $\eta_a$ on $K_{\pi(a)}$ given by $\om \mapsto \om(a)$. In this way, $\eta : W \to C(K)$ is an injective Borel map that identifies $W$ with a Borel field of closed subspaces of $C(K_y)_{y \in Y}$.

By Lemma \ref{lem.Alexandroff-Hausdorff}, we can choose a Borel map $\theta : Y \times \Delta \to K$ such that every $\theta_y$ is a continuous surjective map of $\Delta$ onto $K_y$. Then the map $W \to C(\Delta) : a \mapsto \eta_a \circ \theta_{\pi(a)}$ identifies $W$ with a Borel field of closed subspaces of $C(\Delta)$. The second point of the proposition follows from the first in the same way as in the proof of Proposition \ref{prop.universal-field-RN}.
\end{proof}

\begin{remark}
In the proof of Lemma \ref{lem.Alexandroff-Hausdorff}, we implicitly showed the following observation. Assume that $L$ and $K$ are compact Polish spaces and $\zeta : L \to K$ is a continuous surjective map, then the corresponding map $F(K) \to F(L) : A \mapsto \zeta^{-1}(A)$ is Borel. For this property to hold, compactness is essential and was used in our proof of Lemma \ref{lem.Alexandroff-Hausdorff} when we invoked Theorem \ref{thm.selection-sigma-compact}.

Indeed, by \cite[Exercise 27.7]{Kec95}, we can choose a Polish space $K$ and a closed subset $C \subset K$ such that $\{A \in F(K) \mid A \cap C \neq \emptyset\}$ is not Borel. Define the Polish space $L = C \sqcup K$ and define $\zeta : L \to K$ by the embedding $C \hookrightarrow K$ on $C$ and the identity on $K$. Since
\[
\{A \in F(K) \mid A \cap C = \emptyset\} = \{A \in F(K) \mid \zeta^{-1}(A) = \emptyset \sqcup A \} \; ,
\]
it follows that the map $F(K) \to F(L) : A \mapsto \zeta^{-1}(A)$ is not Borel.
\end{remark}

The following universality result is an immediate consequence of Proposition \ref{prop.equivalence-defs-Borel-field-vNalg}. We leave the proof as an exercise.

Given a separable Hilbert space $H$, we denote by $\vNalg(H)$ the set of von Neumann algebras acting on $M$. Equipped with a version of the Effros Borel structure, see \cite{Eff64}, $\vNalg(H)$ is a standard Borel space.

\begin{proposition}\phantomsection\label{prop.universal-field-vNalg}
\begin{enumlist}
\item Every abstract Borel field $(M_y)_{y \in Y}$ of separable von Neumann algebras is isomorphic with a Borel field of the form $M_\gamma$ for some Borel map $\gamma : Y \to \vNalg(\ell^2(\N))$.
\item Consider the standard Borel space $X = [0,1] \times \vNalg(\ell^2(\N))$. The tautological Borel field $M_{\gammatil}$ given by $\gammatil : X \to \vNalg(\ell^2(\N)) : \gammatil(t,M) = M$ is universal.
\end{enumlist}
\end{proposition}

In \cite{Kec96}, the following Borel parametrization of separable C$^*$-algebras is defined. Let $H$ be a separable infinite dimensional Hilbert space. Consider the standard Borel structure on $B(H)$ generated by the weakly open subsets. Then define the standard Borel space $Y = B(H)^\N$ and, for every $y \in Y$, denote by $A_y \subset B(H)$ the C$^*$-algebra generated by the operators $(y_n)_{n \in \N}$. The following result shows that our definition of a Borel field of separable C$^*$-algebra (see Definition \ref{def.Borel-field-Cstar}) is compatible with the approach of \cite{Kec96}.

\begin{proposition}\phantomsection\label{prop.universal-field-Cstar}
\begin{enumlist}
\item There is a unique standard Borel structure on $A = (A_y)_{y \in Y}$ such that $A$ is a Borel field of separable C$^*$-algebras and the sections $Y \to A : y \mapsto y_n \in A_y$ are Borel for all $n \in \N$.
\item With $X = [0,1] \times Y$, the Borel field $\Atil = (A_y)_{(s,y) \in X}$ of separable C$^*$-algebras is universal.
\end{enumlist}
\end{proposition}
\begin{proof}
1.\ Viewing $A \subset Y \times B(H)$, we claim that $A$ is a Borel set. Denote by $\cP$ the $*$-algebra of noncommutative polynomials in the variables $Y_n$ and $Y_n^*$ with coefficients in $\Q + i \Q$ and without constant term. For every $y \in Y$, denote by $\cP \to B(H) : P \mapsto P(y)$ the unique $(\Q + i \Q)$-$*$-algebra morphism satisfying $Y_n(y) = y_n$. Note that $\cP$ is a countable set and that for all $y \in Y$, $A_y$ is the closure of $\{P(y) \mid P \in \cP\}$. It follows that an element $(y,T) \in Y \times B(H)$ belongs to $A$ if and only if
$$\forall n \geq 1, \exists P \in \cP : \|T - P(y)\| < 1/n \; .$$
Since for every $P \in \cP$, the map $Y \to B(H) : y \mapsto P(y)$ is Borel, the claim is proven.

We define a standard Borel structure on $A$ as a Borel subset of $Y \times B(H)$. By construction, $A$ is a Borel field of separable C$^*$-algebras. The uniqueness follows from Lemma \ref{lem.uniqueness}.

2.\ Let $B = (B_z)_{z \in Z}$ be any Borel field of separable C$^*$-algebras. Choose a dense sequence of Borel sections $\vphi_n : Z \to B$. By Lemma \ref{lem.GNS-field-Cstar}, we can choose a Borel map $\theta : B \to B(H)$ such that for every $z \in Z$, the restriction of $\theta$ to $B_z$ is a faithful $*$-homomorphism. Choose an injective Borel map $\rho : Z \to [0,1]$. Define the Borel map
$$\gamma : Z \to Y : \gamma(z) = (\theta(\vphi_n(z)))_{n \in \N} \; .$$
Consider the injective Borel map $\psi : Z \to X = [0,1] \times Y : \psi(z) = (\rho(z),\gamma(z))$. Still viewing $A \subset Y \times B(H)$ and identifying $\Atil = [0,1] \times A$, we define the Borel map
$$\Theta : B \to \Atil : \Theta(b) = (\rho(\pi(b)),\gamma(\pi(b)),\theta(b)) \; .$$
Define the Borel set $X_0 = \psi(Z)$. By construction, $\psi$ and $\Theta$ implement an isomorphism between $B$ and $(\Atil_x)_{x \in X_0}$, so that the universality of $\Atil$ follows.
\end{proof}

Our final universality result concerns the more subtle case of Borel fields of Polish groups, as introduced in \cite{Sut85}. By a theorem of Uspenski\v{\i} (see \cite[Theorem 9.18]{Kec95}), every Polish group is isomorphic to a closed subgroup of the Polish group $\Homeo( [0,1]^\N)$ of homeomorphisms of the Hilbert cube $[0,1]^\N$. We believe that this proof can also be coded entirely in a Borel way, but we did not check all the details. Instead, we give a more abstract proof for the existence of a universal Borel field of Polish groups. Note that a similar abstract proof can be given to prove the existence of universal Borel fields of Polish spaces or separable Banach spaces.

The proof of the following proposition can be seen as a reformulation in our context of Sutherland's approach in \cite{Sut85} to define the standard Borel space of Polish groups as a space of pseudometrics on the free group $\F_\infty$.

To make this work, we need Proposition \ref{prop.left-invariant-metric} and the following basic facts on topological groups. If $G$ is a metrizable topological group and $D$ is a compatible left invariant metric, which exists by the Birkhoff-Kakutani theorem (see e.g.\ \cite[Theorem 9.1]{Kec95}), then also the metric $d(g,h) = D(g,h) + D(g^{-1},h^{-1})$ induces the topology of $G$. The multiplication and inverse map on $G$ uniquely extend to continuous maps on the completion $G_d$ of the metric space $(G,d)$ and then $G_d$ is a topological group (see e.g.\ \cite[Part 3, Theorem 2.3.5]{THJ80}). If the initial $G$ was a Polish group, by \cite[Proposition 1.2.1]{BK96}, $G \subset G_d$ must be a closed subgroup, which means that $G = G_d$. We thus conclude that for any compatible left invariant metric $D$ on a Polish group $G$, the metric $d$ is complete.

\begin{proposition}\label{prop.universal-field-Polish-groups}
There exists a universal Borel field of Polish groups.
\end{proposition}
\begin{proof}
Denote by $\F_\infty$ the free group on a countably infinite number of generators. Consider the Borel subset $\cE \subset \R^{\F_\infty}$ of functions $E : \F_\infty \to \R$ satisfying the following conditions.
\begin{enumlist}
\item $E(e) = 0$, $E(g) \geq 0$, $E(g) = E(g^{-1})$ and $E(gh) \leq E(g) + E(h)$ for all $g,h \in \F_\infty$.
\item $\forall g \in \F_\infty , \forall n \in \N , \exists k \in \N :$ if $h \in \F_\infty$ and $E(h) < 1/k$, then $E(ghg^{-1}) < 1/n$.
\end{enumlist}
For every $E \in \cE$, the set of $g \in \F_\infty$ with $E(g) = 0$ forms a normal subgroup $N_E \lhd \F_\infty$ and the formula $D(g,h) = E(h^{-1}g)$ yields a well-defined left invariant metric on $\Gamma_E := \F_\infty / N_E$ that turns $\Gamma_E$ into a topological group. As explained above, completing $\Gamma_E$ w.r.t.\ the metric $d(g,h) = D(g,h) + D(g^{-1},h^{-1})$ gives a Polish group $G_E$.

We thus consider the standard Borel space $\cE \times \F_\infty$ with the Borel maps $\pi : \cE \times \F_\infty \to \cE : \pi(E,g) = E$ and
$$d : (\cE \times \F_\infty) \times_\pi (\cE \times \F_\infty) \to [0,+\infty) : d((E,g),(E,h)) = E(h^{-1}g) + E(hg^{-1}) \; .$$
We obtain a Borel field of pseudometric spaces in the sense of Definition \ref{def.Borel-field-metric-spaces}. Proposition \ref{prop.completion} gives us by separation-completion a Borel field $G = (G_E)_{E \in \cE}$ of Polish spaces, with canonical Borel map $\gamma : \cE \times \F_\infty \to G$. We write $\gamma_E : \F_\infty \to G_E$ for every $E \in \cE$. By Proposition \ref{prop.completion} there is a unique Borel multiplication map $G \times_\pi G \to G$ and inverse map $G \to G$ such that $G$ becomes a Borel field of Polish groups and all $\gamma_E$ are group homomorphisms.

Write $X = [0,1] \times P$ and define $G'_{(t,E)} = G_E$ for all $(t,E) \in X$. We prove that $G'$ is a universal Borel field of Polish groups.

Let $(H_y)_{y \in Y}$ be any Borel field of Polish groups. By Proposition \ref{prop.left-invariant-metric}, we can choose a Borel map $\delta : H \times_\pi H \to [0,+\infty)$ such that for every $y \in Y$, $\delta_y$ is a left invariant metric on $H_y$ that is compatible with the topology. Define $\delta'(g,h) = \delta(g,h) + \delta(g^{-1},h^{-1})$. By the explanation above, for every $y \in Y$, $\delta'_y$ is a compatible complete metric on $H_y$. Choose a dense sequence of Borel sections $\vphi_n : Y \to H$.

For every $y \in Y$, denote by $\theta_y : \F_\infty \to H_y$ the unique group homomorphism that maps the $n$'th generator of $\F_\infty$ to $\vphi_n(y) \in H_y$. Then, $\theta : Y \times \F_\infty \to H : \theta(y,g) = \theta_y(g)$ is a Borel map. For every $y \in Y$, define $E_y \in \cE$ by $E_y(g) = \delta(\theta_y(g),e_y)$. Note that $Y \to \cE : y \mapsto E_y$ is Borel. Since for every $y \in Y$, the metric $\delta'_y$ on $H_y$ is complete, there is a unique isomorphism of Polish groups $\zeta_y : H_y \to G_{E_y}$ satisfying $\zeta_y(\theta_y(g)) = \gamma_{E_y}(g)$ for all $g \in \F_\infty$ and $y \in Y$. By the second point of Proposition \ref{prop.completion}, the resulting map $\zeta : H \to G$ is Borel.

Note that $G' = [0,1] \times G$. Choose an injective Borel map $\rho : Y \to [0,1]$. Define the injective Borel map
\[
\Theta : H \to G' : \Theta(h) = (\rho(\pi(h)),\zeta_{\pi(h)}(h)) \; .
\]
Also consider the injective Borel map $\psi : Y \to X : \psi(y) = (\rho(y),E_y)$. Writing $X_0 = \psi(Y)$, we get that $\psi$ and $\Theta$ implement an isomorphism between the fields of Polish groups $H = (H_y)_{y \in Y}$ and $(G'_{(t,d)})_{(t,d) \in X_0}$. So, $G'$ is universal.
\end{proof}

\begin{remark}\label{rem.about-Borel-properties}
Consider any of the separable structures that we considered in this paper, such as Polish spaces, separable Banach spaces, von Neumann algebras with separable predual, etc. Assume now that $\cR$ is an $n$-ary relation between such structures. For example, for $n=1$, $\cR$ is a property such as ``being a reflexive Banach space'' or ``being a type III factor''. For $n=2$, $\cR$ is a relation, with the isomorphism relation being the prime example, but another example would be ``are measure equivalent'', when talking about countable groups.

From our definition of a universal Borel field, it follows tautologically that the following two statements are equivalent, in the context of any of the separable structures considered in this paper.
\begin{itemlist}
\item For every Borel field $(S_x)_{x \in X}$, the set $\{x \in X^n \mid (S_{x_1},\ldots,S_{x_n}) \;\text{satisfies}\; \cR\}$ is Borel, resp.\ analytic.
\item For a universal Borel field $(S_x)_{x \in X}$, the set $\{x \in X^n \mid (S_{x_1},\ldots,S_{x_n}) \;\text{satisfies}\; \cR\}$ is Borel, resp.\ analytic.
\end{itemlist}

There is a lot of literature on properties and relations between separable structures being Borel or analytic, see e.g.\ \cite{FLR06}. In particular, it has been proven in many cases that the isomorphism relation is complete analytic and thus, not Borel. Such statements are usually proven in a specific Borel coding of the specific separable structure that is being studied. Such a Borel coding will, by definition, be a Borel field in the sense of our paper. In \cite{FLR06}, it is then said that \emph{``It is an empirical fact that any other way of defining this leads to equivalent results.''} The precise mathematical formulation of this empirical fact is that the Borel field that is used as a Borel coding, actually is universal, in the sense of Definition \ref{def.universal-field-Polish}.
\end{remark}

\end{document}